\documentclass[reqno]{amsart}

\usepackage{xfrac}

\usepackage{hyperref,todonotes,comment,multicol}
\usepackage[color,matrix,arrow]{xy}
\usepackage{graphics,graphicx,amssymb}
\usepackage{pagecolor,lipsum}
\usepackage{latexsym}
\usepackage{mathrsfs}
\xyoption{curve}
\usepackage{hyperref}
\hypersetup{ 
colorlinks,
citecolor=blue,
filecolor=blue,
linkcolor=blue,
urlcolor=blue
}

\newtheorem{lemma}{Lemma}[section]
\newtheorem{proposition}[lemma]{Proposition}
\newtheorem{theorem}[lemma]{Theorem}
\newtheorem{corollary}[lemma]{Corollary}
\newtheorem{question}[lemma]{Question}

\newtheorem*{theoremA}{Theorem}

\theoremstyle{definition}
\newtheorem{example}[lemma]{Example}
\newtheorem{definition}[lemma]{Definition}
\newtheorem{remark}[lemma]{Remark}

\newtheorem{hypotheses}[lemma]{Hypotheses}

\newcommand{\mfk}[1]{\mathfrak{#1}}

\newcommand{\msf}[1]{\mathsf{#1}}
\newcommand{\msc}[1]{\mathscr{#1}}
\newcommand{\mbf}[1]{\mathbf{#1}}
\newcommand{\opn}[1]{\operatorname{#1}}
\newcommand{\ot}{\otimes}

\DeclareMathOperator{\Hom}{Hom}
\DeclareMathOperator{\End}{End}

\DeclareMathOperator{\rep}{rep}
\DeclareMathOperator{\Rep}{Rep}

\DeclareMathOperator{\Fun}{Fun}
\DeclareMathOperator{\res}{res}
\DeclareMathOperator{\SL}{SL}
\DeclareMathOperator{\QCoh}{QCoh}

\newcommand{\dG}{\check{G}}
\newcommand{\uFun}{\underline{\operatorname{Fun}}}

\renewcommand{\1}{\mathbf{1}}
\renewcommand{\O}{\mathscr{O}}

\renewcommand{\hat}{\widehat}
\renewcommand{\binom}[2]{{\Small\left[\begin{matrix}\ #1\ \\ #2 \end{matrix} \right]}}

\newcommand{\nocontentsline}[3]{}
\newcommand{\tocless}[2]{\bgroup\let\addcontentsline=\nocontentsline#1{#2}\egroup}

\definecolor{page_color}{HTML}{000000}
\definecolor{text_color}{HTML}{F0EAD6}

\makeindex

\title[]{Quantum Frobenius and modularity for quantum groups at arbitrary roots of $1$}
\date{November 22, 2023}

\author{Cris Negron}
\address{Department of Mathematics, University of Southern California, Los Angeles, CA 90007}
\email{cnegron@usc.edu}

\begin{document}

\maketitle

\begin{abstract}
We consider quantum group representations $\Rep(G_q)$ for a semisimple algebraic group $G$ at a complex root of unity $q$.  Here we allow $q$ to be of any order.  We first show that the Tannakian center in $\Rep(G_q)$ is calculated via a twisting of Lusztig's quantum Frobenius functor $\Rep(\dG)\to \Rep(G_q)$, where $\dG$ is a dual group to $G$.  We then consider the associated fiber category $\opn{Vect}\ot_{\Rep(\dG)}\Rep(G_q)$ over $B\dG$, and show that this fiber is a finite, integral braided tensor category.  Furthermore, when $G$ is simply-connected and $q$ is of even order, the fiber in question is shown to be a modular tensor category.  Finally, we exhibit a finite-dimensional quasitriangular quasi-Hopf algebra (aka, small quantum group) whose representations recover the tensor category $\opn{Vect}\ot_{\Rep(\dG)}\Rep(G_q)$, and we describe the representation theory of this algebra in detail.  At particular pairings of $G$ and $q$, our quasi-Hopf algebra is identified with Lusztig's original finite-dimensional Hopf algebra from the 90's.  This work completes the author's project from \cite{negron21}.
\end{abstract}

\tocless\section{Introduction}

Let $k$ be an algebraically closed field of characteristic $0$, $G$ be a semisimple algebraic group over $k$, and $q$ be an arbitrary root of unity in $k$.  We consider the category $\Rep(G_q)$ of quantum group representations for $G$ at $q$.  Modulo some toral details, this is just the category of integrable representations for Lusztig's divided power quantum enveloping algebra $U_q$ \cite{lusztig90,lusztig90II,lusztig93}.  Here we take $\Rep(G_q)$ along with its natural braided tensor structure \cite{drinfeld89,lusztig93}.
\par

As one sees from various works over the past decades, the braided tensor category $\Rep(G_q)$ has proved fantastically useful in studies of both mathematics and mathematical physics.  In the 90's, quantum group representations appeared in analyses of low dimensional TQFTs and $2$-dimensional rational conformal field theories.  In such settings one traditionally encountered the category of so-called \emph{semisimplified} quantum group representations.  However in recent studies of logarithmic CFTs and ``derived" TQFTs the category $\Rep(G_q)$ itself, or the associated category of small quantum group representations, plays a key role. See for example \cite{feigintipunin,gainutdinovetal06,gukovetal21,gannonnegron,creutzigdimoftegarnergeer,gaiottomooreneitzkeyan}, and related works.
\par

Now, in considering these physical studies one observes a certain dissonance between what we ``understand" about quantum groups from a purely algebraic perspective, and the kinds of phenomena which we expect from--and which are necessitated by--the physical perspective.  While we avoid an in depth elaboration of the physics literature, we record some of the main mathematical takeaways below.  Let us proceed:

\begin{enumerate}
\item[({\it i})] At even order $q$ there is some difference between purely algebraic constructions of the small quantum group, as a Hopf subalgebra in the divided power algebra $U_q$, and categorical constructions of the small quantum group as the ``kernel" of a quantum Frobenius functor.  This point was made very clear in type $A_1$ through works of Gainutdinov-Semikhatov-Tipunin-Feigin \cite{gainutdinovetal06} and Kondo-Saito \cite{kondosaito11}, though the situation in higher ranks is even more striking.  (See also \cite{gannonnegron,creutziglentnerrupert2,sugimoto21}.)

\item[({\it ii})] These differences in the constructions of the small quantum group are equivalently expressed as differences in the behaviors of restriction, from $\Rep(G_q)$ to the category of small quantum group representations.

\item[({\it iii})] In returning to the odd order case, one finds similar inconsistent behaviors for quantum group representations which have, somehow, not been accounted for in the literature up to this point.
\end{enumerate}

To highlight one specific issue, the Hopf subalgebra approach produces a tensor category of representations which is non-braidable in general, while physically we expect the category of small quantum group representations to be a \emph{modular} tensor category whose braiding is compatible with the braiding on $\Rep(G_q)$, in some way.  One also expects physically that small quantum group representations deform naturally over local systems \cite{gaiotto19,creutzigdimoftegarnergeer,feiginlentner}, and that such deformations provide a centrally corrected version of the category of De Concini-Kac modules, although addressing this latter point is outside of the domain of the present work.
\par

Prompted by the aforementioned physical data, and related points (\emph{i})--(\emph{iii}), we provide in this text an analysis of the category of (big) quantum group representations, and a \emph{construction} of a category of small quantum group representations, which is completely uniform across all pairings of $G$ and $q$, is transparent and algebraic in nature, and which fits the current demands from conformal and topological field theory.  Most notably, in all physically relevant settings, our category of small quantum group representations is naturally modular.\footnote{For comparisons with the authors' previous work \cite{negron21}, and that of Gainutdinov-Lentner-Ohrmann \cite{gainutdinovlentnerohrmann}, see Sections \ref{sect:related}, \ref{sect:comp}, \ref{sect:glo}.}

\subsection{A uniform model for small quantum group representations} 

In this text we study the following model for the small quantum group: We define the small quantum group, somewhat ambiguously, as the algebraic object which represents the reduction of $\Rep(G_q)$ along its Tannakian center,
\begin{equation}\label{eq:OG_fiber}
\Rep(\text{small quantum group}):=\text{The categorical fiber }\opn{Vect}\ot_{\opn{Tann}_q}\Rep(G_q).
\end{equation}
Here
\[
\opn{Tann}_q:=\ \text{The Tannakina center in}\ \Rep(G_q)
\]
is the maximal (M\"uger) central, Tannakian subcategory in $\Rep(G_q)$, and the subsequent fiber is obtained as the pushout for the diagram
\[
\xymatrix{
\opn{Tann}_q\ar[rr]^{\opn{incl}}\ar[d]_{\begin{array}{c}\text{\SMALL (unique) symm}\\ \text{\SMALL fiber fun}\end{array}} & & \Rep(G_q)\ar@{-->}[d]\\
\opn{Vect}\ar@{-->}[rr] & & \opn{Vect}\ot_{\opn{Tann}_q}\Rep(G_q)
}
\]
(see Proposition \ref{prop:bc}).  We also recall that a Tannakian category is, by definition if one likes, a symmetric tensor category $\msc{E}$ which admits a symmetric equivalence $\msc{E}\cong\Rep(H)$ for some algebraic group $H$.  To explain in words, the category of small quantum group representations is the maximally non-degenerate tensor quotient of the category $\Rep(G_q)$, and it is uniquely determined by this property.
\par

An interpretation \emph{like} \eqref{eq:OG_fiber} was first suggested in work of Arkhipov and Gaitsgory \cite{arkhipovgaitsgory03}, and seems to be a preferred expression for the category of small quantum group representations from the perspective of geometric Langlands.  See for example Gaitsgory's text \cite[\S\ 6.4.2]{gaitsgory21} where a small quantum category is defined as the fiber of $\Rep(G_q)$ over its M\"uger (rather than Tannakian) center.\footnote{Fibering over the M\"uger vs.\ Tannakian centers are different operations in general, and the resulting categories will be algebras over different operads.  This point is implicit in the discussions of Section \ref{sect:comp} below.}  So our text is, with some intent, in conversation with geometric Langlands in this regard.
\par

In addressing the expression \eqref{eq:OG_fiber}, there are two things which one needs to understand.  First, one needs to understand the Tannakian center in $\Rep(G_q)$.  Second, one needs to understand how the reduction operation $\opn{Vect}\ot_{\opn{Tann}_q}-$ behaves when it is applied to $\Rep(G_q)$.
\par

In this introduction we first describe the basic structure of the fiber category \eqref{eq:OG_fiber}.  We then return to a discussion of the Tannakian center below. In Section \ref{sect:tann_intro} we explain, in particular, how the Tannakian center in $\Rep(G_q)$ is calculated via a toral twisting of Lusztig's quantum Frobenius functor.
\par

Under absolutely no restrictions on $G$ and $q$, we prove the following.

\begin{theoremA}[\ref{thm:1}]
Let $G$ be an arbitrary semisimple algebraic group and $q$ be any root of unity.  The fiber category $\opn{Vect}\ot_{\opn{Tann}_q}\Rep(G_q)$ is a finite integral braided tensor category of Frobenius-Perron dimension
\begin{equation}\label{eq:161}
\opn{FPdim}\big(\opn{Vect}\ot_{\opn{Tann}_q}\Rep(G_q)\big)=\#(G,q)\cdot (\prod_{\gamma\in \Phi^+}l_\gamma)^2,
\end{equation}
where $\#(G,q)$ is a positive integer which is explicitly calculable in examples.  This fiber is furthermore identified with the category of representations for a finite-dimensional quasitriangular quasi-Hopf algebra $u_q$ which has vector space dimension \eqref{eq:161}.
\end{theoremA}

At advantageous odd order parameters, the algebra $u_q$ is just Lusztig's original finite-dimensional Hopf algebra from \cite{lusztig90,lusztig90II}, plus-or-minus some grouplikes, and at advantageous even order parameters the algebra $u_q$ recovers Arkhipov and Gaitsgory's finite-dimensional algebra from \cite{arkhipovgaitsgory03}, again plus-or-minus some grouplikes.  In the general setting $u_q$ is obtained as a non-unique renormalization of Lusztig's finite-dimensional algebra.  A precise construction of $u_q$, and a detailed description of its representation theory, can be found in Sections \ref{sect:smallquantumgroup} and \ref{sect:smallrep}.

\begin{remark}
Comparisons with earlier quasi-Hopf constructions from Gainutdinov, Lentner, and Ohrmann are provided in Section \ref{sect:glo}.
\end{remark}
\par

\begin{remark}
In the expression \eqref{eq:161} the integers $l_\gamma$ are the orders of specific roots of unity $q^{(\gamma,\gamma)}$ associated to each $\gamma\in \Phi^+$.
\end{remark}

In the simply-connected setting we can be even more precise in our description of the fiber \eqref{eq:OG_fiber}.  In the statement below by a \emph{modular} tensor category we mean a braided tensor category which is ribbon and has trivial M\"uger center.

\begin{theoremA}[\ref{thm:2}]
Consider a simply-connected semisimple algebraic group $G$ and a root of unity $q$.  Suppose that $2r$ divides the order of $q$, where $r$ is the lacing number for $G$.  Then the fiber category $\opn{Vect}\ot_{\opn{Tann}_q}\Rep(G_q)$ is a finite, integral, modular tensor category of Frobenius-Perron dimension
\[
\opn{FPdim}\big(\opn{Vect}\ot_{\opn{Tann}_q}\Rep(G_q)\big)=|Z(G)|\cdot (\prod_{\alpha\in\Delta}l_\alpha)\cdot (\prod_{\gamma\in \Phi^+}l_\gamma)^2.
\]
Furthermore, there is a ribbon equivalence between the category $\opn{Vect}\ot_{\opn{Tann}_q}\Rep(G_q)$ and the category of representations for a factorizable ribbon quasi-Hopf algebra $u_q$.
\end{theoremA}

As we demonstrate in a number of examples, the hypotheses of Theorem \ref{thm:2} are essentially optimal, as far as modularity of fiber is concerned.  The reader can see Section \ref{sect:examples}, and in particular Lemma \ref{lem:1703}, for more information in this regard.
\par

As a final point, in the adjoint setting we recover the ``usual" small quantum group when $q$ is of a generic odd order.  This result is well-known (see e.g.\ \cite{arkhipovgaitsgory03,davydovetingofnikshych18}) and might be viewed as a progenitor for Theorems \ref{thm:1} and \ref{thm:2}.

\begin{theoremA}[\cite{arkhipovgaitsgory03}]
Let $G$ be of adjoint type, and $q$ be of odd order.  Suppose additionally that $\opn{ord}(q)$ is coprime to both the lacing number and the determinant of the Cartan matrix for $G$. Then the fiber $\opn{Vect}\ot_{\opn{Tann}_q}\Rep(G_q)$ is a finite modular tensor category, and we have an equivalence of modular tensor categories
\begin{equation}\label{eq:193}
\opn{Vect}\ot_{\opn{Tann}_q}\Rep(G_q)\overset{\sim}\to \Rep(u_q)
\end{equation}
where $u_q$ is Lusztig's finite-dimensional Hopf algebra from \cite{lusztig90,lusztig90II}.
\end{theoremA}

The calculation \eqref{eq:193} is recovered in Section \ref{sect:lusztig_intheend}, for the sake of completeness.  We now turn our attention to our mysterious friend, $\opn{Tann}_q$.

\subsection{The Tannakian center and quantum Frobenius}
\label{sect:tann_intro}

While the primary goal of this work is to describe the fiber category $\opn{Vect}\ot_{\opn{Tann}_q}\Rep(G_q)$, we are also interested in understanding the Tannakian center in $\Rep(G_q)$ and its relation to quantum Frobenius \cite{lusztig90II,lusztig93}.
\par

To recall, at any pairing of $G$ with a root of unity $q$ we have Lusztig's quantum Frobenius functor
\begin{equation}\label{eq:203}
\opn{Fr}:\Rep(G^\ast_\varepsilon)\to \Rep(G_q),
\end{equation}
where $G^\ast$ is a semisimple dual group to $G$ and $\varepsilon$ is a dual parameter to $q$.  This functor is a braided tensor embedding, and so is an equivalence onto a tensor subcategory in $\Rep(G_q)$ which is stable under taking subquotients.
\par

The category $\Rep(G^\ast_\varepsilon)$ does not have central image in $\Rep(G_q)$, except in very special cases, and it is not even a symmetric tensor category in general.  However, as we show in Proposition \ref{prop:tannakiancenter}, there is a unique minimal quotient $G^\ast\to \dG$ under which $\Rep(\dG_\varepsilon)$ becomes Tannakian and sits centrally in $\Rep(G_q)$.  Indeed, quantum Frobenius restricts to an equivalence
\[
\opn{Fr}:\Rep(\dG_\varepsilon)\overset{\sim}\to \opn{Tann}_q\subseteq \Rep(G_q)
\]
in this case.
\par

Now, despite the fact that $\Rep(\check{G}_\varepsilon)$ is symmetric, its $R$-matrix does not always vanish.  This is to say, $\Rep(\check{G}_\varepsilon)$ still appears as representations of a \emph{quantum} rather than classical group.  However, we show that one can always nullify the braiding on this category via an explicit tensor equivalence.

\begin{theoremA}[\ref{thm:Z_classical}]
There is a symmetric tensor equivalence $F_\kappa:\Rep(\dG)\overset{\sim}\to \Rep(\dG_\varepsilon)$ and subsequent tensor embedding $\opn{Fr}_{\kappa}:\Rep(\dG)\to \Rep(G_q)$.  This embedding restricts to a symmetric tensor equivalence onto the Tannakian center
\begin{equation}\label{eq:222}
\opn{Fr}_{\kappa}:\Rep(\dG)\overset{\sim}\to \opn{Tann}_q.
\end{equation}
\end{theoremA}

The equivalence $F_\kappa$ is constructed, in part, from a choice of bilinear form $\kappa$ on the character lattice for $\dG$, and this form gives $F_\kappa$ a non-trivial tensor structure.  Furthermore, classical Tannakian reconstruction \cite{delignemilne82} tells us that the algebraic group $\dG$ is determined, up to isomorphism, by the braided tensor category $\Rep(G_q)$.  Hence this dual group $\dG$ is an \emph{invariant} of the category of quantum group representations.

\begin{remark}
When $G$ is simply-connected and $q$ is of even order, as in Theorem \ref{thm:2}, $\dG$ is precisely the Langlands dual group to $G$.  Furthermore in this case the Tannakian center is equal to the M\"uger center in $\Rep(G_q)$.  Such an identification of centers is required for modularity of the fiber category.  See Theorem \ref{thm:sc} and Lemma \ref{lem:1703}.
\end{remark}

As an important consequence of the identification \eqref{eq:222} we obtain an algebraic action of the dual group $\dG$ on the fiber $\opn{Vect}\ot_{\opn{Tann}_q}\Rep(G_q)\cong \Rep(u_q)$. Under this action the category of big quantum group representations is recovered as the subcategory of $\dG$-equivariant objects in $\Rep(u_q)$.

\begin{theoremA}[\ref{thm:act}]
There is an algebraic action of $\dG$ on $\Rep(u_q)$ under which we have a braided tensor equivalence
\begin{equation}\label{eq:236}
\Rep(G_q)\overset{\sim}\to \Rep(u_q)^{\dG}.
\end{equation}
\end{theoremA}

As we recall in Section \ref{sect:dG_action}, the equivalence \eqref{eq:236}
follows by a known calculus of equivariantization de-equivariantization which was established by Arkhipov and Gaitsgory \cite{arkhipovgaitsgory03} (see also \cite{dgno10}).

\begin{remark}
Our desire to realize this $\dG$-action on the category of small quantum group representations, in type $A_n$ at general even order $q$, was the impetuous for the present paper.  Indeed, this case was not covered in \cite{negron21}, but was the specific case of interest in recent work of Creutzig, Dimofte, Garner, and Geer on ``derived" TQFTs and quantum group representations \cite{creutzigdimoftegarnergeer}.  See also the recent contribution from Gaiotto, Moore, Neitzke, and Yan \cite{gaiottomooreneitzkeyan}, and work in progress from Feigin, Gukov, and Reshetikhin \cite{fgr}.
\end{remark}

\subsection{Related works}
\label{sect:related}

This paper completes the author's project from \cite{negron21}, where Theorem \ref{thm:2} was established under some nontrivial restrictions on $q$.  Roughly speaking, the presentation of \cite{negron21} applies at around $1/4$ of all possible values for $q$.  Similar works in this vein include those of Creutzig-Gainutdinov-Runkel \cite{creutziggainutdinovrunkel20}, Gainutdinv-Lenter-Ohrmann \cite{gainutdinovlentnerohrmann}, and Creutzig-Lentner-Rupert \cite{creutziglentnerrupert2}.  A comparison with \cite{gainutdinovlentnerohrmann}, in particular, can be found in Section \ref{sect:glo}. See also Sections \ref{sect:compI} and \ref{sect:comp}.
\par

Many of the philosophical and technical tools we've employed are adapted from papers of Arkhipov and Gaitsgory \cite{arkhipovgaitsgory03}, Takeuchi \cite{takeuchi79}, Nichols and Zoeller \cite{nicholszoeller89}, and Skryabin \cite{skryabin07}.  We rely on work of Chirvasitu and Johnson-Freyd \cite{chirvasitujfreyd13} for our global framing of things via presentable categories.

\subsection{Acknowledgements}

Thanks to Henning Andersen, Tudor Dimofte, Pavel Etingof, Eric Friedlander, Dennis Gaitsgory, David Jordan, Simon Lentner, Julia Pevtsova, Julia Plavnik, and Hans Wenzl for helpful discussions on topics related to this work.  The author was supported by NSF Grant No.\ DMS-2149817, NSF CAREER Grant No.\ DMS-2239698, and Simons Collaboration Grant No.\ 999367.

\bgroup\let\addcontentsline=\nocontentsline\section*{$1\sfrac{3}{4}$.\ Structure of the paper}\egroup

The paper has three main parts,\vspace{2mm}

\begin{tabular}{|r|l}\hline
Block 1 & Background on categorical nonsense and base change\\\hline
Block 2 & Quantum groups and quantum group representations\\\hline
Block 3 & Base change formulae for quantum groups\ \ .\\\hline
\end{tabular}
\vspace{2mm}

\noindent To explain things colloquially, the results from Block 3 are obtained by simply mashing together the findings from Block 1 with those from Block 2.  Having established the desired base change results in Block 3, we conclude the text with a recollection of, and elaboration on, the $\dG$-action on the fiber category $\opn{Vect}\ot_{\Rep(\dG)}\Rep(G_q)$.
\par

We expound: The first portion of the paper consists of Sections \ref{sect:pr}--\ref{sect:basechange_results}.  Here we recall the symmetric monoidal structure on the (2-)category of presentable categories.  We follow works of Kelly \cite{kelly82} and Chirvasitu-Johnson-Freyd \cite{chirvasitujfreyd13} in this regard.  In this setting the product $\msc{A}\otimes \msc{B}$ is constructed as the unique presentable category which represents bicocontinuous functors out of $\msc{A}\times\msc{B}$, and a relative version of this construction realizes the product $\msc{A}\otimes_\msc{E}\msc{B}$ of module categories over a given tensor category $\msc{E}$.
\par

In the second portion of the paper, Sections \ref{sect:quantum_groups}--\ref{sect:base_change}, we provide a detailed accounting of the category $\Rep(G_q)$ of integrable quantum group representations, Lusztig's quantum Frobenius functor, and the $R$-matrix for $\Rep(G_q)$.  In particular, we explain in Section \ref{sect:Tan} how one employs quantum Frobenius to calculate the Tannakian center in $\Rep(G_q)$.  We also introduce an auxiliary small quantum algebra in Section \ref{sect:sqa} which plays a key technical role in our analysis, and provide an intermediate base change result for this algebra in Section \ref{sect:base_change}.
\par

Sections \ref{sect:modular}--\ref{sect:dG_action} constitute the final portion of the paper.  We first apply results from Sections \ref{sect:quantum_groups}--\ref{sect:base_change} to obtain an abstract calculation of the fiber category $\opn{Vect}\ot_{\Rep(\dG)}\Rep(G_q)$.  In particular, we show in Theorems \ref{thm:1} and \ref{thm:2} that this fiber is a finite, integral, braided tensor category, and furthermore modular in the simply-connected even order setting.  In Section \ref{sect:smallquantumgroup}, we show that the aforementioned fiber category is identified with representations for a certain finite-dimensional quasitriangular quasi-Hopf algebra $u_q$.
\par

As we explain in Section \ref{sect:smallrep}, the representation theory for the algebra $u_q$ is regulated by the usual ``quantum group dynamics", just as in the foundational works of Lusztig \cite{lusztig88}, Parshall-Wang \cite{parshallwang91}, and Andersen-Polo-Wen \cite{andersenpolowen91,andersenpolowen92} on the subject.  See Sections \ref{sect:smallrep_1} and \ref{sect:smallrep_2} in particular.
\par

Section \ref{sect:dG_action} is dedicated to a presentation of the action of the dual group  $\dG$ on the fiber $\opn{Vect}\ot_{\Rep(\dG)}\Rep(G_q)$, following \cite{arkhipovgaitsgory03}, and the subsequent calculus of equivariantization and de-equivariantization which connects the categories of big and small quantum group representations.  We explain how such phenomena occur both in concrete and abstract terms.  In Appendices \ref{sect:A} and \ref{sect:B} we cover some technical details for Theorem \ref{thm:1021} and enumerate some basic facts concerning connected group actions on presentable abelian categories.

\subsection{Frequently used symbols}
\hspace{1mm}
\begin{multicols}{2}
\begin{itemize}
\item $G$, $X$ \hfill \S\ \ref{sect:c_data}\vspace{1mm}
\item $q$, $q_\alpha$, $l_\gamma$ \hfill \S\ \ref{sect:q_data}\vspace{1mm}
\item $G^\ast$, $X^\ast$, $\varepsilon$\hfill  \S\ \ref{sect:OGdual}\vspace{1mm}
\item $X^{\text{M\"ug}}$, $X^{\opn{Tan}}$\hfill \S\ \ref{sect:623}\vspace{1mm}
\item $\dG$, $\kappa$, $\opn{Fr}_\kappa$\hfill \S\ \ref{sect:classic_Z}\vspace{1mm}
\item $\Sigma$, $\opn{rad}(q,\kappa)$\hfill \S\ \ref{sect:summary}\vspace{1mm}
\item $u_{q,\kappa}$\hfill \S\ \ref{sect:sqa2}\vspace{1mm}
\item $\overline{\opn{Fr}}_\kappa$\hfill \S\ \ref{sect:vq}\vspace{1mm}
\item $u_q$\hfill \S\ \ref{sect:uqkappa}
\end{itemize}
\end{multicols}

\setcounter{tocdepth}{1}
\tableofcontents

\section{Tensorial backgrounds}
\label{sect:pr}

Throughout $k$ is an algebraically closed field of characteristic $0$.  By a vector space, algebra, scheme, etc.\ we mean a $k$-vector space, $k$-algebra, $k$-scheme, etc.  We recall some basic notions for presentable categories, and (braided) tensor categories.

\subsection{Key}

Our standard reference for monoidal categories is \cite{egno15}.  We take the following notions for granted:
\begin{itemize}
\item Monoidal categories, braided monoidal categories, and symmetric monoidal categories \cite[Definitions 2.1.1, 8.1.1, 8.1.12]{egno15}.
\item The Drinfeld center of a monoidal category \cite[Definition 7.13.1]{egno15}, and central functors $F:\msc{E}\to \msc{A}$ from a braided monoidal category $\msc{E}$ to an arbitrary monoidal category $\msc{A}$ \cite[Definition 8.8.6]{egno15}.
\item (Bi)module categories over a monoidal category, and (bi)module category functors \cite[Definitions 7.1.2 \& 7.2.1]{egno15}.
\item Frobenius-Perron dimensions for objects and tensor categories \cite[\S\ 4.5]{egno15}.
\end{itemize}

We let $\opn{Vect}$ and $\opn{sVect}$ denote the symmetric monoidal categories of arbitrary vector spaces and super vector spaces over $k$, respectively \cite[Example 8.2.2]{egno15}.  The symbol $\ot$ denotes the product operation on a generic monoidal category $\msc{A}$, without further decorations, and we let $\ot_k$ denote the action bifunctor on a $\opn{Vect}$-module category.  The braiding on a braided monoidal category $\msc{E}$ is generally denoted by a roman $c$, $c_{V,W}:V\ot W\overset{\sim}\to W\ot V$.

\subsection{Representations and corepresentations}

All modules over an algebra $A$ are \emph{left} modules, unless otherwise specified, and all comodules over a coalgebra $C$ are \emph{right} comodules, unless otherwise specified.  By an $A$-representation we mean a (left) $A$-module $V$ which is the sum of its finite-dimensional submodules.  By a $C$-corepresentation $W$ we simply mean a (right) $C$-comodule, and we note that $W$ is automatically the sum of its finite-dimensional subcomodules \cite[Theorem 5.1.1]{montgomery93}.  Morphisms between (co)representations are arbitrary maps of (co)modules, and we let $\Rep(A)$ and $\opn{Corep}(C)$ denote the categories of $A$-representations and $C$-corepresentations respectively.
\par

When $A$ is finite-dimensional and $C=A^\ast$ we have a natural identification of categories $\Rep(A)=\opn{Corep}(C)$.  See \cite[\S\ 1.6]{montgomery93} for more details.

\subsection{Presentable categories}

We recall that a presentable category (also called a \emph{locally presentable category}) is a cocomplete category $\msc{C}$ which is $\lambda$-compactly generated for some regular cardinal $\lambda$.  For an expanded definition one can see \cite{kelly82,adamekrosicky94}, or \cite{chirvasitujfreyd13}.
\par

From a practical perspective, any cocomplete abelian category which is generated by a small collection of compact object is presentable.  This covers all categories of sheaves, representations, corepresentations, modules, etc.\ which your everyday representation theorist might encounter in practice, and in particular covers all categories which appear in this text.
\par

The following definition is standard.

\begin{definition}
A cocontinuous functor between presentable categories $\msc{M}$ and $\msc{N}$ is a functor $F:\msc{M}\to \msc{N}$ which commutes with arbitrary (small) colimits.
\end{definition}

\subsection{Presentable monoidal categories}
\label{sect:pmon}

A presentable monoidal category is a presentable category with a monoidal structure for which the product operation $\msc{A}\times \msc{A}\to \msc{A}$ commutes with small colimits in each factor.
\par

A presentable left (resp.\ right) module category over a presentable monoidal category $\msc{A}$ is a presentable category $\msc{M}$ which is equipped with an associative action map $\msc{A}\times \msc{M}\to \msc{M}$ (resp.\ $\msc{M}\times \msc{A}\to \msc{M}$) which commutes with small colimits in each factor.

\subsection{Linear categories}

For us, a presentable linear category is a presentable module category over $\opn{Vect}$.  We let $\ot_k$ denote the $\opn{Vect}$-action on a given linear category $\msc{M}$.  One can check that our notion of a presentable linear category agrees with the standard notion of a presentable $\opn{Vect}$-enriched category.

\begin{lemma}
For any linear category $\msc{M}$, the morphisms $\Hom_{\msc{M}}(X,Y)$ admit a unique $k$-vector space structure under which we have an adjunction
\[
\Hom_{\msc{M}}(V\ot_k X,Y)\cong \Hom_{\opn{Vect}}(V,\Hom_{\msc{C}}(X,Y))
\]
and under which composition becomes bilinear.  Furthermore, for any linear functor $F:\msc{M}\to \msc{N}$, i.e.\ $\opn{Vect}$-module category functor, the induced maps
\[
F:\Hom_{\msc{M}}(X,Y)\to \Hom_{\msc{N}}(FX,FY)
\]
are all maps of $k$-vector spaces.
\par

The resulting functor from the $2$-category of presentable linear categories to presentable $\opn{Vect}$-enriched categories is an equivalence.
\end{lemma}

The proof is by abstract nonsense, and appears in \cite[Lemma 2.2.2]{schauenburg92} for example.

\subsection{Tensor categories and their module categories}

A tensor category is a presentable abelian monoidal category $\msc{A}$ which satisfies the following:
\begin{itemize}
\item $\msc{A}$ is compactly generated, and the compact objects $\msc{A}^c$ in $\msc{A}$ form an essentially small abelian monoidal subcategory.
\item The compact and rigid objects in $\msc{A}$ agree.
\item $\msc{A}$ comes equipped with a cocontinuous central monoidal functor $\opn{unit}_{\msc{A}}:\opn{Vect}\to \msc{A}$ (so that $\msc{A}$ becomes a presentable linear category via the induced action of $\opn{Vect}$).
\item $\msc{A}^c$ is a locally finite category over $k$, in the sense that all objects are of finite length and $\Hom$ sets are finite-dimensional vector spaces.
\item The unit object $\1$ in $\msc{A}$ is simple.
\end{itemize}
The final three points just say that $\msc{A}^c$ is a tensor category in the usual sense of \cite{egno15}.
\par

A tensor functor between tensor categories is an exact, cocontinuous, monoidal functor $F:\msc{A}\to \msc{B}$ which is paired with a choice of natural isomorphism $F\opn{unit}_{\msc{A}}\cong \opn{unit}_{\msc{B}}$ of central monoidal functors from $\opn{Vect}$.  Note that any tensor functor $F$ must preserve rigid objects, via monoidality, and hence must preserve compact objects.  So we obtain an equivalence
\[
\{\text{Our tensor categories}\}\overset{\sim}\to \{\text{\cite{egno15}'s tensor categories}\},\ \ \msc{A}\mapsto \msc{A}^c,
\]
whose inverse is given by taking the Ind-category.

\begin{definition}[{\cite{etingnikshychostrik05,etingofostrik04}}]
A tensor category is called finite if it admits a compact projective generator.  A tensor category is called fusion if it is finite and semisimple.
\end{definition}

All module categories $\msc{M}$ over a given tensor category $\msc{A}$ are assumed to be presentable and abelian, and to furthermore satisfy the following (cf.\ \cite[Definition 7.3.1]{egno15}):
\begin{itemize}
\item $\msc{M}$ is compactly generated, and the compact objects $\msc{M}^c$ form an essentially small abelian subcategory in $\msc{M}$.
\item $\msc{M}^c$ is locally finite over $k$.
\item The action functor $\msc{M}\times \msc{A}\to \msc{M}$ (or $\msc{A}\times \msc{M}\to \msc{M}$) is exact in each factor.
\end{itemize}
Any module category over a tensor category $\msc{A}$ is again linear, via the central embedding $\opn{Vect}\to \msc{A}$ and subsequent action of $\opn{Vect}$, and we use this linear structure to assess local finiteness of $\msc{M}^c$.  Functors between such module categories are just cocontinuous module category functors, i.e.\ maps of presentable module categories over $\msc{A}$.
\par

We note that the action of $\msc{A}$ on $\msc{M}$ restricts to an action of $\msc{A}^c$ on $\msc{M}^c$, since all of the compact objects in $\msc{A}$ are rigid.  However, since a given module category functor $F:\msc{M}\to \msc{N}$ needn't preserve compact objects, restriction to the compacts is not a (2-)functorial operation in general.

\subsection{Special functors}
\label{sect:surj_embed}

Let $\msc{M}$ and $\msc{N}$ be presentable \emph{abelian} categories.  For an exact cocontinuous functor $F:\msc{M}\to \msc{N}$ we take $\langle F(\msc{M})\rangle$ to be the full subcategory in $\msc{N}$ with objects
\[
\{Y\in \msc{N}:Y\text{ is a subquotient of $F(X)$, for some $X$ in }\msc{M}\}.
\]
We say $F$ is surjective if $\langle F(\msc{M})\rangle=\msc{N}$.  We say $F$ an embedding if it restricts to an equivalence $F:\msc{M}\overset{\sim}\to \langle F(\msc{M})\rangle$.

To expand on the latter definition, an exact cocontinuous functor $F$ is an embedding provided it is fully faithful and for any $X$ in $\msc{M}$, and injection $Y\to F(X)$, there exists an injective map $Y'\to X$ in $\msc{M}$ which fits into a diagram
\[
\xymatrixrowsep{4mm}
\xymatrix{
Y\ar[rr]^\cong\ar[dr] & & F(Y')\ar[dl]\\
 & F(X) & .
}
\]
Via exactness we can similarly lift quotients of $F(X)$ along $F$ in this case.

\subsection{Symmetric \& Tannakain categories}
\label{sect:symmtann}

We say a tensor category $\msc{A}$ is of subexponential growth if any compact object $V$ in $\msc{A}$ satisfies $\opn{length}(V^{\ot n})\leq d^n_V$, for some fixed positive number $d_V$ which depends on $V$ and all $n\geq 0$.  In practice one can simply observe the following.

\begin{lemma}[{\cite[Lemma 9.11.3]{egno15}}]\label{lem:675}
If $\msc{A}$ admits a tensor functor $F:\msc{A}\to \msc{B}$ to a finite tensor category $\msc{B}$, then $\msc{A}$ is of subexponential growth.  If $\msc{A}$ admits a tensor functor $F:\msc{A}\to \msc{B}'$ to a tensor category $\msc{B}'$ which is of subexponential growth, then $\msc{A}$ is of subexponential growth.
\end{lemma}

\begin{proof}
In the first case take $d_V=\opn{FPdim}(FV)$.  In the second case take $d_V=d_{FV}$.
\end{proof}

Via Lemma \ref{lem:675} one sees immediately that all tensor categories considered in this text are of subexponential growth.  The following is Deligne's theorem.

\begin{theorem}[\cite{deligne02}]
Suppose $\msc{E}$ is a symmetric tensor category.  Then $\msc{E}$ admits a symmetric tensor functor $F:\msc{E}\to \opn{sVect}$ if and only if $\msc{E}$ is of subexponential growth.  Furthermore, in this case the functor $F$ is uniquely determined up to an isomorphism of tensor functors.
\end{theorem}

We note that Deligne's theorem fails when the base field is allowed to be of finite characteristic.  So, it is very important that we work in characteristic $0$ here.

\begin{definition}
A symmetric tensor category $\msc{E}$ is called Tannakian if it admits a symmetric tensor functor to $\opn{Vect}$.
\end{definition}

As established in works of Saavedra Rivano, Deligne, and Milne \cite{saavedra72,delignemilne82}, any Tannakian category $\msc{E}$ is recovered as the representation category $\Rep(H)$ for an associated affine algebraic group $H$.  This algebraic group is, in particular, the automorphism group of any choice of a symmetric fiber functor for $\msc{E}$.

\begin{corollary}\label{cor:E_tann}
Suppose $\msc{E}$ is a symmetric tensor category of subexponential growth.  Then $\msc{E}$ admits a unique maximal Tannakian subcategory $\msc{E}_{\opn{Tan}}$.  This subcategory is identified with the kernel of any symmetric fiber functor $F:\msc{E}\to \opn{sVect}$.
\end{corollary}

By the kernel of $F$ we simply mean the preimage of $\opn{Vect} \subseteq \opn{sVect}$ along $F$.

\begin{proof}
Let $\msc{E}'$ be any Tannakian subcategory in $\msc{E}$.  Then any symmetric fiber functor $F:\msc{E}\to \opn{sVect}$ restricts to a fiber functor for $\msc{E}'$.  Hence $\msc{E}'$ lies in the kernel of $F$, by uniqueness of the fiber functor for $\msc{E}'$.  It follows that all Tannakian subcategories in $\msc{E}$ lie in the kernel of $F$.  One sees that $\opn{ker}(F)$ itself is Tannakian, so that the kernel of $F$ is the maximal Tannakian subcategory in $\msc{E}$.  In particular, such a maximal subcategory exists.
\end{proof}

\subsection{M\"uger centers \& Tannakian centers}

Recall that the M\"uger center in a braided tensor category $\msc{A}$ is, by definition, the subcategory of all objects $V$ in $\msc{A}$ which have trivial square braiding,
\[
c^2_{V,-}:=c_{-,V}c_{V,-}=id_{V\ot-}.
\]
Following standard practice we let $Z_2(\msc{A})$ denote the M\"uger center in $\msc{A}$, and note that $Z_2(\msc{A})$ is a symmetric tensor subcategory in $\msc{A}$.\footnote{The notation $Z_3(\msc{A})$ is relatively common as well, to emphasize that the output category is $E_3$ rather than $E_2$.}

\begin{definition}\label{def:Ztann}
Let $\msc{A}$ be a braided tensor category which is of subexponential growth.  The Tannakian center $Z_{\opn{Tan}}(\msc{A})$ in $\msc{A}$ is the maximal Tannakian subcategory in the M\"uger center $Z_2(\msc{A})$.
\end{definition}

\section{Exact sequences and pointed module categories}

We discuss a generalization of Bruguieres and Natale's notion \cite{bruguieresnatale11} of exact sequences for tensor categories.  The model for such a generalized exact sequence is a reduction sequence
\[
\msc{E}\to \msc{A}\to \opn{Vect}\ot_{\msc{E}}\msc{A},
\]
where $\msc{E}\to \msc{A}$ is an embedding of tensor categories and $\opn{Vect}\ot_{\msc{E}}\msc{A}$ inherits a module category structure over $\msc{A}$, but not a tensor structure in general.  (See Section \ref{sect:module_prod}.)

\subsection{Comodule algebras and module coalgebras}

Before beginning our general discussion, we first recall a basic source of examples.
\par

Given a Hopf algebra $A$, a right $A$-comodule algebra is an algebra $B$ which is equipped with an $A$-comodule structure for which the coaction map $\rho:B\to B\ot A$ is a map of algebras.  Similarly, a right $A$-module coalgebra is a coalgebra $C$ which is equipped with a right $A$-module structure for which the action map $\opn{act}:C\ot A\to C$ is a map of coalgebras.  Given an $A$-comodule algebra $B$, and an $A$-module coalgebra $C$, the $A$-(co)actions provide the categories
\[
\opn{Rep}(B)\ \ \text{and}\ \ \opn{Corep}(C)
\]
with right module category structures over $\Rep(A)$ and $\opn{Corep}(A)$, respectively.

\subsection{Pointed module categories}

\begin{definition}
A module category $\msc{M}$ over a tensor category $\msc{A}$ is called pointed if $\msc{M}$ comes equipped with a distinguished simple object $\1_\msc{M}\in \msc{M}$.
\end{definition}

Given any pointed module category over $\msc{A}$, the distinguished simple object $\1_{\msc{M}}$ specifies, and is specified by, a uniquely associated map of module categories
\[
u_{\msc{M}}:\msc{A}\to \msc{M}
\]
which sends $\1_{\msc{A}}$ to $\mbf{1}_\msc{M}$.

\begin{example}
Let $A$ be a Hopf algebra and $B\subseteq A$ be a right $A$-comodule subalgebra.  So, the comultiplication for $A$ restricts to a coaction $\rho:B\to B\ot A$ for $B$.  This coaction gives $\Rep(B)$ the structure of a right $\Rep(A)$-module category, and restriction of the counit $\epsilon:B\to k$ identifies a trivial representation for $B$.  This trivial representation gives $\Rep(B)$ the structure of a pointed $\Rep(A)$-module category.
\end{example}

\begin{example}
Let $A$ be a Hopf algebra and $A\to C$ be a map of right $A$-module coalgebras.  Then $\opn{Corep}(C)$ is naturally a right $\opn{Corep}(A)$-module category and the unit map for $A$ provides $C$ with a trivial corepresentation $k\to C$.  This trivial corepresentation gives $\opn{Corep}(C)$ the structure of a pointed $\opn{Corep}(A)$-module category.
\end{example}

\subsection{Normality for pointed module categories}

Let $\msc{A}$ be a tensor category and $\msc{M}$ be a pointed module category over $\msc{A}$.
\par

We define the kernel of the structure map $u_{\msc{M}}:\msc{A}\to\msc{M}$ to be the full subcategory of objects $X$ in $\msc{A}$ with $u_{\msc{M}}(X)\cong \1_{\msc{M}}^{\oplus I}$, for some indexing set $I$.  We note that the unit object in $\msc{M}$ specifies a linear embedding $\opn{Vect}\to \msc{M}$, and the kernel of $u_{\msc{M}}$ is alternatively identified as the pullback
\begin{equation}\label{eq:ker}
\xymatrix{
\opn{ker}(u_{\msc{M}})\ar@{-->}[rr]\ar@{-->}[d] & & \msc{A}\ar[d]\\
\opn{Vect}\ar[rr] & & \msc{M}.
}
\end{equation}
For any $M$ in $\msc{M}$ we have the largest trivial subobject in $\msc{M}$,
\[
\opn{triv}(M)=\opn{im}\left(\bigoplus_{f\in \Hom_{\msc{M}}(\1,M)}\1\to M\right).
\]

\begin{definition}[{cf.\ \cite{bruguieresnatale11}}]\label{def:normal}
For a pointed module category $\msc{M}$ over a tensor category $\msc{A}$, the structure map $u_{\msc{M}}:\msc{A}\to \msc{M}$ is said to be normal, or $\msc{M}$ is said to be normal, if the following hold:
\begin{enumerate}
\item[(a)] The kernel $\opn{ker}(u_{\msc{M}})$ is a tensor subcategory in $\msc{A}$.
\item[(b)] For any $X$ in $\msc{A}$ there is a subobject $X'\subseteq X$ with $u_{\msc{M}}(X')=\opn{triv}(u_{\msc{M}}(X))$.
\end{enumerate}
\end{definition}

\begin{example}\label{ex:529}
Let $B\subseteq A$ be a right coideal subalgebra.  Suppose that the two reductions $k\ot_BA$ and $A\ot_Bk$ are equal.  In this case $C=k\ot_BA$ is a Hopf quotient of $A$, and the kernel of the restriction functor $\Rep(A)\to \Rep(B)$ is identified with the embedded subcategory $\Rep(C)\subseteq \Rep(A)$.  For any $A$-representation $V$ the $B$-invariants $V^B$ are identified with the maximal $C$-submodule in $V$, and we see that $\Rep(B)$ is normal over $\Rep(A)$.
\end{example}

This example dualizes in the obvious way.

\subsection{Mixed exact sequences}

\begin{definition}\label{def:531}
A mixed exact sequence of tensor/module categories is a pairing of a tensor functor $F:\msc{E}\to \msc{A}$, and a pointed $\msc{A}$-module category $\msc{M}$, such that
\begin{enumerate}
\item[(a)] The structure map $u_{\msc{M}}:\msc{A}\to \msc{M}$ is normal.
\item[(b)] $u_{\msc{M}}$ is surjective.
\item[(c)] $F$ is an equivalence onto the kernel of $u_{\msc{M}}$.
\item[(d)] The right adjoint $\opn{ind}:\msc{M}\to \msc{A}$ to the structure map $u_{\msc{M}}$ is faithfully exact.
\end{enumerate}
\end{definition}

Note that condition (c) requires $F$ to be an embedding, in the sense of Section \ref{sect:surj_embed}.  One might express a mixed exact sequence compactly via its constituent functors
\[
\msc{E}\overset{F}\to \msc{A}\overset{u}\to \msc{M}.
\]

\begin{example}\label{ex:556}
Suppose that $A$, $B$, and $C$ are as in Example \ref{ex:529}, and that $A$ is finite-dimensional.  Then we have the mixed possibly-exact sequence
\begin{equation}\label{eq:587}
\Rep(C)\to \Rep(A)\to \Rep(B),
\end{equation}
and Skryabin's faithful flatness theorem \cite[Theorem 6.1]{skryabin07} tells us that the induction functor $\opn{ind}:\Rep(B)\to \Rep(A)$ is faithfully exact.  So the above sequence \eqref{eq:587} is in fact mixed exact.
\end{example}

As with Example \ref{ex:529}, Example \ref{ex:556} dualizes via corepresentations.

\begin{example}
Suppose $\msc{E}\to \msc{A}\to \msc{B}$ is an exact sequence of finite tensor categories in the sense of Bruguieres and Natale \cite{bruguieresnatale11}.  If we give $\msc{B}$ its natural pointed module category structure over $\msc{A}$, with $\1_{\msc{B}}$ equal to the unit in $\msc{B}$, this sequence is mixed exact in the sense of Definition \ref{def:531}.  Here faithful exactness of induction $\opn{ind}:\msc{B}\to \msc{A}$ follows from the categorical freeness result of Etingof and Ostrik \cite[Theorem 2.5]{etingofostrik04}.
\end{example}

\subsection{Fiber functors for module categories}

In addressing mixed exact sequences in practice, the following notion proves useful.

\begin{definition}
Given a pointed module category $\msc{M}$ over a tensor category $\msc{A}$, a fiber functor for $\msc{M}$ is the specification of a fiber functor for $\msc{A}$ along with an exact map of pointed module categories $f:\msc{M}\to \opn{Vect}$.
\end{definition}

This situation occurs, for example, when $\msc{A}$ is the category of corepresentations for a Hopf algebra $A$ and $\msc{M}$ is the category of corepresentations for a quotient module coalgebra $A\to C$.  In this case the fiber functors are simply the forgetful functors to $\opn{Vect}$.

\section{Functor categories and relative tensor products}

We recall how one ``does algebra" in the 2-category of presentable categories.  All of the results in this section are known, and can be deduced from Chirvasitu and Johnson-Fryed \cite{chirvasitujfreyd13} for example.  More robust forms of the results herein can be found in works of Lurie \cite{lurie09,lurieha}.
\par

The materials from this section, as well as Sections \ref{sect:tann_reconstructor} and \ref{sect:basechange_results}, form the categorical foundations for our analysis of quantum group representations which follows.

\subsection{Functor categories}

\begin{definition}
For presentable categories $\msc{A}$ and $\msc{B}$ we let $\uFun(\msc{A},\msc{B})$ denote the category of cocontinuous functors from $\msc{A}$ to $\msc{B}$, with all natural tranformations.  For left (or right) presentable $\msc{E}$-module categories $\msc{M}$ and $\msc{N}$, let $\uFun_{\msc{E}}(\msc{M},\msc{N})$ denote the category of cocontinuous $\msc{E}$-module functors between $\msc{M}$ and $\msc{N}$ with corresponding natural transformations.
\par

We let $\Fun(\msc{A},\msc{B})$ and $\Fun_{\msc{E}}(\msc{M},\msc{N})$ denote the groupoids of functors and natural isomorphisms in $\uFun(\msc{A},\msc{B})$ and $\uFun_{\msc{E}}(\msc{M},\msc{N})$, respectively.
\end{definition}

Note that each functor category $\uFun_{\msc{E}}(\msc{M},\msc{N})$ is cocomplete, with colimits calculated pointwise via colimits in $\msc{N}$.  The category of plain cocontinuous functors $\uFun(\msc{M},\msc{N})$ is furthermore presentable \cite[Corollary 2.2.5]{chirvasitujfreyd13}.
\par

When $\msc{M}$ is a presentable $\msc{E}$-module category, and $\msc{B}$ is an arbitrary presentable category, the functor category $\uFun(\msc{M},\msc{B})$ inherits a natural $\msc{E}$-action via the action on $\msc{M}$.  This action realizes $\uFun(\msc{M},\msc{B})$ as a presentable $\msc{E}$-module category.

\subsection{Bilinear functors and the absolute product}

We let $\underline{\opn{Bilin}}(\msc{A}_1\times \msc{A}_2, \msc{B})$ denote the category of functors from the cartesian product which are cocontinuous in each variable, and let $\opn{Bilin}(\msc{A}_1\times\msc{A}_2,\msc{B})$ denote the corresponding groupoid.
\par

For a right $\msc{E}$-module category $\msc{M}$, a left $\msc{E}$-module category $\msc{N}$, and an arbitrary presentable category $\msc{B}$, an $\msc{E}$-bilinear functor is a functor
\[
F:\msc{M}\times \msc{N}\to \msc{B}
\]
which is cocontinuous in each variable and comes equipped with natural isomorphisms $F(M,V\ot N)\cong F(M\ot V,N)$ at each $M$ in $\msc{M}$, $N$ in $\msc{N}$, and $V$ in $\msc{E}$.  We require these natural isomorphisms to be associative and unital in the expected ways.
\par

We let $\underline{\opn{Bilin}}_{\msc{E}}(\msc{M}\times \msc{N},\msc{B})$ denote the category of $\msc{E}$-bilinear functors and natural transformations.  These natural transformations are assumed to commute with the structure maps $F(-,-\ot -)\cong F(-\ot -,-)$.  As usual, $\opn{Bilin}_{\msc{E}}(\msc{M}\times \msc{N},\msc{B})$ denotes the groupoid of $\msc{E}$-bilinear functors with natural isomorphisms.
\par

We have the following representability result of Chirvasity and Johnson-Freyd (cf.\ \cite[\S\ 4.8.1]{lurieha}).

\begin{theorem}[{\cite[Corollary 2.2.5]{chirvasitujfreyd13}}]\label{thm:123}
For presentable categories $\msc{A}_1$ and $\msc{A}_2$, there is a presentable category $\msc{A}_1\otimes \msc{A}_2$ which admits a universal bilinear functor $\opn{univ}(=\opn{univ}_{\msc{A}_1,\msc{A}_2}):\msc{A}_1\times \msc{A}_2\to \msc{A}_1\ot\msc{A}_2$.  This is to say, restricting along $\opn{univ}$ provides an equivalence of categories
\[
\opn{univ}^\ast:\uFun(\msc{A}_1\ot\msc{A}_2,\msc{B})\overset{\sim}\to \underline{\opn{Bilin}}(\msc{A}_1\times \msc{A}_2,\msc{B})
\]
at arbitrary presentable $\msc{B}$.
\end{theorem}

We note that $\opn{univ}^\ast$ induces an equivalence of groupoids
\[
\opn{univ}^\ast:\Fun(\msc{A}_1\ot\msc{A}_2,\msc{B})\overset{\sim}\to \opn{Bilin}(\msc{A}_1\times \msc{A}_2,\msc{B})
\]
as well, and also that the product category $\msc{A}_1\ot\msc{A}_2$ is uniquely determined up to equivalence.  We have the expected adjunctions
\[
\uFun(\msc{A}_1\ot\msc{A}_2,\msc{B})\cong\underline{\opn{Bilin}}(\msc{A}_1\times \msc{A}_2,-)\cong \uFun(\msc{A}_1,\uFun(\msc{A}_2,-))
\]
\[
F\mapsto (X\mapsto F(X,-))
\]
and
\[
\uFun(\msc{A}_1\ot\msc{A}_2,\msc{B})\cong\underline{\opn{Bilin}}(\msc{A}_1\times\msc{A}_2,-)\cong \uFun(\msc{A}_2,\uFun(\msc{A}_1,-))
\]
\[
F'\mapsto (Y\mapsto F'(-,Y))
\]
(cf.\ \cite[\S\ 6.5]{kelly82}).
\par

One can employ the above adjunctions to find that the products $(\msc{A}_1\ot\msc{A}_{12})\ot\msc{A}_2$ and $\msc{A}_1\ot(\msc{A}_{12}\ot\msc{A}_2)$ both represent a $2$-functor of multilinear maps from $\msc{A}_1\times\msc{A}_{12}\times\msc{A}_2$.  From this one deduces an associativity equivalence for the product $\ot$ which is unique up to unique natural isomorphism.  The symmetry for the product $\times$ furthermore provides a symmetry on $\ot$, so that the $2$-category of presentable categories with cocontinuous functors becomes symmetric monoidal under $\ot$ \cite[Corollary 2.2.5]{chirvasitujfreyd13}.  The unit for this product is the monoidal category $\opn{Set}$.
\par

In terms of the symmetric monoidal structure $\ot$, a presentable monoidal category $\msc{A}$ (Section \ref{sect:pmon}) is a presentable category equipped with an associative product operation $\msc{A}\ot \msc{A}\to \msc{A}$ and unit map $\opn{Set}\to \msc{A}$.  A left (resp.\ right) module category over $\msc{A}$ is a presentable category $\msc{M}$ with an associative and unital action $\msc{A}\ot \msc{M}\to \msc{M}$ (resp.\ $\msc{M}\ot \msc{A}\to \msc{M}$).

\subsection{Relative products and relative functor categories}

The following result is known, and can be deduced from Theorem \ref{thm:123} in conjunction with cocompleteness of the $2$-category of presentable categories \cite[Proposition 2.1.11]{chirvasitujfreyd13}.  See for example \cite[Definition 3.14]{benzvibrochierjordan18}.

\begin{proposition}[\cite{benzvibrochierjordan18}]
Let $\msc{M}$ and $\msc{N}$ be right and left $\msc{E}$-module categories, respectively.  There is a presentable category $\msc{M}\ot_{\msc{E}}\msc{N}$ which admits a universal $\msc{E}$-bilinear functor $\opn{univ}^{\msc{E}}_{\msc{M},\msc{N}}:\msc{M}\times\msc{N}\to \msc{M}\ot_{\msc{E}}\msc{N}$.
\end{proposition}

Following \cite{etingofnikshychostrik10,douglasspsnyder19,benzvibrochierjordan18}, we refer to the product $\ot_{\msc{E}}$ as the relative tensor product for module categories over $\msc{E}$.  We now employ the absolute product $\ot$ to verify that the relative functor categories $\underline{\Fun}_{\msc{E}}(\msc{M},\msc{N})$ are presentable.

\begin{proposition}
For presentable $\msc{E}$-module categories $\msc{M}$ and $\msc{N}$, the functor category $\uFun_{\msc{E}}(\msc{M},\msc{N})$ is presentable.
\end{proposition}

\begin{proof}
We can realize $\uFun_{\msc{E}}(\msc{M},\msc{N})$ as the limit of a diagram of categories
\[
\begin{tikzpicture}
\node at (-3.2,0) {$\uFun(\msc{M},\msc{N})$};
\node at (0,0) {$\uFun(\msc{E}\ot\msc{M},\msc{N})$};
\node at (4.02,0) {$\uFun(\msc{E}\ot\msc{E}\ot\msc{M},\msc{N})$};
\draw[->] (-2.15,.15) -- (-1.4,.15);
\draw[<-] (-2.15,0) -- (-1.4,0);
\draw[->] (-2.15,-.15) -- (-1.4,-.15);
\draw[->] (1.4,.3) -- (2.2,.3);
\draw[<-] (1.4,.15) -- (2.2,.15);
\draw[->] (1.4,0) -- (2.2,0);
\draw[<-] (1.4,-.15) -- (2.2,-.15);
\draw[->] (1.4,-.3) -- (2.2,-.3);
\end{tikzpicture}
\]
where all of the arrows are provided by action maps or the unit map $\opn{Set}\to \msc{E}$.  Since the categories $\uFun(\msc{E}^{\ot m}\ot\msc{M},\msc{N})$ are presentable by \cite[Corollary 2.2.5]{chirvasitujfreyd13}, and the 2-category of presentable categories is closed under limits in the ambient category of all locally small categories \cite[Proposition 2.1.11]{chirvasitujfreyd13}, we see that $\uFun_{\msc{E}}(\msc{M},\msc{N})$ is presentable.
\end{proof}

As with plain bilinear functors, we have the adjunctions
\[
\uFun_{\msc{E}}(\msc{M},\Fun(\msc{N},-))\cong \underline{\opn{Bilin}}_{\msc{E}}(\msc{M}\times \msc{N},-)\cong \uFun_{\msc{E}}(\msc{N},\uFun(\msc{M},-)).
\]

\subsection{The monoidal $2$-category of bimodule categories}
\label{sect:module_prod}

Consider $(\msc{E}_i,\msc{E}_j)$-bimodule categories $\msc{M}_{ij}$.  The actions maps for $\msc{M}_{ij}$ induce a bimodule structure on the absolute product
\[
\msc{E}_0\ot (\msc{M}_{01}\ot \msc{M}_{12})\ot \msc{E}_2\to \msc{M}_{01}\ot \msc{M}_{12}
\]
which is unique up to unique isomorphism.  This bimodule structure then descend to bimodule structures on the relative product $\msc{M}_{01}\ot_{\msc{E}_1}\msc{M}_{12}$ \cite[Proposition 3.13]{greenough10}.
\par

We now observe an associativity equivalence which fits into a $2$-diagram
\[
\xymatrix{
& \msc{M}_{01}\times \msc{M}_{12}\times \msc{M}_{23}\ar[dl]\ar[dr]\\
\msc{M}_{01}\ot_{\msc{E}_1} (\msc{M}_{12}\ot_{\msc{E}_2} \msc{M}_{23})\ar[rr]^{\sim} & & (\msc{M}_{01}\ot_{\msc{E}_1}\msc{M}_{12})\ot_{\msc{E}_2}\msc{M}_{23},
}
\]
and is uniquely determined up to unique natural isomorphism.  This equivalence furthermore admits a natural $(\msc{E}_0,\msc{E}_3)$-bimodule structure \cite[Proposition 3.13]{greenough10}.  In this way we obtain a ``relative monoidal structure" on the category of presentable bimodule categories.  Or, more precisely, we have a monoidal $2$-category of presentable monoidal categories whose morphisms are given by presentable bimodules \cite{johnsonscheimbauer17,brochierjordansnyder}.

\subsection{Products of tensor categories}

As one finds in \cite[Theorem 6.2]{greenough10}, for example, the product $\msc{A}\ot_{\msc{E}}\msc{B}$ of presentable monoidal categories over a fixed braided monoidal category $\msc{E}$ inherits a natural monoidal structure.  In order to give a more efficient presentation however, we ``go up a dimension" and describe this monoidal structure in the braided context.

\begin{definition}
Given a presentable symmetric monoidal category $\msc{E}$, a braided monoidal category over $\msc{E}$ is a presentable braided monoidal category $\msc{A}$ equipped with a M\"uger central monoidal functor $u_{\msc{A}}:\msc{E}\to \msc{A}$.  A map of braided monoidal categories over $\msc{E}$ is a $2$-diagram of cocontinuous braided monoidal functors
\[
\xymatrix{
 & \msc{E}\ar[dl]_{u_{\msc{A}}}\ar[dr]^{u_{\msc{A}}} \\
\msc{A}\ar[rr]^{F} & & \msc{B}.
}
\]
A natural isomorphism between braided monoidal functors $F,\ F':\msc{A}\to \msc{B}$ over $\msc{E}$ is simply an isomorphism of monoidal functors $\xi:F\overset{\sim}\to F'$.
\end{definition}

\begin{remark}
We note that any isomorphisms between such functors is automatically compatible with the braidings on $\msc{A}$ and $\msc{B}$.  Furthermore, any isomorphism between functors fits into a uniquely associated tetrahedron over $\msc{E}$ whose sides are provided by $u_\msc{A}$, $u_\msc{B}$, $F:\msc{A}\to \msc{B}$, $F':\msc{A}\to \msc{B}$, and $id_{\msc{B}}$.
\end{remark}

For braided monoidal categories $\msc{A}$ and $\msc{B}$ over $\msc{E}$, the product $\msc{A}\ot_{\msc{E}}\msc{B}$ inherits a unique braided monoidal structure under which:
\begin{enumerate}
\item[(a)] The structure maps $f:\msc{A}\to \msc{A}\ot_{\msc{E}}\msc{B}$ and $f':\msc{B}\mapsto \msc{A}\ot_{\msc{E}}\msc{B}$ are maps of braided monoidal categories.
\item[(b)] The images of $\msc{A}$ and $\msc{B}$ are mutually (M\"uger) centralizing in $\msc{A}\ot_{\msc{E}}\msc{B}$.
\item[(c)] The isomorphism $f|_\msc{E}\cong f'|_\msc{E}$ provided by the balanced structure is an isomorphism of monoidal functors.
\end{enumerate}
We note that these conditions imply that $\msc{E}$ sits (M\"uger) centrally in $\msc{A}\ot_{\msc{E}}\msc{B}$.  In somewhat more explicit terms, the monoidal structure on $\msc{A}\ot_{\msc{E}}\msc{B}$ is provided by the unique cocontinuous functor which completes the $2$-diagram
\[
\xymatrix{
\msc{A}\times \msc{B}\times \msc{A}\times \msc{B}\ar[d]\ar[r] & \msc{A}^2\times \msc{B}^2\ar[r] & \msc{A}\times \msc{B}\ar[d]\\
(\msc{A}\ot_{\msc{E}}\msc{B})\ot(\msc{A}\ot_{\msc{E}}\msc{B})\ar[rr] & & \msc{A}\ot_{\msc{E}}\msc{B}.
}
\]

For $\msc{A}$ and $\msc{B}$ as above, the braided monoidal category $\msc{A}\ot_\msc{E}\msc{B}$ additionally represents a $2$-functor
\[
\opn{Bilin}^{\opn{E}_2}_{\msc{E}}(\msc{A}\times \msc{B},-),
\]
whose value on a given braided monoidal category $\msc{C}$ over $\msc{E}$ is the groupoid of pairs
\begin{equation}\label{eq:658}
\opn{Bilin}^{\opn{E}_2}_{\msc{E}}(\msc{A}\times \msc{B},\msc{C})=\left\{
\begin{array}{c}
\text{Pairs of braided monoidal functors }f:\msc{A}\to \msc{C}\ \text{and}\\
f':\msc{B}\to \msc{C}\ \text{over $\msc{E}$ for which $f(\msc{A})$ centralizes $f'(\msc{B})$ in $\msc{C}$}
\end{array}\right\}.
\end{equation}
We note that the $\msc{E}$-structures here specifies an isomorphism of monoidal functors $f|_{\msc{E}}\cong f'|_{\msc{E}}$, for any such pair of $f$ and $f'$.  Morphisms between such pairs are independent pairs of natural isomorphisms.
\par

Given a symmetric monoidal functor $\tau:\msc{E}\to \msc{K}$, and a braided monoidal category $\msc{C}$ over $\msc{K}$, we consider $\msc{C}$ as a category over $\msc{E}$ by restricting along $\tau$.  Below we let $\Fun^{\opn{E}_2}_{\msc{E}}(\msc{A},\msc{B})$ denote the groupoid of braided monoidal functors over a given symmetric base $\msc{E}$.

\begin{proposition}\label{prop:bc}
Let $\msc{E}\to \msc{K}$ be a symmetric tensor functor, and $\msc{A}$ be a braided tensor category over $\msc{E}$.  Then the braided monoidal category $\msc{K}\ot_{\msc{E}}\msc{A}$ represents the $2$-functor
\[
\Fun_{\msc{E}}^{\opn{E}_2}(\msc{A},-):\{\text{\rm braided monoidal cats over }\msc{K}\}\to \opn{Groupoid}.
\]
\par

Furthermore, when the image of $\msc{E}$ generates $\msc{K}$ under colimits, $\msc{K}\ot_{\msc{E}}\msc{A}$ completes the pushout diagram
\[
\xymatrix{
\msc{E}\ar[r]\ar[d] & \msc{A}\ar@{-->}[d]\\
\msc{K}\ar@{-->}[r] & \msc{K}\ot_{\msc{E}}\msc{A}
}
\]
in the $2$-category of presentable braided monoidal categories.
\end{proposition}

\begin{proof}
For a braided monoidal category $\msc{C}$ over $\msc{K}$, the information of a braided monoidal functor $f:\msc{A}\to \msc{C}$ over $\msc{E}$ specifies an object $(u,f)$ in $\opn{Bilin}^\ot_{\msc{E}}(\msc{K}\times \msc{A},\msc{C})$ for which the map $u:\msc{K}\to \msc{C}$ is the unit map.  This gives a fully faithful embedding
\[
\Fun^{\opn{E}_2}_{\msc{E}}(\msc{A},\msc{C})\to \opn{Bilin}^{\opn{E}_2}_{\msc{E}}(\msc{K}\times \msc{A},\msc{C})\cong\Fun^{\opn{E}_2}(\msc{K}\ot_{\msc{E}}\msc{A},\msc{C}).
\]
We note that the forgetful functor $\Fun^{\opn{E}_2}_{\msc{K}}(\msc{K}\otimes_{\msc{E}}\msc{A},\msc{C})\to \Fun^{\opn{E}_2}(\msc{K}\ot_{\msc{E}}\msc{A},\msc{C})$ is fully faithful to see that the above embedding is an equivalence onto the full subcategory $\Fun^{\opn{E}_2}_{\msc{K}}(\msc{K}\otimes_{\msc{E}}\msc{A},\msc{C})$ of braided maps over $\msc{K}$.  Hence $\msc{K}\ot_{\msc{E}}\msc{A}$ represents the functor $\Fun^{\opn{E}_2}_{\msc{E}}(\msc{A},-)$.
\par

For the claim about the pushout, if we have a $2$-diagram
\begin{equation}\label{eq:717}
\xymatrix{
\msc{E}\ar[r]\ar[d] & \msc{A}\ar[d]^f\\
\msc{K}\ar[r]_u & \msc{D}
}
\end{equation}
then the fact that the image of $\msc{E}$ generates $\msc{K}$, and lies centrally in $\msc{A}$, forces $u(\msc{K})$ to centralize $f(\msc{A})$ in $\msc{D}$.  Though the image of $\msc{E}$ in $\msc{D}$ needn't be central, and hence $\msc{D}$ needn't be a braided monoidal category over $\msc{E}$, the pair $(u,f)$ still defines an $\msc{E}$-bilinear functor from $\msc{K}\times \msc{A}$.  We then obtain a uniquely associated braided monoidal functor from $\msc{K}\ot_{\msc{E}}\msc{A}$ which splits the above diagram.
\end{proof}

The point of Proposition \ref{prop:bc} is simply to say that a given map from the product
\[
F:\msc{K}\ot_{\msc{E}}\msc{A}\to \msc{C}
\]
admits a braided monoidal structure whenever the constituent maps $f:\msc{A}\to \msc{C}$ and $u:\msc{K}\to \msc{C}$ are braided monoidal in a compatible way.  The realization of $\msc{K}\ot_{\msc{E}}\msc{A}$ as the representing object for the $2$-functor $\Fun^{\opn{E}_2}_{\msc{E}}(\msc{A},-)$, in particular, characterizes the operation $\msc{K}\ot_{\msc{E}}-$ as categorical base change along the map $\msc{E}\to \msc{K}$ (cf. \cite[Theorem 4.5.3.1]{lurieha}).

\section{Hopf structures and Tannakian reconstruction}
\label{sect:tann_reconstructor}

In this section we recall the basic principles of Tannakian reconstruction.  Through an application of Tannakian reconstruction one can translate more-or-less freely between categories and algebras.  Or, more immediately, between categories and coalgebras.

\subsection{Tannakian reconstruction}

We consider the category $\opn{AbCat}_k/\opn{Vect}$ (resp.\ $\opn{Tens}_k/\opn{Vect}$) of $k$-linear, locally finite, compactly generated, presentable abelian categories (resp.\ tensor categories) which are fibered over $\opn{Vect}$.  These are categories which come equipped with an exact and cocontinuous linear (resp.\ tensor) functor to $\opn{Vect}$, which furthermore preserves compact objects.   A morphism in $\opn{AbCat}_k/\opn{Vect}$ (resp.\ $\opn{Tens}_k/\opn{Vect}$) is a strictly commuting diagram of exact cocontinuous linear functors (resp.\ tensor functors)
\begin{equation}\label{eq:632}
\xymatrix{
\msc{A}\ar[rr]\ar[dr]_{\eta_{\msc{A}}} & & \msc{B}\ar[dl]^{\eta_{\msc{B}}}\\
 & \opn{Vect} & .
}
\end{equation}
There are no $2$-morphisms in these categories other than the identity, so that they are proper categories rather than $2$-categories.
\par

Tannakian reconstruction, or Tannaka duality, or Tannaka-Krein duality makes the following fundamental assertion.

\begin{theorem}[{\cite[\S 5]{takeuchi77}, \cite[Theorem 2.2.8, Corollary 2.3.8]{schauenburg92}}]
The functors
\[
\opn{Corep}:\opn{Coalg}_k\to \opn{AbCat}_k/\opn{Vect}\ \ \text{and}\ \ \opn{Corep}:\opn{HopfAlg}_k\to \opn{Tens}_k/\opn{Vect}
\]
are equivalence of categories.
\end{theorem}

The inverse to the $\opn{Corep}$ functor is the $\opn{Coend}$ functor,
\[
\msc{A}\mapsto \opn{Coend}_k(\eta:\msc{A}\to \opn{Vect}).
\]
Explicitly, we first have the $\End$ functor, which takes $\msc{A}$ to the topological algebra
\[
\opn{End}_k(\eta:\msc{A}\to \opn{Vect})=\left\{\begin{array}{c}
\text{The algebra of natural linear endomorphisms}\\
\text{$a$ of }\eta,\ a_V:\eta(V)\to \eta(V)\ \text{for all $V$ in }\msc{A}
\end{array}\right\}.
\]
This algebra acts naturally on $\eta(V)$ for each $V$ in $\msc{A}$, and the closed ideals in $\End_k(\eta)$ are the annihilators of $\eta(V)$ for compact $V$.  The coend for $\eta$ is obtained as the continuous dual
\[
\opn{Coend}_k(\eta)=\End_k(\eta)^\ast
\]
\[
=\{\alpha:\opn{End}_k(\eta)\to k\ \text{such that }\alpha \text{ vanishes on a closed ideal $I$ in }\opn{End}_k(\eta)\}
\]
\cite[\S\ 1.10]{egno15}.  We refer to $\opn{Coend}_k(\eta)$ as the representing coalgebra for $\msc{A}$.
\par

One should recall our notions of surjective functors and embeddings, from Section \ref{sect:surj_embed}.

\begin{theorem}[{\cite[Lemmas 2.2.12, 2.2.13]{schauenburg92}}]
A functor $F:\msc{A}\to \msc{B}$ between objects in $\opn{AbCat}_k/\opn{Vect}$ is surjective if and only the corresponding map of representing coalgebras $f:A\to B$ is surjective.  Similarly, $F$ is an embedding if and only if $f$ is injective.
\end{theorem}

\subsection{Reconstruction and normality}

For a pointed $\msc{A}$-module category $\msc{M}$ equipped with a fiber functor
\[
\{\ \msc{A}\to\opn{Vect},\ \msc{M}\to \opn{Vect}\ \},
\]
one reconstructs $\msc{A}$ as corepresentations for a Hopf algebra $A$, and $\msc{M}$ as corepresentations for an $A$-module coalgebra $C$ which is equipped with a map of module coalgebras $\pi:A\to C$.  To elaborate, the product $C\ot A$ lives as an object in $\opn{Corep}(C)\cong \msc{M}$ and the $A$-action on $C$ is extracted from the $C$-coaction on $C\ot A$ via the composite
\[
C\ot A\to (C\ot A)\ot C\overset{\epsilon\ot id_C}\to C.
\]

\begin{lemma}[{cf.\ \cite[Lemma 7.4]{negron}}]\label{lem:1025}
Let $A$ be a Hopf algebra and $\pi:A\to C$ be a map of right $A$-module coalgebras.  The corestriction functor $F:\opn{Corep}(A)\to \opn{Corep}(C)$ is normal, in the sense of Definiton \ref{def:normal}, if and only if the right coinvariants $A^C$ are a Hopf subalgebra in $A$.  Furthermore, in this case $\opn{ker}(F)=\opn{Corep}(A^C)$, and the left and right coinvariants in $A$ agree $A^C={^CA}$.
\end{lemma}

\begin{proof}
Suppose first that $A^C$ is a Hopf subalgebra in $A$.  For an $A$-comodule $V$ with coaction $\rho:V\to V\ot A$, let $\bar{\rho}=(1\ot \pi)\rho$ denote the corresponding $C$-coaction.  For any $A$-corepresentation $V$ the $C$-coinvariants $V^C$ have restricted coaction
\[
\rho|_{V^C}:V^C\to V\ot A
\]
landing in $V\ot A^C$.  Since $A^C$ is a Hopf subalgebra in $A$, coassociativity gives
\[
(\rho_V\ot 1)\rho|_{V^C}:V^C\to (V\ot A^C)\ot A^C
\]
and thus $(\bar{\rho}\ot 1)\rho$ has image in $(V\ot k)\ot A^C$.  It follows that $\rho|_{V^C}$ has image in $V^C\ot A^C$, and hence that $V^C$ is an $A$-subcomodule in $V$ which lies in the subcategory $\opn{Corep}(A^C)\subseteq \opn{Corep}(A)$.  We therefore see that $F:\opn{Corep}(A)\to \opn{Corep}(C)$ is normal with kernel $\opn{Corep}(A^C)$.
\par

Conversely, if $F$ is normal then $A^C\subseteq A$ is an $A$-subcorepresentation.  But we have, generally, that $A^C$ is a left $A$-comodule algebra,
\[
\Delta|_{A^C}:A^C\to A\ot A^C.
\]
So the fact that $A^C$ is a subcomodule in the regular corepresentation implies that $A^C\subseteq A$ is a sub-bialgebra.  By the argument given above we also have $\opn{ker}(F)=\opn{Corep}(A^C)$.  Since this subcategory is a tensor subcategory by hypothesis, $A^C$ is a Hopf subalgebra in $A$.  Finally, agreement of left and right coinvariants follows by \cite[Lemma 7.4]{negron}.
\end{proof}

\section{Practical base change results}
\label{sect:basechange_results}

Via an application of Tannakian reconstruction, a number of classical faithful flatness results for Hopf extensions can be reframed as base change calculations for categories of (co)representations.  We enumerate a number of useful instances in this section.

\subsection{Tensoring with categories of modules}

We have the following result of Douglas, Schommer-Pries, and Snyder.

\begin{proposition}[{\cite{douglasspsnyder19,lurieha}}]\label{prop:mods}
Let $\msc{E}$ be a tensor category, $R$ be an algebra in $\msc{E}$, and $R\text{\rm-mod}_{\msc{E}}$ be the category of left $R$-modules in $\msc{E}$.  We consider $R\text{\rm-mod}_{\msc{E}}$ as a right $\msc{E}$-module category.  Then for any left $\msc{E}$-module category $\msc{M}$ the $\msc{E}$-bilinear map
\[
\opn{univ}:R\text{\rm-mod}_{\msc{E}}\times \msc{M}\to R\text{\rm-mod}_{\msc{M}},\ \ \opn{univ}(X,M):=X\ot M,
\]
induces an equivalence from the corresponding product over $\msc{E}$
\[
(R\text{\rm-mod}_\msc{E})\ot_{\msc{E}}\msc{M}\overset{\sim}\to R\text{\rm-mod}_{\msc{M}}.
\]
\end{proposition}

The proof is exactly as in \cite[proof of Theorem 3.3(2)]{douglasspsnyder19}.  For some details, if we are given an appropriately cocontinuous $\msc{E}$-bilinear functor
\[
T:R\text{-mod}_{\msc{E}}\times \msc{M}\to \msc{C}
\]
the corresponding functor $t:R\text{-mod}_{\msc{M}}\to \msc{C}$ is determined by the formula
\[
\xymatrixcolsep{13mm}
t(Y):=\opn{colim}\left(
\xymatrix{
T(R,R\ot Y)\ar@<2pt>[r]^(.55){\opn{act}'_R}\ar@<-2pt>[r]_(.55){\opn{act}_Y} & T(R,Y)
}
\right),
\]
where $Y$ is any given $R$-module in $\msc{M}$ and $\opn{act}_\star$ (resp.\ $\opn{act}'_\star$) denotes the left (resp.\ right) action map on $Y$ (resp.\ $R$).

\begin{remark}
Proposition \ref{prop:mods} can alternatively be deduced from \cite[Theorem 4.8.4.6]{lurieha} and the embedding $\opn{Nerve}:\opn{Pr}\to \opn{Pr}_\infty$.  This approach is exceedingly inefficient however.
\end{remark}

\subsection{Base change for finite representation categories}

The following is a straightforward application of results from Takeuchi \cite{takeuchi79} and Nichols-Zoeller \cite{nicholszoeller89}.

\begin{lemma}\label{lem:fin_basechange}
\begin{enumerate}
\item Suppose $A$ is a finite-dimensional Hopf algebra and $B\subseteq A$ is a right coideal subalgebra.  Suppose also that the restriction functor $F:\Rep(A)\to \Rep(B)$ is normal, and take $C=k\ot_B A$.  Then $C$ is a Hopf algebra quotient of $A$, $\opn{ker}(F)=\Rep(C)$, and the restriction functor induces an equivalence of right $\Rep(A)$-module categories
\[
\opn{Vect}\ot_{\Rep(C)}\Rep(A)\overset{\sim}\to \Rep(B).
\]
\item Suppose $A$ is a finite-dimensional Hopf algebra and $A\to C$ is a quotient right $A$-module coalgebra.  Suppose also that the corestriction functor $F:\opn{Corep}(A)\to \opn{Corep}(C)$ is normal.  Then $B=A^{C}$ is a Hopf subalgebra in $A$, $\opn{ker}(F)=\opn{Corep}(B)$, and the restriction functor induces an equivalence of right $\opn{Corep}(A)$-module categories
\[
\opn{Vect}\ot_{\opn{Corep}(B)}\opn{Corep}(A)\overset{\sim}\to \opn{Corep}(C).
\]
\end{enumerate}
\end{lemma}

\begin{proof}
Statement (1) is deduced from statement (2) via duality.  So we need only establish (2).
\par

(2) The fact that $B=A^C$ is a Hopf algebra is covered in Lemma \ref{lem:1025}, and by Nichols-Zoeller \cite{nicholszoeller89} $A$ is a free $B$-module on the left and right.  It follows by Takeuchi \cite[Theorem 1]{takeuchi79} that the functor
\begin{equation}\label{eq:1075}
k\ot_B-: B\text{-mod}_{\opn{Corep}(A)}\overset{\sim}\to \opn{Corep}(C)
\end{equation}
is a linear equivalence.  For any $B$-module $M$ and $A$-corepresentation $V$, the natural identification
\[
(k\ot_B M)\ot V\overset{\sim}\to k\ot_B(M\ot V)
\]
realizes this equivalence as an equivalence of right $\opn{Corep}(A)$-module categories.
\par

Now, the fundamental theorem of Hopf modules provides an equivalence of right $\opn{Corep}(B)$-module categories $\opn{Vect}\cong B\text{-mod}_{\opn{Corep}(B)}$, so that we have an identification of $\opn{Corep}(A)$-module categories
\[
\opn{Vect}\ot_{\opn{Corep}(B)}\opn{Corep}(A)=B\text{-mod}_{\opn{Corep}(A)}
\]
via Proposition \ref{prop:mods}.  Hence the equivalence \eqref{eq:1075} provides the desired calculation of the fiber, $\opn{Vect}\ot_{\opn{Corep}(B)}\opn{Corep}(A)\overset{\sim}\to \opn{Corep}(C)$.
\end{proof}

\subsection{Results for infinite categories}

\begin{theorem}[{Takeuchi's theorem \cite{takeuchi79}}]\label{thm:takeuchi}
Suppose $\msc{E}\to \msc{A}\to \msc{M}$ is a mixed exact sequence of tensor/module categories, in the sense of Definition \ref{def:531}.  Suppose additionally that $\msc{M}$ admits a fiber functor.  Then the induced map
\[
\opn{Vect}\ot_{\msc{E}}\msc{A}\to \msc{M}
\]
is an equivalence of right $\msc{A}$-module categories.
\end{theorem}

\begin{proof}
First note that the diagram \eqref{eq:ker}, with $\opn{ker}(u_\msc{M})=\msc{E}$ in this case, provides an $\msc{E}$-bilinear functor from the cartesian product $\opn{Vect}\times\msc{A}$ and subsequent map from the relative product $\opn{Vect}\ot_\msc{E}\msc{A}$.  We are claiming that this map $\opn{Vect}\ot_\msc{E}\msc{A}\to\msc{M}$ is an equivalence.
\par

The existence of a fiber functor for $\msc{M}$ allows us to reconstruct the sequence $\msc{E}\to \msc{A}\to \msc{M}$ from a corresponding sequence
\[
k\to \O\to A\to C\to k,
\]
via corepresentations.  Here $\O$ and $A$ are Hopf algebras, $\O\to A$ is a Hopf inclusion, and $A\to C$ is a surjection of right $A$-module coalgebras.  Furthermore, Lemma \ref{lem:1025} provides an identification of $\O$ with the coinvariant subalgebra $\O=A^C$ in this case.
\par

Under such an identification, the right adjoint $\opn{ind}:\msc{M}\to \msc{A}$ is given explicitly by the cotensor $\opn{ind}=-\square_{C}A$ \cite[Proposition 6]{doi81}.  Exactness of induction implies that $A$ is coflat over $C$, and faithfulness implies that $A$ is furthermore faithfully coflat over $C$.  Hence Takeuchi \cite[Theorem 2]{takeuchi79} tells us that the module category functor
\[
k\ot_{\O}-:\opn{Vect}\ot_{\msc{E}}\msc{A}=\O\text{-mod}_{\opn{Corep}(A)}\to \opn{Corep}(C)=\msc{M}
\]
is an equivalence.
\end{proof}

\begin{remark}
We expect that Theorem \ref{thm:takeuchi} still holds in the absence of a fiber functor for $\msc{M}$.
\end{remark}

\subsection{Result for braided tensor categories}

We consider a diagram of braided tensor functors
\begin{equation}\label{eq:potdig}
\xymatrix{
\msc{E}\ar[rr]^{F}\ar[d]_{G'} & & \msc{A}\ar[d]^G\\
\msc{K}\ar[rr]^{F'} & & \msc{C},
}
\end{equation}
where $F$ and $F'$ are M\"uger central tensor embeddings, and $G$ and $G'$ are surjective.
\par

We provide sufficient conditions under which such a diagram induces an equivalence of braided monoidal categories
\[
\msc{K}\ot_{\msc{E}}\msc{A}\overset{\sim}\to \msc{C}.
\]

\begin{proposition}\label{prop:braided_bc}
Suppose we have a diagram as in \eqref{eq:potdig}, and that $\msc{K}$ is a fusion category.  Suppose additionally that:
\begin{enumerate}
\item[(a)] The right adjoint $\opn{ind}_{\msc{C}}^{\msc{A}}:\msc{C}\to \msc{A}$ to $G$ is faithfully exact.
\item[(b)] The natural map $F\opn{ind}_{\msc{K}}^{\msc{E}}(\1)\to \opn{ind}_{\msc{C}}^{\msc{A}}(\1)$ is an isomorphism.
\item[(c)] $\msc{C}$ admits a fiber functor.
\end{enumerate}
Then the induced map $\msc{K}\ot_{\msc{E}}\msc{A}\overset{\sim}\to \msc{C}$ is an equivalence of braided monoidal categories.
\end{proposition}

\begin{proof}
We recall that the induced functor $\msc{K}\ot_{\msc{E}}\msc{A}\overset{\sim}\to \msc{C}$ is naturally braided monoidal, by Proposition \ref{prop:bc}.  So we need only prove that it is an equivalence of plain categories.
\par

Supposing $\msc{C}$ admits a fiber functor, the diagram \eqref{eq:potdig} is reproduced (via corepresentations) from a corresponding diagram of Hopf algebra maps
\[
\xymatrix{
E\ar[rr]^{f}\ar[d]_{g'} & & A\ar[d]^g\\
K\ar[rr]^{f'} & & C.
}
\]
Here $f$ and $f'$ are injective and $g$ and $g'$ are surjective, by hypothesis.
\par

We note that the induction functor $\opn{ind}^{\msc{E}}_{\msc{K}}$ is both faithful and exact, since $\msc{K}$ is semisimple and $G'$ is surjective.  Exactness follows from the fact that any additive functor from a semisimple category is exact.  Faithfulness follows from the fact that, in this case, each simple in $\msc{K}$ necessarily appears as a summand in $G(V)$ for some $V$ in $\msc{E}$.
\par

We recall that the value at the unit $\opn{ind}_{\msc{K}}^{\msc{E}}(\1)={^KE}$ is naturally an algebra object in $\msc{E}$.  The algebra structure on ${^KE}$ is the expected one, i.e.\ the one inherited from $E$, and is abstractly obtained as the adjoint to the counit maps $G'\opn{ind}_{\msc{K}}^{\msc{E}}(\1)\ot G'\opn{ind}_{\msc{K}}^{\msc{E}}(\1)\to \1\ot\1=\1$.
Let's call this algebra $\O={^KE}$.
\par

Given faithful exactness of $\opn{ind}_{\msc{K}}^{\msc{E}}$, Takeuchi's theorem \cite[Theorem 2]{takeuchi79} provides an equivalence of $\msc{E}$-module categories
\[
(\1\ot_{G'\O}-)\circ G':\O\text{-mod}_\msc{E}\overset{\sim}\to \msc{K},
\]
and from (b) we understand that the inclusion
\[
\O={^KE}=\opn{ind}_{\msc{K}}^{\msc{E}}(\1)\to {^CA}=\opn{ind}_{\msc{C}}^{\msc{A}}(\1)
\]
is an equality.  We recall finally that the induction functor $\opn{ind}_{\msc{C}}^{\msc{A}}$ is faithfully exact, by hypothesis, and apply Takeuchi's theorem \cite[Theorem 2]{takeuchi79} again to see that the natural map
\[
(\1\ot_{G\O}-)\circ G:\msc{K}\ot_\msc{E}\msc{A}= \O\text{-mod}_\msc{A}\overset{\sim}\to \msc{C}.
\]
is an equivalence of $\msc{A}$-module categories.
\end{proof}

\begin{remark}
We expect that conditions (a) and (b) in Proposition \ref{prop:braided_bc} are necessary, at least when $\msc{K}$ is fusion.  Condition (c) should be superfluous.
\end{remark}

\subsection{Preservation of finiteness and rigidity}

The information relayed above helps us to understand the base change $\msc{K}\ot_{\msc{E}}\msc{A}$ as a linear monoidal category.  We recall one basic result which concerns rigidity of the base change.  The following appears in \cite{dgno10}.  See in particular \cite[Proof of Theorem 4.18]{dgno10}.

\begin{lemma}[\cite{dgno10}]\label{lem:de_ftc}
Let $\msc{A}$ be a finite tensor category, and $\msc{E}\to \msc{A}$ be a central embedding from a finite Tannakian category $\msc{E}$.  Then the fiber $\opn{Vect}\ot_{\msc{E}}\msc{A}$ is a finite tensor category.  In particular, it is rigid.
\end{lemma}

\begin{proof}
In this case $\msc{E}=\opn{Rep}(\Sigma)$ for a finite group $\Sigma$.  So we are in the setting of \cite[\S\ 4.2]{dgno10}.
\par

Under the structure map $F:\msc{A}\to \opn{Vect}\ot_{\msc{E}}\msc{A}$, every compact object $X$ in $\opn{Vect}\ot_{\msc{E}}\msc{A}$ appears as a summand of the image $F(X')$ of a rigid object in $\msc{A}$ \cite[Lemma 4.6 (iii)]{dgno10}.  Since the subcategory of rigid objects in a linear monoidal category is closed under taking summands \cite[Ex 9.10.4]{egno15}, and since tensor functor preserve rigid objects \cite[Ex 2.10.6]{egno15}, we conclude that all compact objects in $\opn{Vect}\ot_{\msc{E}}\msc{A}$ are rigid.
\par

The fact that $F$ has an exact right adjoint \cite[Lemma 4.6 (i)]{dgno10} implies that $F$ preserves projective objects, and the fact that all objects in $\opn{Vect}\ot_{\msc{E}}\msc{A}$ appear as a summand of some $F(X')$ implies that the image $F(P)$ of a compact projective generator $P$ in $\msc{A}$ provides a compact projective generator for $\opn{Vect}\ot\msc{A}$.  Since $F(P)$ has finite length it follows that $\opn{Vect}\ot_{\msc{E}}\msc{A}$ has finitely many simples as well, and is hence a finite tensor category.
\end{proof}

\section{Quantum group representations}
\label{sect:quantum_groups}

We now change gears from our study of generic nonsense for module categories to a rather focused study of quantum groups, and quantum group representations.  As outlined in the introduction, our primary goal is to provide a completely uniform construction of the small quantum group via a fibering of the big representation category $\Rep(G_q)$ over its Tannakian center.  We begin our analysis by recalling the definition of the category $\Rep(G_q)$ itself, along with its natural ribbon tensor structure.

As before, $k$ is an algebraically closed field of characteristic $0$.  We follow the presentation of \cite[\S\ 3]{negron} specifically and \cite{lusztig93} generally.

\subsection{Data for an algebraic group}
\label{sect:c_data}

We fix a semisimple algebraic group $G$ over $k$, along with a choice of maximal torus $T$ in $G$.  Throughout this work $\mfk{g}$ denotes the Lie algebra for $G$ and $\msc{W}$ denotes the Weyl group for $G$.  We call $G$ almost-simple if its Lie algebra is simple, and we recall that any semisimple algebraic group decomposes uniquely into almost-simple factors $G=G_1\times \dots \times G_t$.  We define the \emph{lacing number} for $G$ to be the least common multiple of the lacing numbers for its almost-simple factors.
\par

We let $X=\Hom_{\opn{GrpSch}_k}(T,\mathbb{G}_m)$ denote the character lattice for $G$, $\Phi\subseteq X$ denote the root system in $X$, and $Q$ denote the root lattice $Q=\mathbb{Z}\cdot \Phi$.  We consider the unique normalization of the Killing form
\[
(-,-):X\times X\to \mathbb{Q}
\]
so that $(\alpha,\alpha)=2$ at any short root in $\Phi$, and record the relative root lengths $d_\gamma=(\gamma,\gamma)/2$.  We now obtain the following expression of the weight lattice
\[
P=\{\lambda\in \mathbb{Q}\ot_{\mathbb{Z}}X:(\lambda,\gamma)\in d_\gamma\cdot \mathbb{Z}\ \text{at each }\gamma\in \Phi\}.
\]
\par

We fix a base $\Delta$ of simple roots in $\Phi$ and subsequent splitting of $\Phi$ into its positive and negative halves $\Phi=\Phi^+\amalg\Phi^-$.  Dominant weights in $X$ are now defined as
\[
X^+=\{\lambda\in X:(\lambda,\alpha)\geq 0\ \text{at all }\alpha\in \Delta\}.
\]

\subsection{Data for a quantum group}
\label{sect:q_data}

Consider a semisimple algebraic group $G$ with fixed data as above.  

\begin{definition}
A quantum paramater for $G$ is a symmetric bilinear form on the weight lattice $q:P\times P \to k^{\times}$ which is invariant under the action of the Weyl group, and which satisfies $q(\lambda,\mu)=1$ at all weights $\lambda,\mu\in P$ with $(\lambda,\mu)=0$.
\end{definition}

\begin{hypotheses}
All quantum parameters considered in this text are assumed to be torsion, i.e.\ to take values in the roots of unity $\opn{tors}(k^\times)$.
\end{hypotheses}

For any root $\gamma\in \Phi$ we define
\[
l_\gamma:= \opn{ord}(q^2(\gamma,-))=\opn{ord}(q(\gamma,\gamma)),
\]
and note that each $l_\gamma$ is a (well-defined) positive integer since our form $q$ is torsion.  
\par

For each simple root $\alpha$ we have the corresponding fundamental weight $\omega_{\alpha}\in P$, and we define the \emph{scalar parameters} for $G$ at $q$ as
\[
q_{\alpha}:= q(\alpha,\omega_{\alpha}),\ \ \alpha\in \Delta.
\]
These scalars extend to a unique $\msc{W}$-invariant function on $\Phi$ from which we define scalar parameters $q_\gamma$ at arbitrary $\gamma\in \Phi$.  We note that, at any root $\gamma$, we have $l_\gamma=\opn{ord}(q^2_\gamma)$.

\begin{remark}
For each almost-simple factor $H\subseteq G$ we have the uniquely associated root of unity 
\[
q_H:= q_{\alpha}\ \text{for any short simple root }\alpha\in\Delta_H.
\]
These $q_H$ are the parameters which one employs in traditional presentations of the quantum group, as in \cite{lusztig90,lusztig90II,lusztig93}.  Our form-forward framing is adapted from \cite{gaitsgory21}.
\end{remark}

\subsection{Big quantum group representations}

We have the generic quantum enveloping algebra $U^{\opn{gen}}_v=U^{\opn{gen}}_v(\mfk{g})$ over the fraction field $\mathbb{Q}(v_\alpha:\alpha\in \Delta)$ of the polynomial ring
\[
\mathcal{O}=\mathbb{Z}[v_\alpha:\alpha\in \Delta]/\big(v_\beta=v_\alpha^{d_\beta/d_\alpha}\ \text{whenever}\ (\alpha,\beta)\neq 0\ \text{and}\ d_\beta\geq d_\alpha\big),
\]
and Lusztig's integral Hopf subalgebra $U_v$ over $\mathcal{O}$. The algebra $U_v$ has generators $E_{\alpha}^{(n)}$, $F_{\alpha}^{(n)}$ and $K_{\alpha}$, where $\alpha$ runs over all simple roots, $n$ runs over all nonnegative integers, and the divided powers satisfy
\[
E_\alpha^n=[n]_{v_\alpha}!\cdot E^{(n)}_\alpha\ \ \text{and}\ \ F_\alpha^n=[n]_{v_\alpha}!\cdot F^{(n)}_\alpha
\]
at all $n$ \cite{lusztig90II,lusztig93}.  We also consider the redundant toral generators
\[
\binom{K_\alpha;0}{l_\alpha}:=\frac{K_\alpha-K_\alpha^{-1}}{v_\alpha-v_\alpha^{-1}}=[E_\alpha, F_\alpha].
\]
We let $U^+_v$ and $U^-_v$ denote the subalgebras in $U_v$ generated by the elements $E_{\alpha}^{(n)}$ and $F_{\alpha}^{(n)}$, respectively.
\par

From the above quantum group data we define the usual quantum enveloping algebra $U_q=U_q(\mfk{g})$ by specializing Lusztig's integral algebra along the evaluation map $ev_q:\mathcal{O}\to k$, $v_\alpha\mapsto q_\alpha$, and we have the associated tensor category of quantum group representations
\[
\Rep(G_q)= \text{Integrable, character graded $U_q$-modules}.
\]
\par

To elaborate, an object $V$ in $\Rep(G_q)$ is a vector space which is graded by the character lattice $V=\oplus_{\lambda \in X} V_{\lambda}$ and which comes equipped with linear endomorphisms for each of the generators $E_{\alpha}^{(n)}$, $F_{\alpha}^{(n)}$, $K_{\alpha}$ in $U_q$.  The homogeneous subspaces $V_\lambda$ are eigenspaces for the actions of the toral generators in $U_q$, and the operators $E^{(n)}_\alpha$ and $F^{(n)}_\alpha$ shift degrees on $V$ by $n\cdot\alpha$ and $-n\cdot\alpha$ respectively.
\par

The toral elements act specifically via the eigenvalues
\[
K_\alpha\cdot v= q(\alpha,\lambda)\!\ v\ \ \text{and}\ \ \binom{K_\alpha;0}{l_\alpha}\cdot v=\binom{\langle \alpha,\lambda\rangle}{l_\alpha}_{q_\alpha} v,\ \ \text{for }v\in V_\lambda,
\]
and the endomorphisms $E_\alpha^{(n)},F_\alpha^{(n)}:V\to V$ satisfy the defining relations for the quantum enveloping algebra $U_q$ \cite[Theorem 6.6]{lusztig90II}.  Finally, any $G_q$-representation is required to be the sum of its finite-dimensional subrepresentations.
\par

Our category of quantum group representations is the same as the category of unital, integrable modules for the modified quantum algebra $\dot{U}_q$ from \cite[Ch 31]{lusztig93}, $\Rep(G_q)=\Rep(\dot{U}_q)$.

\subsection{Bases}

The remainder of the section is dedicated to a presentation of some of the finer details for $U_q$ and its associated category of quantum group representations.  We begin by recalling Lusztig's bases for the positive and negative subalgebras in $U_q$.

Having fixed an enumeration of the simple roots $\Delta=\{\alpha_1,\dots,\alpha_r\}$, and a reduced expression for the longest element in the Weyl group $w_0=s_{i_t}\dots s_{i_1}$, we obtain an enumeration of the positive roots as
\[
\Phi^+=\{\gamma_1,\dots,\gamma_t\},\ \ \gamma_j=s_{i_t}\dots s_{i_{j+1}}(\alpha_j).
\]
We recall that the simple reflections $s_i$ in $\msc{W}$ lift to braid group operators $T_i:U_q\to U_q$, as in \cite[Theorem 3.1]{lusztig90II}, and we have the corresponding root vectors
\[
E_{\gamma_j}=T_{i_t}\dots T_{i_{j+1}}(E_j)\ \in\ U_q^+,\ \ F_{\gamma_j}=T_{i_t}\dots T_{i_{j+1}}(T_j)\ \in\ U_q^-.
\]
Via ordered monomial in the above vectors we obtain bases
\[
\{E_{\gamma_1}^{(m_1)}\dots E_{\gamma_t}^{(m_t)}:m:\Phi^+\to \mathbb{Z}_{\geq 0}\}\ \ \text{and}\ \ \{F_{\gamma_1}^{(m_1)}\dots F_{\gamma_t}^{(m_t)}:m:\Phi^+\to \mathbb{Z}_{\geq 0}\}
\]
for $U_q^+$ and $U^-_q$ respectively \cite[Proposition 41.1.4]{lusztig93} \cite[Proposition 4.2]{lusztig90II}.

\subsection{Lusztig's Hopf pairing}

In \cite[Proposition 1.2.3]{lusztig93} Lusztig defines a non-degenerate Hopf pairing \begin{equation}\label{eq:pairing}
\langle-,-\rangle:U^+_v\ot U^-_v\to \mathbb{Q}(v_\alpha:\alpha\in\Delta).
\end{equation}

\begin{lemma}[{\cite[Lemma 1.4.4]{lusztig93}}]
For the generating root vectors $E^{(n)}_\alpha\in U^+_v$ and $F^{(n)}_\beta\in U^-_v$, the above pairing satisfies
\begin{equation}\label{eq:638}
\langle E_\alpha^{(n)},F_\beta^{(m)}\rangle=\delta_{n,m}\delta_{\alpha,\beta}\left(v_\alpha^{n(n+1)/2}(v_\alpha-v_{\alpha}^{-1})^{-n}([n]_{v_\alpha}!)^{-1}\right).
\end{equation}
\end{lemma}

\begin{proof}
The form $\langle-,-\rangle$ is homogenous of degree $0$ with respect to the root lattice gradings on $U^\pm_v$.  Hence the above pairing vanishes whenever $\alpha$ and $\beta$ are distinct, and also when $n$ and $m$ are distinct.  So we reduce to the case where $\beta=\alpha$ and $m=n$, at which point the expression \eqref{eq:638} is covered in \cite[Lemma 1.4.4]{lusztig93}.
\end{proof}

\begin{lemma}\label{lem:647}
For functions $n,m:\Phi^+\to \mathbb{Z}_{\geq 0}$,
\[
\langle E_{\gamma_1}^{(n_1)}\dots E_{\gamma_t}^{(n_t)},\ F_{\gamma_1}^{(m_1)}\dots F_{\gamma_t}^{(m_t)}\rangle =\delta_{n,m}\prod_{\gamma\in \Phi^+}\left(v_\gamma^{n(n+1)/2}(v_\gamma-v_\gamma^{-1})^{-n}([n]_{v_\gamma}!)^{-1}\right).
\]
\end{lemma}

\begin{proof}
Follows from \cite[Proposition 38.2.3]{lusztig93}, applied to the elements $L(h,c,n,1)$.  (See also \cite[Proposition 4.1.3]{lusztig93}.)
\end{proof}

\subsection{The $R$-matrix}
\label{sect:R}

From the pairing \eqref{eq:pairing} Lusztig produces an $R$-matrix for the category of quantum group representation $\Rep(G_q)$ \cite[Theorem 4.1.2, Ch. 32]{lusztig93}.  This $R$-matrix then provides a braiding on the category of quantum group representations, which is given as
\[
c_{V,W}:V\ot W\overset{\cong}\to W\ot V,\ \ c_{V,W}(v\ot w):=R_{21}(w\ot v).
\]
This $R$-matrix has the specific form
\begin{equation}\label{eq:R}
R=(\sum_{n:\Phi^+\to \mathbb{Z}_{\geq 0}}\opn{coeff}(n,q)\cdot E_{\gamma_1}^{(n_1)}\dots E_{\gamma_t}^{(n_t)}\ot F_{\gamma_1}^{(n_1)}\dots F_{\gamma_t}^{(n_t)})\cdot \Omega,
\end{equation}
where $\Omega$ acts on a product of representation $V\ot W$ according to the formula $\Omega(v\ot w)=q^{-1}(\deg v,\deg w) v\ot w$ \cite[Theorem 4.1.2, Proposition 38.2.3]{lusztig93}.
\par

The coefficients $\opn{coeff}(n,q)$ in the expression \eqref{eq:R} are polynomials in $q$ which are determined by the pairings from Lemma \ref{lem:647}, and they appear explicitly as
\[
\opn{coeff}(n,q)=\hspace{9.7cm}
\]
\[
(-1)^{(\sum_{\gamma}n_\gamma \opn{ht}(\gamma))}q(\sum_{\gamma\in \Phi^+} n_\gamma \gamma, \sum_{\alpha\in\Delta}\omega_{\alpha})\prod_{\gamma\in \Phi^+}\left(q_{\gamma}^{-n_\gamma(n_\gamma+1)/2}(q_{\gamma}-q_{\gamma}^{-1})^{n_\gamma}[n_\gamma]_{q_{\gamma}}!\right)
\]
\cite[Theorem 4.1.2]{lusztig93}.
\par

We isolate the factor $[n_\gamma]_{q_\gamma}!$ in the above expression to see that a given coefficient $\opn{coeff}(n,q)$ vanishes whenever $n_\gamma\geq l_\gamma$ at any root $\gamma$.  It follows that the infinite sum \eqref{eq:R} has only finitely many nonzero terms.

\subsection{The $R$-matrix and the completed enveloping algebra}

The element $R$ does not exist as an element in the (second tensor power) of the quantum enveloping algebra.  However, it can be located in the natural completion of the modified algebra
\[
\hat{U}_q:=\varprojlim_V \dot{U}_q/\opn{Ann}(V),\ \ V\ \text{finite-dimensional in }\Rep(G_q),
\]
\[
=\opn{End}_k(\ forget:\Rep(G_q)\to \opn{Vect}\ ).
\]
\par

Via its expression as the above limit, $\hat{U}_q$ is naturally a topological Hopf algebra. It is furthermore topologically generated by the standard root vectors $E_{\alpha}$ and $F_{\alpha}$, along with the characters on the torus, where each character $\xi:X\to k^{\times}$ lives in $\hat{U}_q$ as the convergent series $\xi=\sum_{\lambda \in X}\xi(\lambda)\cdot 1_{\lambda}$.  Our analysis from \ref{sect:R} implies the following.

\begin{lemma}\label{lem:R_small}
The $R$-matrix for $\Rep(G_q)$ lies in the second tensor power of the subalgebra $A_R$ in $\hat{U}_q$ which is generated by the root vectors $E_\gamma$ and $F_\gamma$ with $l_\gamma>0$, and the characters $\xi:X\to k^\times$ which vanish on the finite index sublattice $\opn{rad}(q)\subseteq X$.
\end{lemma}

\subsection{The ribbon structure}

For $\rho=\frac{1}{2}\sum_{\gamma\in\Phi^+}\gamma$, the associated character $K_{2\rho}$ gives $\Rep(G_q)$ the structure of a pivotal tensor category.  This pivotal structure produces a corresponding twist $\theta= u^{-1}K_{2\rho}$, where $u$ is the Drinfeld morphism \cite[\S\ 8.9]{egno15}.  Hence $\Rep(G_q)$ is naturally a balanced tensor category.

\begin{proposition}
The twist $\theta$, defined as above, satisfies $\theta^\ast_V=\theta_{V^\ast}$ at all finite-dimensional $G_q$-representations $V$.  This is to say, the twist $\theta$ endows $\Rep(G_q)$ with the structure of a ribbon tensor category.
\end{proposition}

The result follows from the fact that the associated element $\theta\in \hat{U}_q$ satisfies $S(\theta)=\theta$ \cite[\S\ 1]{sawin06}, so that $\hat{U}_q$ is a ribbon Hopf algebra \cite[Proposition 8.11.2]{egno15}.

\subsection{Useful restrictions on $q$}

When considering examples it is helpful to institute \emph{some} restrictions on the parameter $q$, and hence on the nature of the $R$-matrix for $\Rep(G_q)$.

\begin{definition}\label{def:max_nondegen}
We call a quantum parameter $q$ maximally non-degenerate if its radical $\opn{rad}(q)$, i.e.\ the lattice of all elements $\lambda$ in $P$ with $q(\lambda,-)= 1$, lies in the root sublattice $Q\subseteq P$.
\end{definition}

After choosing an embedding $\bar{\mathbb{Q}}\subseteq k$ if necessary, the prototypical example of a maximally non-degenerate parameter $q$ is one of the form
\begin{equation}\label{eq:1375}
q=\exp\left(\frac{\pi i (-,-)}{l}\right):P\times P\to k^\times\ \ \text{or}\ \ \exp\left(\frac{2\pi i (-,-)}{l}\right):P\times P\to k^\times,
\end{equation}
where $l$ is an integer with all $d_\alpha\mid l$.  Indeed, one sees directly in this case that any element $\lambda\in P$ which is in the radical of $q$ necessarily satisfies $(\lambda,\omega_\alpha)\in d_\alpha\mathbb{Z}$ at all simple $\alpha$, and hence lies in the root lattice.  So the quantum parameters \eqref{eq:1375} are in fact maximally non-degenerate, in the above sense.
\par

In general, one can simply think of a maximally non-degenerate quantum parameter $q$ as one which ``acts like" a direct exponentiation of the Killing form as in \eqref{eq:1375}.

\begin{remark}
In the simply-laced case the divisibility condition $d_\alpha\mid l$ is vacuous, so that $l$ can be arbitrary in the expressions \eqref{eq:1375}.
\end{remark}

\section{Quasi-classical representations and quantum Frobenius}

We recall Lusztig's quantum Frobenius functor
\[
\opn{Fr}:\Rep(G^\ast_{\varepsilon})\to \Rep(G_q)
\]
and its relationship to the subcategory of quasi-classical representations in $\Rep(G_q)$.

\subsection{Quasi-classical representations}

Following \cite{lusztig93}, we take
\[
X^\ast=\{\lambda\in X:q^2(\lambda,\alpha)=1\ \text{at all simple }\alpha\}.
\]
In the simply-connected case $P^\ast=\mathbb{Z}\cdot \{l_\alpha\omega_\alpha:\alpha\in \Delta\}$, where each $\omega_\alpha$ is the fundamental weight associated to $\alpha$, and we recover $X^\ast$ as the intersection $X^\ast=X\cap P^\ast$.  We also consider the ``$l$-dualized" root lattice
\[
lQ=\mathbb{Z}\cdot\{l_\alpha \alpha:\alpha\in \Delta\}.
\]
One can check that $lQ\subseteq X^\ast$.

The following is covered in \cite[Proposition 35.3.2]{lusztig93} (see also \cite[Proposition 4.3]{negron}).

\begin{lemma}[{\cite[Proposition 35.3.2]{lusztig93}}]\label{lem:674}
A simple representation $L(\lambda)$ is annihilated by all simple root vectors $E_\alpha$ and $F_\alpha$ with $l_\alpha>1$ if and only if $\lambda\in X^\ast$.  Furthermore, in this case $L(\lambda)$ is graded by the sublattice $X^\ast$ in $X$.
\end{lemma}

We now consider the full subcategory of quasi-classical representations $\msc{E}_q$ in $\Rep(G_q)$.

\begin{lemma}\label{lem:696}
Let $\msc{E}_q$ denote the full subcategory in $\Rep(G_q)$ consisting of all representations which are annihilated by the simple root vectors $E_{\alpha}$ and $F_{\alpha}$ with $l_{\alpha}>1$.
\begin{enumerate}
\item All representations in $\msc{E}_q$ are graded by the sublattice $X^{\ast}$ in $X$.
\item A quantum group representation $V$ lies in $\msc{E}_q$ if and only if $V$ is annihilated by \emph{all} of the root vectors $E_\gamma$ and $F_\gamma$ for $\gamma\in \Phi^+$ with $l_\gamma>1$.
\item A representation $V$ lies in $\msc{E}_q$ if and only if, for each $W$ in $\Rep(G_q)$, the squared braidings appear as
\begin{equation}\label{eq:793}
c^2_{V,W}:V\ot W\to V\ot W,\ \ v\ot w\mapsto q^{-2}(\opn{deg}(w),\opn{deg}(v)) v\ot w
\end{equation}
and
\[
c^2_{W,V}:W\ot V\to W\ot V,\ \ w\ot v\mapsto q^{-2}(\opn{deg}(w),\opn{deg}(v)) w\ot v.
\]
\item A representation $V$ lies in $\msc{E}_q$ if and only if, for each $W$ in $\Rep(G_q)$, the squared braidings $c^2_{V,W}$ and $c^2_{W,V}$ are both homogenous endomorphisms with respect to the $X\times X$-gradings on $V\ot W$ and $W\ot V$.
\end{enumerate}
\end{lemma}

\begin{proof}
Statements (1) and (2) are covered in \cite[Lemmas 4.4, 4.5]{negron}.  For (3), one sees directly from the expression of the $R$-matrix that
\[
c_{V,W}^2=\sum_{\lambda,\mu\in X} q^{-2}(\lambda,\mu) 1_{\lambda}\ot 1_{\mu}\ \ \text{whenever $V$ or $W$ is in }\msc{E}_q,
\]
and hence verifies the formulas \eqref{eq:793}.  So we need only deal with the converse claim.  Suppose $V$ is \emph{not} in $\msc{E}_q$.  We can find a homogeneous vector $v\in V$ which is not annihilated by some $E_\alpha$, or $F_\alpha$, with $l_\alpha>1$.  We suppose arbitrarily that $v$ is not annihilated by such a positive root vector $E_\alpha$.

Take a simple representation $W=L(\mu)$ with highest weight vector $w$ for which $F_\alpha\cdot w\neq 0$.  Then
\[
\begin{array}{ll}
c_{W,V}^2(w,v)&=q^{-1}(\lambda,\mu)c_{V,W}(v,w)\\
&=q^{-2}(\lambda,\mu)w\ot v- (q_\alpha- q_\alpha^{-1})F_{\alpha}w\ot E_{\alpha}v+ \text{terms of other bidegrees}.
\end{array}
\]
Since the linear term $-(q_\alpha- q_\alpha^{-1})F_{\alpha}w\ot E_{\alpha}v$ is nonvanishing in this case, we have
\[
c_{W,V}^2(w,v)\neq q^{-2}(\opn{deg}(w),\opn{deg}(v)) w\ot v.
\]
The argument in the case where $v$ is not annihilated by some $F_\alpha$ is similar.
\par

Statement (4) was already established in the argument for (3).
\end{proof}

\subsection{Lusztig's dual group $G^\ast_\varepsilon$}
\label{sect:OGdual}

In \cite{lusztig90,lusztig90II,lusztig93} Lusztig considers a dual group $G^\ast$ to $G$ which is specified by the following data: The weight and root lattices for $G^\ast$ are $P^\ast$ and $lQ$ respectively, and the character lattice is $X^\ast$.  The simple roots are
\[
\alpha^\ast:=l_\alpha \alpha\ \in\ lQ
\]
and the Cartan integers are given by
\[
\langle \alpha^\ast,\beta^\ast\rangle_l := \frac{l_\alpha}{l_\beta} \frac{2 (\alpha,\beta)}{(\beta,\beta)}
\]
\cite[\S\ 2.2.5]{lusztig93}.  The normalized Killing form on $P^\ast$ is provided by a rescaling of the original form $(-,-)$ on $P$ \cite[\S\ 4.4]{negron}, and for each simple root $\alpha^\ast$ the corresponding fundamental weight $\omega_{\alpha^\ast}$ in $P^\ast$ is
$l_{\alpha}\omega_{\alpha}$.
\par

We note that the almost-simple factors in the dual group $G^\ast$ are in natural bijection with the almost-simple factors in $G$, and that each factor in $G^\ast$ is either of the same Dynkin type as its corresponding factor in $G$, or of Langlands dual type.  See \cite[Lemma 4.7]{negron}.
\par

Consider the dual parameter $\varepsilon$ on $G^\ast$ provided by restricting $q$ to $P^\ast$,
\[
\varepsilon:=q|_{P^\ast\times P^\ast}:P^\ast\times P^\ast\to k^{\times}.
\]
This form has corresponding scalar parameters
\[
\varepsilon_\alpha= \varepsilon(\alpha^\ast, \omega_{\alpha^\ast})= q_\alpha^{l_\alpha^2}=\pm 1.
\]
Hence $\varepsilon$ is a quasi-classical parameter for $G^\ast$, just in the sense that all of its scalar parameters take values $\pm 1$, and the corresponding quantum group $G^\ast_\varepsilon$ has semisimple representation theory \cite[Proposition 33.2.3]{lusztig93}.  Furthermore, since all of the $l_{\alpha^\ast}$ are $1$ in this case, the braiding on $\Rep(G^\ast_\varepsilon)$ appears as the semisimple operator
\[
c_{V,W}:V\ot W\to W\ot V,\ \ c_{V,W}(v,w)=\varepsilon^{-1}(\deg v,\deg w) w\ot v.
\]

\subsection{Original quantum Frobenius}

Let $\dot{U}^\ast_{\varepsilon}$ denote the modified quantum enveloping algebra for $G^\ast$ at $\varepsilon$, and let
\[
e^{(m)}_\alpha 1_\mu,\ f^{(m)}_\alpha 1_\mu\ \in \dot{U}^\ast_\varepsilon
\]
denote the standard generators, with $\mu\in X^\ast$.  By \cite[Theorem 35.1.9, \S\ 35.5.2]{lusztig93} (see also \cite{lentner16}) there is a surjective map of (non-unital) algebras
\[
fr^\ast:\dot{U}_q\to \dot{U}^\ast_\varepsilon
\]
which sends $E^{(n)}_\alpha 1_\lambda$ to $e^{(n/l_\alpha)}_\alpha 1_\lambda$ and  $F^{(n)}_\alpha 1_\lambda$ to $f^{(n/l_\alpha)}_\alpha 1_\lambda$ whenever $\lambda\in X^\ast$ and $l_\alpha\mid n$, and otherwise sends these vectors to $0$.
\par

The map $fr^\ast$ is furthermore compatible with the Hopf structures on $\dot{U}_q$ and $\dot{U}^\ast_\varepsilon$, in the sense that restricting along $fr^\ast$ provides an embedding of tensor categories
\[
\opn{Fr}:=\res_{fr^\ast}:\Rep(G^\ast_{\varepsilon})\to \Rep(G_q)
\]
with tensor compatibility $\opn{Fr}(V)\ot \opn{Fr}(W)\to \opn{Fr}(V\ot W)$ given by the identity.

The following is deducible from the original work \cite{lusztig93}.

\begin{proposition}[\cite{lusztig93}]\label{prop:og_frob}
The quantum Frobenius functor $\opn{Fr}:\Rep(G^\ast_\varepsilon)\to \Rep(G_q)$ restricts to a braided tensor equivalence onto the subcategory of quasi-classical representations in $\Rep(G_q)$,
\[
\opn{Fr}:\Rep(G^\ast_{\varepsilon})\overset{\sim}\to \msc{E}_q.
\]
\end{proposition}

\begin{proof}
The fact that $\opn{Fr}$ restricts to a tensor equivalence onto $\msc{E}_q$ is covered in \cite[Theorem 4.9]{negron}.  As for the braiding, one sees directly from the expression of the $R$-matrix for $\Rep(G_q)$ that the braiding on $\msc{E}_q$ is instituted by the bilinear form
\[
c_{V,W}(v,w)=q^{-1}(\deg v,\deg w)\cdot w\ot v.
\]
The identification with the braiding on $\Rep(G^\ast_\varepsilon)$ now follows from the fact that $\varepsilon$ is defined by restricting $q$ to $P^\ast$.
\end{proof}

\section{The Tannakian center in $\Rep(G_q)$}
\label{sect:Tan}

We calculate the Tannakian center in $\Rep(G_q)$ via a twisting of Lusztig's original quantum Frobenius functor.

\subsection{The Tannakian sublattice}
\label{sect:623}

Recall our lattices $lQ\subseteq X^\ast\subseteq P^\ast$ for the dual group $G^\ast$ (Section \ref{sect:OGdual}) and consider the refinements
\[
X^{\text{M\"ug}}:=\{\lambda \in X^\ast:q^2(\lambda,\mu)=1\ \text{at all }\mu\in X\}\vspace{1mm}
\]
and
\[
X^{\opn{Tan}}:=\{\lambda \in X^\ast:q^2(\lambda,\mu)=1\ \text{at all }\mu\in X,\ \text{and}\ q(\lambda,\lambda)=1\}.
\]
By definition, we have a sequence of inclusions $X^{\opn{Tan}}\subseteq X^{\text{\rm M\"ug}}\subseteq X^{\ast}$.  We are particularly interested in $X^{\opn{Tan}}$, as it eventually provides the character lattice for the Tannakian center in $\Rep(G_q)$.

\begin{lemma}\label{lem:Tan_lattice}
The subsets $X^{\opn{Tan}}$ and $X^{\text{\rm M\"ug}}$ are additive subgroups in $X^\ast$, and there is a sequence of inclusions
\[
lQ\ \subseteq\ X^{\opn{Tan}}\ \subseteq\ X^{\text{\rm M\"ug}}\ \subseteq\ X^\ast.
\]
\end{lemma}

\begin{proof}
The subset $X^{\text{M\"ug}}$ is a sublattice just because $q$ is bilinear.  For $\lambda,\lambda'\in X^{\opn{Tan}}$ we have
\[
q(\lambda+\lambda',\lambda+\lambda')=q(\lambda,\lambda)q^2(\lambda,\lambda')q(\lambda',\lambda').
\]
Each of the terms on the righthand side is $1$, so that the entire expression is trivial.  It follows that $X^{\opn{Tan}}$ is stable under addition.  Since this subset is clearly stable under negation, we see that $X^{\opn{Tan}}$ is in fact a sublattice in $P^\ast$.
\par

As for the inclusion $lQ\subseteq X^{\opn{Tan}}$ we have first that
\[
q^2(l_\alpha \alpha,-)=\big(q^2(\alpha,-)\big)^{l_\alpha}=1
\]
at each simple root $\alpha$, by the definition of $l_\alpha$, so that $lQ$ is contained in $X^{\text{M\"ug}}$.  To see that each $l_\alpha \alpha$ is in $X^{\opn{Tan}}$ we simply check
\[
q(l_\alpha \alpha,l_\alpha\alpha)=q(\alpha,\alpha)^{l_\alpha l_\alpha}=q^2(\alpha,\omega_{\alpha})^{l_\alpha l_\alpha}=1.
\]
\end{proof}

\subsection{The M\"ger center in $\Rep(G_q)$}

Recall that the subcategory $\msc{E}_q$ of quasi-classical representations in $\Rep(G_q)$ is the full subcategory of representations which are annihilated by all of the simple root vectors $E_\alpha$ and $F_\alpha$ with $l_\alpha>1$.
\par

From the calculation of the relative braiding operators from Lemma \ref{lem:696} (3) and (4) we obtain the following calculation of the M\"uger center in $\Rep(G_q)$.

\begin{proposition}\label{prop:mugercenter}
The M\"uger center in $\Rep(G_q)$ is precisely the subcategory
\[
Z_2(\Rep(G_q))=\left\{
\begin{array}{c}
\text{\rm The full subcategory of representations $V$ in}\\
\msc{E}_q\text{\rm\ whose $X^\ast$-grading is supported on }X^{\text{\rm M\"ug}}
\end{array}\right\}.
\]
\end{proposition}

\begin{proof}
Immediate from the expression \eqref{eq:793}, Lemma \ref{lem:696} (4), and the defintion of the sublattice $X^{\text{M\"ug}}$.
\end{proof}

\subsection{The Tannakian center in $\Rep(G_q)$}

Let $\msc{A}$ be a general braided tensor category which is of subexponential growth.  We recall that the Tannakian center in $\msc{A}$ is the maximal Tannakian subcategory $Z_{\opn{Tan}}(\msc{A})$ in the M\"uger center $Z_2(\msc{A})$.  As was shown in Corollary \ref{cor:E_tann}, this subcategory can be identified with the kernel of any symmetric fiber functor $Z_2(\msc{A})\to \opn{sVect}$.
\par

We have the following description of the Tannakian center in $\Rep(G_q)$.

\begin{proposition}\label{prop:tannakiancenter}
The Tannakian center in $\Rep(G_q)$ is the subcategory
\begin{equation}\label{eq:tannakiancenter}
Z_{\opn{Tan}}(\Rep(G_q))=\left\{
\begin{array}{c}
\text{\rm The full subcategory of representations $V$ in}\\
\msc{E}_q\text{\rm\ whose $X^\ast$-grading is supported on }X^{\opn{Tan}}
\end{array}\right\}.
\end{equation}
\end{proposition}

In the proof we employ a standard dimension function on the subcategory of rigid objects in a symmetric tensor category $\msc{E}$,
\[
\dim(V):=(\1\overset{\opn{coev}}\to V\ot V^\ast\overset{c_{V,V^\ast}}\to V^\ast\ot V\overset{\opn{ev}}\to \1).
\]
We note that the dimension of an object $\dim(V)$ is preserved under any symmetric tensor functors, and hence takes integer values whenever $\msc{E}$ has subexponential growth.  One sees this by considering the dimensions of objects in  $\opn{sVect}$ and noting the existence of a symmetric fiber functor for $\msc{E}$ in this case \cite{deligne02}.

\begin{proof}
Let $\msc{Z}\subseteq \msc{E}_q$ be the subcategory of quasi-classical $G_q$-representations whose gradings are supported on the sublattice $X^{\opn{Tan}}\subseteq X^\ast$.  Recall that the braiding on $\msc{E}_q$, and hence on $\msc{Z}$, is instituted directly by the dual parameter $\varepsilon=q|_{X^\ast\times X^\ast}$, by Lemma \ref{lem:696} (3).  We must establish an equality of subcategories $Z_{\opn{Tan}}(\Rep(G_q))=\msc{Z}$.
\par

Let us fix a symmetric fiber functor $fib:Z_2(\Rep(G_q))\to\opn{sVect}$.  Since the target category $\opn{sVect}$ is semisimple the kernel of this functor is stable under the formation of extensions, as well as subquotients.  So we need only determine the simple objects which lie in $Z_{\opn{Tan}}(\Rep(G_q))$ in order to determine the entire category.
\par

Suppose that the inclusion $X^{\opn{Tan}}\subseteq X^{\text{M\"ug}}$ is not an equality and consider a dominant weight $\lambda$ in the complement $X^{\text{M\"ug}}-X^{\opn{Tan}}$.  Then the corresponding simple $L(\lambda)$ lies in the M\"uger center $Z_2(\Rep(G_q))$, by the description of Proposition \ref{prop:mugercenter}, but not in our subcategory $\msc{Z}$.  Via M\"uger centrality of $L(\lambda)$ we see that $\varepsilon^2(\lambda,\lambda)=1$ while $\varepsilon(\lambda,\lambda)\neq 1$.  Hence $\varepsilon(\lambda,\lambda)=-1$ necessarily.
\par

Furthermore, for such $\lambda$ the simple $L(\lambda)$ has weights lying in the coset $\lambda+lQ\subseteq \lambda+X^{\opn{Tan}}$, by Lemma \ref{lem:Tan_lattice} and Proposition \ref{prop:og_frob}.  This implies that $\varepsilon(\mu,\mu)=-1$ for all weights $\mu$ with $L(\lambda)_\mu\neq 0$.  Hence the dimension of any such simple in $Z_2(\Rep(G_q))$ is the negation of the vector space dimension
\[
\dim(L(\lambda))=-\dim_k(L(\lambda)),
\]
and in particular is negative.  By preservation of dimensions under symmetric tensor functors we find $fib(L(\lambda))\notin \opn{Vect}$ and hence $L(\lambda)\notin Z_{\opn{Tan}}(\Rep(G_q))$. We now observe an inclusion
\begin{equation}\label{eq:1631}
Z_{\opn{Tan}}(\Rep(G_q))\ \subseteq\ \msc{Z}.
\end{equation}
\par

Note that the inclusion \eqref{eq:1631} holds when $X^{\opn{Tan}}=X^{\text{M\"ug}}$ as well, since in this case $\msc{Z}=Z_2(\Rep(G_q))$.  Hence in all cases we have $Z_{\opn{Tan}}(\Rep(G_q))\subseteq\msc{Z}$.  To show that the above inclusion is an equality it suffices to prove that the category $\msc{Z}$ is Tannakian.  So, we simply construct a fiber functor $\msc{Z}\to \opn{Vect}$.
\par

By the definition of $X^{\opn{Tan}}$ we have that $\varepsilon$ takes values $\pm 1$ on $X^{\opn{Tan}}$ and vanishes on the diagonal. Hence $\varepsilon$ restricts to an alternating form on $X^{\opn{Tan}}$, and therefore admits an alternating square root $\kappa$.  This form $\kappa$ on $X^{\opn{Tan}}$ specifically satisfies
\[
\kappa_{21}=\kappa^{-1}\ \ \text{and}\ \  \kappa^2=\varepsilon|_{X^{\opn{Tan}}\times X^{\opn{Tan}}}.
\]
One now checks that the tensor functor $fib_\kappa:\msc{Z}\to \opn{Vect}$, which we define as the usual forgetful functor equipped with the non-trivial tensor compatibility
\[
fib_{\kappa}(V)\ot fib_\kappa(W)\overset{\sim}\to fib_{\kappa}(V\ot W),\ \ v\ot w\mapsto \kappa^{-1}(\deg v,\deg w)v\ot w,
\]
provides a symmetric fiber functor for $\msc{Z}$.  It follows that $\msc{Z}$ is in fact Tannakian, and hence that the inclusion $Z_{\opn{Tan}}(\Rep(G_q))\subseteq \msc{Z}$ is an equality.
\end{proof}

\subsection{The Tannakian center via classical representations}
\label{sect:classic_Z}

Consider again the Tannakian sublattice $X^{\opn{Tan}}\subseteq X^\ast$.  Since $X^{\opn{Tan}}$ lies between the root and character lattices for the dual group $G^{\ast}$, characters on the quotient $(X^\ast/X^{\opn{Tan}})^{\vee}$ live naturally as a discrete central subgroup in $G^{\ast}$.  We can take the quotient by this subgroup to deduce a new semisimple algebraic group $\dG$ whose classical representations, and quantum representations at $\varepsilon$, are identified with full subcategories
\[
\Rep(\dG)\subseteq \Rep(G^\ast)\ \ \text{and}\ \ \Rep(\dG_\varepsilon)\subseteq \Rep(G^{\ast}_\varepsilon).
\]
These subcategories are precisely the subcategories of representations whose $X^\ast$-gradings are supported on the sublattice $X^{\opn{Tan}}$.

\begin{theorem}\label{thm:1021}
For $\dG$ as above, any choice of alternating square root $\kappa$ for $\varepsilon|_{X^{\opn{Tan}}\times X^{\opn{Tan}}}$ determines an equivalence of symmetric tensor categories
\[
F_\kappa:\Rep(\dG)\overset{\sim}\to \Rep(\dG_{\varepsilon})
\]
whose underlying linear equivalence fits into a diagram
\[
\xymatrix{
\Rep(\dG)\ar[rr]^{F_{\kappa}}\ar[d]_{\res} & & \Rep(\dG_\varepsilon)\ar[d]\ar[d]^{\res}\\
\Rep(\check{T})\ar[rr]^{=} & & \Rep(\check{T}_\varepsilon),
}
\]
and whose tensor compatibility is provided by the form $\kappa$,
\begin{equation}\label{eq:bleh}
F_\kappa(V)\ot F_\kappa(W)\to F_\kappa(V\ot W),\ \ v\ot w\mapsto \kappa(\deg v,\deg w)v\ot w.
\end{equation}
\end{theorem}

To be very explicit, by alternating here we mean that $\kappa$ satisfies $\kappa(\lambda,\lambda)=1$ at each $\lambda\in X^{\opn{Tan}}$.  As in the proof of Proposition \ref{prop:tannakiancenter}, one sees that such a square root $\kappa$ exists since $\varepsilon$ takes values $\pm 1$ on $X^{\opn{Tan}}$ and vanishes on the diagonal.
\par

For the proof, one rescales the generators in the quantum enveloping algebra $U_\varepsilon(\check{\mfk{g}})$ via some toral characters in order to allow for an algebra identification $U_\varepsilon(\check{\mfk{g}})\overset{\sim}\to U(\mfk{g})$ which couples appropriately with the tensor compatibility \eqref{eq:bleh}.  The details are provided in Appendix \ref{sect:A}.

\begin{remark}
In \cite[Proposition 33.2.3]{lusztig93} Lusztig produces a non-tensor equivalence between $\Rep(G^\ast)$ and $\Rep(G^\ast_\varepsilon)$, which then restricts to a non-tensor equivalence between $\Rep(\dG)$ and $\Rep(\dG_\varepsilon)$.  So it was previously known that we have a linear identification as in Theorem \ref{thm:1021}.
\end{remark}

We compose the symmetric equivalence $F_\kappa:\Rep(\dG)\overset{\sim}\to \Rep(\dG_\varepsilon)$ with the central embedding $\Rep(\dG_\varepsilon)\to \Rep(G_q)$ to obtain a central embedding from $\Rep(\dG)$ to the category of quantum group representations.  The calculation of the Tannakian center from Proposition \ref{prop:tannakiancenter} now reduces to the following.

\begin{theorem}\label{thm:Z_classical}
Consider the quotient $\dG$ of $G^\ast$ by the discrete central subgroup $(X^\ast/X^{\opn{Tan}})^\vee$, and any choice of alternating square root $\kappa$ for the dual parameter $\varepsilon|_{X^{\opn{Tan}}\times X^{\opn{Tan}}}$.  The form $\kappa$ determines a fully faithful braided tensor functor
\[
\opn{Fr}_{\kappa}:\Rep(\dG)\to \Rep(G_q)
\]
which is an equivalence onto the Tannakian center in $\Rep(G_q)$.  This functor sends each highest weight semple representation $L(\lambda)$ in $\Rep(\dG)$, for $\lambda$ dominant in $X^{\opn{Tan}}$, to the corresponding highest weight simple in $\Rep(G_q)$.
\end{theorem}

\subsection{Twisted quantum Frobenius and the Tannakian center}

\begin{definition}\label{def:Tan_Fr}
For $\dG$ and $\kappa$ as above, we let $\opn{Fr}_\kappa$ denote the braided tensor embedding
\[
\opn{Fr}_{\kappa}:\Rep(\dG)\to \Rep(G_q)
\]
from Theorem \ref{thm:Z_classical}.  We refer to $\opn{Fr}_{\kappa}$ as the $\kappa$-twisted quantum Frobenius functor.
\end{definition}

Despite the fact that the form $\kappa$ is not uniquely determined by the pairing of $G$ with $q$, the image of $\opn{Fr}_\kappa$ \emph{is} uniquely determined.  (As should be clear at this point, the image is the Tannakian center in $\Rep(G_q)$.)  Hence the functor $\opn{Fr}_\kappa$ is determined, at a completely abstract level, up to a symmetric automorphism of $\Rep(\dG)$, and for any two choices of $\kappa$ and $\kappa'$ there is a unique symmetric automorphism of $\Rep(\dG)$ which fits into a diagram
\[
\xymatrix{
\Rep(\dG)\ar@{-->}[rr]^{\exists !}_{\cong}\ar[dr]_{\opn{Fr}_{\kappa}} & & \Rep(\dG)\ar[dl]^{\opn{Fr}_{\kappa'}}\\
 & \Rep(G_q) &.
}
\]

\section{Simply-connected groups \& other examples}
\label{sect:examples}

We explicitly determine the Tannakian center in $\Rep(G_q)$ in a number of (both generic and specific) examples.  As a parallel calculation, we determine the dual groups $\dG$ in these examples and describe the nature of the functor $\opn{Fr}_{\kappa}$.  We find that the M\"uger and Tannakian centers in $\Rep(G_q)$ are generally distinct, but agree when $G$ is simply-connected and all of the scalar parameters $q_\alpha$ are of even order.  Such phenomena have concrete implications for our analysis of small quantum group representations which follows.  This final point is discussed in Section \ref{sect:why}.

\subsection{Simply-connected groups at even orders}

Consider a simply-connected semisimple algebraic group $G$ at a maximally non-degenerate parameter $q$ (Definiton \ref{def:max_nondegen}).  In this case the character lattice $X$ is the entire weight lattice $P$, and we have the basis of fundamental weights $\omega_\alpha$ associated to the simple roots $\alpha\in \Delta$.  For a general element $\lambda=\sum_\alpha c_\alpha \alpha$ in $P$, where the $c_\alpha$ here are allowed to be rational, we have
\begin{equation}\label{eq:1752}
q^2(\omega_\alpha,\lambda)=q(\omega_\alpha, 2c_\alpha \alpha).
\end{equation}
By our non-degeneracy assumption on $q$, vanishing of these values implies $2c_\alpha\in \mathbb{Z}$ for all $\alpha$.  Hence the expression \eqref{eq:1752} reduces to $q^2(\omega_\alpha,\lambda)=q_{\alpha}^{2c_\alpha}$ in this case.  If we suppose further that all of the scalar parameters $q_\alpha$ are of even order, and hence of order $2l_\alpha$, then we conclude that all of the $c_\alpha$ are integers which are divisible by the respective $l_\alpha$.

\begin{theorem}\label{thm:sc}
Consider a simply-connected semisimple algebraic group $G$ and a maximally non-degenerate parameter $q$.  Suppose that all of the scalar parameters $q_\alpha$, for simple $\alpha$, are of even order.  Then the following hold:
\begin{enumerate}
\item The Tannakian and M\"uger centers in $\Rep(G_q)$ agree.
\item The dual group $\dG$ is the Langlands dual group to $G$.
\item The Tannakian center in $\Rep(G_q)$ is equivalent, as a symmetric tensor category, to the category of classical $\dG$-representations.
\item Twisted quantum Frobenius provides an equivalence
\[
\opn{Fr}_{\kappa}:\Rep(\dG)\overset{\sim}\to Z_{\opn{Tan}}(\Rep(G_q)).
\]
\end{enumerate}
\end{theorem}

To highlight one class of examples, let $l$ be a positive integer which is divisible by the lacing number for $G$ and consider the associated parameter $q=\exp(\pi i (-,-)/l)$.  In this case $q$ is maximally non-degenerate and all of the $q_\alpha$ are of even order.  So the conclusions of Theorem \ref{thm:sc} hold at all such parameters.

\begin{proof}
By the calculations above we find that any weight $\lambda$ in $X^{\text{\rm M\"ug}}$ lies in the dualized root lattice $lQ$.  The sequence of inclusions
\[
lQ\subseteq X^{\opn{Tan}}\subseteq X^{\text{\rm M\"ug}}
\]
from Lemma \ref{lem:Tan_lattice} then forces equalities $lQ=X^{\opn{Tan}}= X^{\text{\rm M\"ug}}$.  From the descriptions of the M\"uger and Tannakian centers in $\Rep(G_q)$ given in Propositions \ref{prop:mugercenter} and \ref{prop:mugercenter} we now deduce an equality $Z_{\opn{Tan}}(\Rep(G_q))=Z_2(\Rep(G_q))$.  This establishes (1).
\par

For (2) we note that $\dG$ is the adjoint form of $G^\ast$, by the definiton of $\dG$ as a central quotient of $G^\ast$ and the fact that the root lattice for $G^\ast$ is $lQ$.  As explained in \cite[\S\ 4.4]{negron} $G^\ast$ is of Langlands dual type to $G$ in this case.  Since $G$ is simply-connected by hypothesis, we conclude that this adjoint form of $G^\ast$ is in fact the Langlands dual group $\dG=G^{\opn{Lan}}$.  Statements (3) and (4) are now obtained as a particular instance of Theorem \ref{thm:Z_classical}.
\end{proof}

We recall that, in general, twisted quantum Frobenius involves some non-trivial tensor compatibility.  For one example, for $G=\opn{SL}(n)$ at the parameter $q=\exp(\pi i (-,-)/3)$ we have all $q_\alpha=e^{\pi i/3}$ and the Tannakian center in $\Rep(\opn{SL}(n)_q)$ is directly identified with representations for quantum $\opn{PSL}(n)$ at the quasi-classical parameter $\varepsilon=\exp(\pi i(-,-))$.  It is only after twisting that we obtain an identification with classical representations for $\opn{PSL}(n)$.
\par

As we'll see below, the conclusions of Theorem \ref{thm:sc} do not hold when the specific hypotheses therein are not met, generally speaking.

\subsection{Deviations at odd orders and non-simply-connected groups}

We record some examples which demonstrate the necessity of the hypotheses in Theorem \ref{thm:sc}.

\begin{example}\label{ex:10}
Consider $\opn{SL}(2n)$ at $q=\exp(2\pi i (-,-)/l)$, where $l$ and $n$ are positive odd integers.  In this case all $q_\alpha$ are primitive $l$-th roots of unity.  For a complete list of simple roots $\{\alpha_1,\dots, \alpha_{2n-1}\}$, where consecutive roots are neighbors, we consider the weight
\[
\lambda_0=\frac{l}{2}(\sum_{r=1}^{n}\alpha_{2r-1}).
\]
One checks all pairings $(\alpha,\lambda_0)=\pm l$ to see that $\lambda_0$ is in fact in the weight lattice, and furthermore an element in $X^\ast$.  For the square we have $q^2(\mu,\lambda_0)=1$ at all $\mu\in X$, so that $\lambda_0$ is in $X^{\text{\rm M\"ug}}$.  However one checks
\[
(\lambda_0,\lambda_0)=\sum_r \frac{l^2\cdot(\alpha_{2r-1},\alpha_{2r-1})}{4}=\frac{nl^2}{2}
\]
so that $q(\lambda_0,\lambda_0)=(-1)^{ln}=-1$.  Hence $\lambda_0$ is in not in the Tannakian sublattice.  We therefore find that the inclusion $X^{\opn{Tan}}\subseteq X^{\text{\rm M\"ug}}$ is not an equality, and hence that the inclusion of symmetric tensor categories
\[
Z_{\opn{Tan}}(\Rep(\opn{SL}(2n)_q))\ \subseteq\ Z_2(\Rep(\opn{SL}(2n)_q)
\]
is not an equality.
\end{example}

\begin{example}
Consider $\opn{Sp}(2n)$ at $q=\exp(\pi i (-,-)/l)$, with $l$ is odd.  Then $q_\alpha$ is a primitive $2l$-th root of unity at all short $\alpha$, and $q_\beta$ is a primitive $l$-th root of unity at the unique long simple root $\beta$.  For $\lambda_0=l\beta/2$, again at the unique long simple root, one checks that $\lambda_0$ is in the weight lattice and we have $\lambda_0\in X^{\text{\rm M\"ug}}$.  For the self-pairing we have
\[
q(\lambda_0,\lambda_0)=\exp(\pi i l)= -1,
\]
so that $\lambda_0$ is not in the Tannakian sublattice.  This implies that the Tannakian center does not agree with the M\"uger center in $\Rep(\opn{Sp}(2n)_q)$.
\end{example}

We consider now an example of a different type, where the hypotheses of Theorem \ref{thm:sc} are violated but the Tannakian and M\"uger centers still agree.  This provides a discrete counterpoint to Example \ref{ex:10}.

\begin{example}
Consider $\SL(3)$ at $q=\exp(2\pi i (-,-)/l)$, with $l$ odd.  Elements in $\opn{X}^{\text{M\"ug}}$ are all weights of the form
\[
\lambda=\frac{c_1l}{2}\alpha_1+ \frac{c_2l}{2}\alpha_2,
\]
with the $c_j$ integers.  In order for $\lambda$ to lie in the weight lattice we must have
\[
(\alpha_1,\lambda)=c_1l-\frac{c_2l}{2}\ \in\mathbb{Z}\ \ \text{and}\ \ (\alpha_2,\lambda)=-\frac{c_1l}{2}+c_2l\ \in\mathbb{Z},
\]
which forces $c_1$ and $c_2$ to be even.  So we have $X^{\opn{Tan}}=X^{\text{M\"ug}}=lQ$
in this case.
\end{example}

One can also show that the Tannakian and M\"uger centers in $\Rep(G_q)$ diverge when $G$ is not simply-connected, even when the scalar parameters $q_\alpha$ are all of even order, generally speaking.  This occurs, for example, for quantum representations of projective $\opn{Sp}(6)$ at $q=\exp(\pi i(-,-)/4)$.  We leave the details to the intersted reader.

We note, as a final example, that we can produce some regularity at odd order parameters in the adjoint setting.  This point is well-established in the literature.  See for example \cite[\S\ 5]{davydovetingofnikshych18}.

\begin{example}\label{ex:odd}
Let $G$ be of adjoint type and $l$ be an odd integer which is coprime to the lacing number for $G$ and the determinant of the Cartan matrix.  Consider quantum group representations for $G$ at the parameter $q=\exp(2\pi i(-,-)/l)$.  In this case all $l_\alpha=l$, all $q_\alpha$ are of order $l$, and by our coprimeness assumption $q$ reduces to a non-degenerate form on the quotient $Q/lQ$.  This implies
\[
X^{\opn{Tan}}=X^{\text{M\"ug}}=lQ,
\]
and provides a calculation of the Tannakian$=$M\"uger center in $\Rep(G_q)$ as a copy of the classical representation category $\Rep(G)$.
\end{example}

\begin{remark}
Here we've focused on maximally non-degenerate parameters $q$.  These are the types of parameters which seem to appear in nature, but they are not the only advantageous parameters which one might consider.  One can, for example, define a class of ``odd order" parameters $q$ at which the Tannakian and M\"uger centers in $\Rep(G_q)$ always agree.  However, in this case one loses control of the nature of the dual group $\dG$.
\end{remark}

\subsection{Relevance for small quantum groups}
\label{sect:why}

We begin with a trivial observation.

\begin{lemma}\label{lem:1703}
Suppose that $\msc{A}$ is a braided tensor category whose M\"uger center is non-Tannakian.  Then $\msc{A}$ does not admit a surjective braided tensor functor to a non-degenerate tensor category.
\end{lemma}

To clarify, by a non-degenerate braided tensor category we mean a braided tensor category $\msc{D}$ whose M\"uger center $Z_2(\msc{D})$ is identified with the image of the unit $\opn{Vect}\to \msc{D}$ (cf.\ \cite{shimizu19}).

\begin{proof}
Any surjective braided tensor functor $F:\msc{A}\to \msc{D}$ must send the M\"uger center in $\msc{A}$ to the M\"uger center in $\msc{D}$.  In the event that the M\"uger center in $\msc{D}$ is trivial, i.e.\ identified with $\opn{Vect}$, the functor $F$ then restricts to a symmetric fiber functor for $Z_2(\msc{A})$.  This cannot happen when the M\"uger center in $\msc{A}$ is non-Tannakian.
\end{proof}

Lemma \ref{lem:1703} tells us that there is no ``modular small quantum group" for $G$ at $q$ whenever the M\"uger center in $\Rep(G_q)$ is non-Tannakian.  At least, this is the case if one expects the small quantum group to have a representation category which admits a surjective braided tensor functor from the original quantum group category $\Rep(G_q)$.  Hence our interest in the deviations between the Tannakian and M\"uger centers in $\Rep(G_q)$ discussed above.

\section{An intermediate small quantum algebra}
\label{sect:sqa}

Our category of small quantum group representations will eventually be identified with representations for a finite-dimensional quasi-Hopf algebra $u_q$, in Section \ref{sect:smallquantumgroup}.  The algebra $u_q$ is strongly related to an ``intermediate small quantum algebra" $u_{q,\kappa}$ which we introduce in this section, as well as a ``smallest quantum algebra" $\bar{u}_q$ which was studied extensively by Andersen, Polo, and Wen \cite{andersenpolowen91,andersenpolowen92}, as well as Lentner \cite{lentner16} and the author \cite{negron}.
\par

As the name suggest, representations of the intermediate algebra $\Rep(u_{q,\kappa})$ provide an intermediate stage between the big category $\Rep(G_q)$ and its fiber $\opn{Vect}\ot_{\Rep(\dG)}\Rep(G_q)$ over the Tannakian center.

\begin{center}
\emph{Throughout this section we fix an alternating form $\kappa$ as in Theorem \ref{thm:1021}.  All constructions employed below are relative to this choice of $\kappa$.}
\end{center}
 
\subsection{Summary}
\label{sect:summary}

The following two sections are somewhat technical, as we introduce and analyze a number of mechanisms which will ultimately be discarded. Indeed, on a first reading one might only skim their contents then proceed directly to Section \ref{sect:modular}.  In order to facilitate such a reading we describe the main points of this section and its successor below.
\par

Let's fix an alternating form $\kappa$ on $X^{\opn{Tan}}$, as in Theorem \ref{thm:1021}, and consider the finite character group
\begin{equation}\label{eq:sig_ma}
\Sigma:=(X^{\opn{Tan}}/\opn{rad}(q,\kappa))^\vee,\ \ \text{where}\ \ \opn{rad}(q,\kappa)=\opn{rad}(q)\cap\opn{rad}(\kappa).
\end{equation}
Here one might note that the radical of $q$ lies in $X^{\opn{Tan}}$, so that the simultaneous radical $\opn{rad}(q,\kappa)$ is unambiguously a finite index sublattice in $X^{\opn{Tan}}$.  We have the inclusion into the dual torus $\Sigma\to \Hom_{\opn{Grp}}(X^{\opn{Tan}},k^\times)=\check{T}\subseteq \dG$ and corresponding restriction functor $\res:\Rep(\dG)\to \Rep(\Sigma)$.
\par 

As our primary deliverable, in Proposition \ref{prop:int_basechange} we provide an equivalence of braided tensor categories
\[
\Rep(\Sigma)\ot_{\Rep(\dG)}\Rep(G_q)\overset{\sim}\to \Rep(u_{q,\kappa}),
\]
where $u_{q,\kappa}$ is a finite-dimensional Hopf algebra which is, up to some toral extension, identified with Lusztig's finite-dimensional algebra from \cite{lusztig90,lusztig90II}.  Via associativity of base change, we then obtain an expression
\[
\opn{Vect}\ot_{\Rep(\dG)}\Rep(G_q)\cong \opn{Vect}\ot_{\Rep(\Sigma)}\Rep(\Sigma)\ot_{\Rep(\dG)}\Rep(G_q)
\]
\begin{equation}\label{eq:1841}
\cong \opn{Vect}\ot_{\Rep(\Sigma)}\Rep(u_{q,\kappa}),
\end{equation}
and can therefore leverage the finite tensor category $\Rep(u_{q,\kappa})$ in our analysis of the above fiber.  (See Lemma \ref{lem:2159}.)
\par

In terms of its algebraic structure, $u_{q,\kappa}$ is specifically the subalgebra in the completed quantum enveloping algebra $\hat{U}_q$ which is generated by the characters on the quotient $X/\opn{rad}(q,\kappa)$, along with the root vectors $E_\gamma$ and $F_\gamma$ associated to all $\gamma$ with $l_\gamma>1$ (Lemma \ref{lem:surj}).  We have the expected triangular decomposition
\[
u^+_{q,\kappa}\ot k\Lambda\ot u^-_{q,\kappa}\overset{\sim}\to u_{q,\kappa},
\]
where $\Lambda=(X/\opn{rad}(q,\kappa))^\vee$, corresponding calculation of the dimension for $u_{q,\kappa}$, analysis of projectives and simples in $\Rep(u_{q,\kappa})$ via highest weights, etc., etc., which help us to understand the nature of the fiber category \eqref{eq:1841}.

\subsection{Luztig's finite-dimensional algebra}

We consider again the quantum enveloping algebra $U_q$ for $G$ at $q$.  For the moment, let us regard $U_q$ as a ``copy" of the quantum enveloping algebra introduced in Sectin \ref{sect:quantum_groups}, and for each root $\gamma$ let $\msf{K}_\gamma$ denote the standard grouplike element in our new copy of $U_q$.  As a starting point, we consider Lusztig's finite-dimensional subalgebra
\begin{equation}\label{eq:1854}
\mfk{v}_q:=\left\{
\begin{array}{c}
\text{The subalgebra in $U_q$ generated by all of the}\\
\text{grouplikes $\msf{K}_\alpha$, for simple $\alpha$, and all of the root}\\
\text{vectors $E_\gamma$ and $F_\gamma$ for $\gamma\in \Phi^+$ with $l_\gamma>1$}
\end{array}
\right\}.
\end{equation}
This algebra already appears in the original text \cite[\S\ 8.1, 8.2]{lusztig90II} and was studied extensively by Lentner \cite{lentner16}.

\begin{remark}
Our notation for the subalgebra \eqref{eq:1854} is unorthodox, relative to the primary sources \cite{lusztig90,lusztig90II,lusztig93}.
\end{remark}

By \cite{lusztig90II} \cite[Theorems 5.2, 5.4]{lentner16}, $\mfk{v}_q$ is a Hopf subalgebra in $U_q$ and the positive and negative subalgebras $\mfk{v}_q^+$ and $\mfk{v}_q^-$ in $\mfk{v}_q$ admit bases
\[
\{E_{\gamma_1}^{m_1}\dots E_{\gamma_t}^{m_t}:0\leq m_i\leq l_{\gamma_i}\ \text{at all }i\}\ \ \text{and}\ \ \{F_{\gamma_1}^{m'_1}\dots F_{\gamma_t}^{m'_t}:0\leq m'_i\leq l_{\gamma_i}\ \text{at all }i\}
\]
respectively.  Multiplication provides a triangular decomposition
\[
\mfk{v}_q^+\otimes k[K_\alpha:\alpha\in \Delta]\otimes \mfk{v}_q^-\overset{\sim}\to \mfk{v}_q.
\]

The algebra $\mfk{v}_q$ has a distinguished set of skew primitive generating root vectors
\[
E_\beta\ \text{and} \  F_\beta\ \text{for }\beta \in \Delta_l,
\]
which are labeled by a particular ``$l$-base" $\Delta_l$ in $\Phi^+$ \cite[\S\ 6.2]{negron}.  When all of the $l_\alpha$ are positive this ``$l$-base" is just the usual base $\Delta$ for the positive roots, and in most other cases $\Delta_l$ is a base for the subsystem $\{\gamma\in \Phi:l_\gamma >1\}$ (though there are some unique phenomena for type $G_2$ at at a $4$-th root of unity).
\par

To avoid getting sidetracked by an accounting of quantum groups at extraordinarily small order parameters, we truncate our introduction to $\mfk{v}_q$ here and invite the reader to see \cite[\S\ 6]{negron}, or the original text \cite{lentner16}, for any further details.

\subsection{An intermediate small quantum algebra}
\label{sect:sqa2}

We consider the simultaneous radical $\opn{rad}(q,\kappa)=\opn{rad}(q)\cap\opn{rad}(\kappa)$ in $X$ and characters on the resulting quotient group
\begin{equation}\label{eq:1930}
\Lambda{\color{gray}(=\Lambda_\kappa)}:=(X/\opn{rad}(q,\kappa))^\vee.
\end{equation}
Note that the radical $\opn{rad}(q,\kappa)$ is a finite index subgroup in $X$, since both of the forms in question are torsion, and that all of the characters $K_\gamma:X\to k^\times$, $\lambda\mapsto q(\gamma,\lambda)$, associated to roots $\gamma\in \Phi$ vanish on $\opn{rad}(q,\kappa)$.  Hence these characters exist as elements in the group \eqref{eq:1930}.
\par

We have a natural action of $\Lambda$ on $\mfk{v}_q$ via Hopf automorphisms,
\[
\xi\cdot E_\alpha=\xi(\alpha) E_\alpha,\ \ \xi\cdot F_\alpha=\xi(-\alpha)F_\alpha,\ \ \xi\cdot K_\alpha=K_\alpha\ \ \text{for}\ \ \xi\in \Lambda,
\]
and take the associated smash product Hopf algebra $\mfk{v}_q\rtimes \Lambda$.  We observe the central grouplike elements $K_\gamma\cdot \msf{K}_\gamma^{-1}$ in this smash product and form the quotient
\[
u_{q,\kappa}:= \mfk{v}_q\rtimes \Lambda/(K_\gamma\cdot \msf{K}_\gamma^{-1}:\gamma\in \Phi).
\]
The algebra $u_{q,\kappa}$ admits a unique Hopf structure under which all of the toral characters $\Lambda\subseteq u_{q,\kappa}$ are grouplike and the map $\mfk{v}_q\to u_{q,\kappa}$ is a map of Hopf algebras.  Furthermore, the triangular decomposition for $\mfk{v}_q$ induces a triangular decomposition for $u_{q,\kappa}$,
\[
u^+_{q,\kappa}\otimes k\Lambda\otimes u^-_{q,\kappa}\overset{\sim}\to u_{q,\kappa},
\]
where $u_{q,\kappa}^{\pm}=\mfk{v}^\pm_q$.  From this triangular decomposition and the standard bases for $\mfk{v}^{\pm}_q$ recalled above we obtain both bases and a dimension calculation for $u_{q,\kappa}$.

\begin{lemma}\label{lem:dims}
The positive and negative subalgebras $u^{\pm}_{q,\kappa}$ have respective bases
\[
\{E_{\gamma_1}^{m_1}\dots E_{\gamma_t}^{m_t}:0\leq m_i\leq l_{\gamma_i}\ \text{\rm at all }i\}\ \ \text{and}\ \ \{F_{\gamma_1}^{m'_1}\dots F_{\gamma_t}^{m'_t}:0\leq m'_i\leq l_{\gamma_i}\ \text{\rm at all }i\}.
\]
Furthermore, the dimension of $u_{q,\kappa}$ is precisely
\[
\dim(u_{q,\kappa})=[X:\opn{rad}(q,\kappa)]\cdot (\prod_{\gamma \in \Phi^+}l_\gamma)^2.
\]
\end{lemma}

\subsection{Irreducible $u_{q,\kappa}$-representations}

In $u_{q,\kappa}$ we have the nonnegative subalgebra $u^{\geq 0}_{q,\kappa}$, and define the baby Verma modules in the standard way
\[
M(\bar{\lambda}):=u_{q,\kappa}\ot_{u^{\geq 0}_{q,\kappa}}k(\bar{\lambda}).
\]
Here $\bar{\lambda}$ is an element of the truncated lattice $\Lambda^{\vee}=X/\opn{rad}(q,\kappa)$ and $k(\bar{\lambda})$ is the associated $1$-dimensional simple representation over the small Borel $u^{\geq 0}_{q,\kappa}$.
\par

Via the usual analysis one sees that $M(\bar{\lambda})$ admits a unique simple quotient $L(\bar{\lambda})$, and that this quotient has a unique up-to-scaling highest weight vector $v\in L(\bar{\lambda})$ which is of weight $\bar{\lambda}$.  (The vector $v$ is simply the image of the generating vector in $k(\bar{\lambda})$.)  In this way we obtain a classification of irreducible representations via highest weights
\[
\Lambda^{\vee}=X/\opn{rad}(q,\kappa)\overset{\cong}\longleftrightarrow \opn{Irrep}(u_{q,\kappa}).
\]

\subsection{Restriction from $\Rep(G_q)$}

Any $G_q$-representation $V$ admits a natural action of the algebra $u_{q,\kappa}$.  Explicitly, the root vectors $E_\gamma$ and $F_\gamma$ in $u_{q,\kappa}$ act via their corresponding vectors in $U_q$, and the toral elements $\xi\in \Lambda$ act as the characters
\[
\xi\cdot v=\xi(\deg v)\cdot v\ \ \text{for homogeneous }v\in V.
\]
In this way we observe a restriction functor
\[
\res:\Rep(G_q)\to \Rep(u_{q,\kappa}).
\]
\par

Since our Hopf structure on $u_{q,\kappa}$ is induced by the Hopf structure on $\mfk{v}_q$, and hence induced by the Hopf structure on the quantum enveloping algebra $U_q$, we see that $\res$ is a tensor functor with trivial tensor compatibility
\[
\res(V)\ot \res(W)\overset{=}\to \res(V\ot W).
\]

\subsection{Relation to the smallest quantum algebra from \cite{negron}}
\label{sect:uquqk}

In \cite{negron} we introduce a ``smallest quantum algebra" $\bar{u}_q$ for $G$ at $q$, where $G$ and $q$ are arbitrary.  At an odd order parameter which is also coprime to the lacing number for $G$, this algebra is just Lusztig's algebra $\mfk{v}_q$.  So, between the works of Parshall and Wang \cite{parshallwang91}, Andersen, Polo, and Wen \cite{andersenpolowen91,andersenpolowen92}, and the author \cite{negron}, the algebra $\bar{u}_q$ and its representations are very well-understood.  We describe the precise relationship between $\bar{u}_q$ and $u_{q,\kappa}$ below.

\begin{remark}
Our algebra $u_{q,\kappa}$ is essentially the algebra $v_q$ from \cite{negron}.
\end{remark}

To begin, the composite of quantum Frobenius with restriction defines a tensor functor $\Rep(G^\ast_{\varepsilon})\to \Rep(u_{q,\kappa})$, and via Lemma \ref{lem:696} we see that all $G^\ast_{\varepsilon}$-representation decomposes into $1$-dimensional representations over $u_{q,\kappa}$. These $1$-dimensional representations are precisely the simples $L(\bar{\lambda})$ whose highest weights lie in the subgroup $X^\ast/\opn{rad}(q,\kappa)$ of $X/\opn{rad}(q,\kappa)$.  We consider the corresponding group of characters
\begin{equation}\label{eq:theta}
\Theta{\color{gray}(=\Theta_\kappa)}:=\big(X^\ast/\opn{rad}(q,\kappa)\big)^\vee.
\end{equation}
\par

From the above analysis we observe a tensor embedding $\Rep(\Theta)\to \Rep(u_{q,\kappa})$ with trivial tensor compatibility, and corresponding Hopf quotient $u_{q,\kappa}\to k\Theta$.  We then take the subalgebra of left $k\Theta$-coinvariants to recover the smallest quantum algebra
\[
\bar{u}_q:= {^\Theta(u_{q,\kappa})}
\]
from \cite[\S\ 6.4]{negron}. This subalgebra is generated by the normalized root vectors $\mbf{E}_\gamma=K_\gamma E_\gamma$ and $F_\gamma$ in $u_{q,\kappa}$, for $\gamma\in \Phi^+$ with $l_\gamma>1$, and the characters $\xi \in \Lambda$ which vanish on $X^{\ast}$.

\begin{remark}
The subalgebra $\bar{u}_q$ is independent of our choice of $\kappa$, and in fact lives in the intersection $\cap_\kappa u_{q,\kappa}$ over all possible choices of $\kappa$, where we take this intersection in the completed enveloping algebra $\hat{U}_q$, for example.
\end{remark}

\subsection{Projectives vs.\ projectives}

We consider again the subalgebra $\bar{u}_q\subseteq u_{q,\kappa}$ and its corresponding finite group $\Theta$ from \eqref{eq:theta}.

\begin{lemma}\label{lem:2034}
\begin{enumerate}
\item $\bar{u}_q$ is a normal coideal subalgebra in $u_{q,\kappa}$.
\item $u_{q,\kappa}$ is free over $\bar{u}_q$.
\item For any $u_{q,\kappa}$-representations $V$ and $W$, there is a natural $\Theta$-action on the morphisms $\Hom_{\bar{u}_q}(V,W)$ and corresponding identification
\[
\Hom_{u_{q,\kappa}}(V,W)=\Hom_{\bar{u}_q}(V,W)^{\Theta}.
\]
\end{enumerate}
\end{lemma}

For (1) we mean specifically that the restriction functor $\res:\Rep(u_{q,\kappa})\to \Rep(\bar{u}_q)$ is normal, in the sense of Definition \ref{def:normal}.

\begin{proof}
(1) Since the normalized root vectors $\mbf{E}_{\gamma}=K_\gamma E_{\gamma}$ and $F_{\gamma}$ lie in $\bar{u}_q$, the subalgebra $\bar{u}_q$ is stable under conjugation by these elements
\[
[\mbf{E}_\gamma,-],\ [F_{\gamma},-]:\bar{u}_q\to \bar{u}_q.
\]
One also sees directly that $\bar{u}_q$ is stable under the adjoint action of the grouplikes $\Lambda$ on $u_{q,\kappa}$.  Since the vectors $\mbf{E}_\gamma$, $F_\gamma$, and the grouplikes $\Lambda$ generate $u_{q,\kappa}$, it follows that for any $u_{q,\kappa}$-representation $V$ the subspace $V^{\bar{u}_q}$ of vectors on which $\bar{u}_{q}$ acts trivially is a $u_{q,\kappa}$-subrepresentation in $V$.  Hence the restriction functor $\res:\Rep(u_{q,\kappa})\to \Rep(\bar{u}_q)$ is normal.  (Cf.\ \cite[Proposition 7.1]{negron}.)

(2) Follows from Skryabin's theorem \cite[Theorem 6.1]{skryabin07}, or more directly from the respective bases for $\bar{u}_q$ and $u_{q,\kappa}$ provided in \cite[Lemma 5.2]{negron} and Lemma \ref{lem:dims}.

(3) From (1), the $\bar{u}_q$-invariants $T^{\bar{u}_q}$ in any $u_{q,\kappa}$-representation form a $u_{q,\kappa}$-subrepresentation in $T$.  Furthermore, the $u_{q,\kappa}$-action on the invariants factor through the quotient $u_{q,\kappa}\to k\Theta$, since $\Rep(\Theta)\subseteq \Rep(u_{q,\kappa})$ is seen to be the kernel of the restriction functor $\Rep(u_{q,\kappa})\to \Rep(\bar{u}_q)$. So we have directly
\[
T^{u_{q,\kappa}}=(T^{\bar{u}_q})^{u_{q,\kappa}}=(T^{\bar{u}_q})^{\Theta}.
\]
The claimed result now follows from the expressions
\[
\Hom_{\bar{u}_q}(V,W)=\Hom_k(V,W)^{\bar{u}_q}\ \ \text{and}\ \ \Hom_{u_{q,\kappa}}(V,W)=\Hom_k(V,W)^{u_{q,\kappa}}.
\]
\end{proof}

We recall that both of the categories $\Rep(u_{q,\kappa})$ and $\Rep(\bar{u}_q)$ are Frobenius, in the sense that their projectives and injectives agree.  In the former case this follows by Larson and Sweedler \cite{larsonsweedler69}, and the latter case follows by Skryabin \cite[Theorem 6.1]{skryabin07} (or \cite[Theorem 11.3]{negron}).  One now applies Lemma \ref{lem:2034} (2) and (3) to obtain the following.

\begin{corollary}\label{cor:1328}
An object $V$ in $\Rep(u_{q,\kappa})$ is projective (equivalently injective) if and only if its restriction to $\Rep(\bar{u}_q)$ is projective (equivalently injective).
\end{corollary}

\section{Intermediate base change}
\label{sect:base_change}

In this section we calculate the intermediate base change
\begin{equation}\label{eq:2034}
\Rep(\Sigma)\ot_{\Rep(\dG)}\Rep(G_q)
\end{equation}
of $\Rep(G_q)$ along the restriction functor $\Rep(\dG)\to \Rep(\Sigma)$, where $\Sigma$ is the toral subgroup from \eqref{eq:sig_ma}.  As we show in Theorem \ref{thm:base_change} below, this category is specifically identified with representations $\Rep(u_{q,\kappa})$ for our intermediate small quantum algebra from Section \ref{sect:sqa}.

\subsection{Surjectivity of restriction, exactness of induction}

The following is established in \cite{andersenpolowen91,andersenpolowen92} at odd order parameters.  The general case can be found in \cite{negron}.

\begin{theorem}[{\cite[Lemma 11.1]{negron}}]
The category $\Rep(G_q)$ has enough projectives and injectives, and is Frobenius.  Furthermore, the subcategory $\rep(G_q)$ of finite-dimensional representations is closed under taking projective covers and injective hulls.
\end{theorem}

We now apply \cite[Theorem 11.3]{negron} to obtain the following.

\begin{theorem}\label{thm:1359}
An object $V$ in $\Rep(G_q)$ is projective (equivalently injective) if and only if its restriction to $\Rep(u_{q,\kappa})$ is projective (equivalently injective).
\end{theorem}

\begin{proof}
By \cite[Theorem 11.3]{negron}, a $G_q$-representation $\Rep(G_q)$ is projective if and only if it is projective over $\bar{u}_q$.  So the result follows by Corollary \ref{cor:1328} and the fact that the restriction functor $\Rep(G_q)\to \Rep(\bar{u}_q)$ \cite[\S 5.3]{negron} factors through $\Rep(u_{q,\kappa})$.
\end{proof}

An analysis of the simples now gives the following.

\begin{lemma}\label{lem:surj}
The restriction functor $\res:\Rep(G_q)\to \Rep(u_{q,\kappa})$ is surjective
\end{lemma}

\begin{proof}
Since the restriction functor preserves projectives, by Theorem \ref{thm:1359}, it suffices to show that each simple $L(\bar{\lambda})$ in $\Rep(u_{q,\kappa})$ admits a surjection $\res(P)\to L(\bar{\lambda})$ from some projective $P$ in $\Rep(G_q)$.  By taking projective covers in $\Rep(G_q)$, we can find such a $P$ in $\Rep(G_q)$ provided we can find a simple representation $L'$ in $\Rep(G_q)$ which admits a surjection $\res(L')\to L(\bar{\lambda})$.  However, we can find such a simple by lifting $\bar{\lambda}$ to a dominant weight $\lambda\in X^+$ and taking $L'=L(\lambda)$.
\end{proof}

\begin{remark}
From a slightly more refined perspective, one can show that each simple in $\Rep(G_q)$ restricts to a semisimple object in $\Rep(u_{q,\kappa})$.  This follows from a consideration of the Steinberg decomposition in the simply-connected case \cite{lusztig89} \cite[Theorem 8.7, Corollary 8.8]{negron}.
\end{remark}

We can also apply the findings from Theorem \ref{thm:1359} to deduce exactness, and faithfulness, of induction.

\begin{proposition}\label{prop:ind_exact}
The right adjoint to restriction $\opn{ind}:\Rep(u_{q,\kappa})\to \Rep(G_q)$ is both exact and faithful.
\end{proposition}

\begin{proof}
As explained in \cite[proof of Proposition 10.3]{negron}, which itself follows \cite[Theorem 4.8]{andersenpolowen92}, exactness of induction follows from the fact that restriction preserves injectives.  (See Theorem \ref{thm:1359}.)
\par

As for faithfulness, we've argued in the proof of Lemma \ref{lem:surj} that each simple $L$ in $\Rep(u_{q,\kappa})$ admits a surjection $\res(L')\to L$ from some simple $L'$ in $\Rep(G_q)$.  From the formula
\[
\Hom_{u_{q,\kappa}}(\res L', L) = \Hom_{G_q}(L', \opn{ind} L)
\]
it follows that $\opn{ind}(L)$ is nonzero for each simple $L$ in $\Rep(u_{q,\kappa})$.  Via exactness we find that $\opn{ind}(V)\neq 0$ at arbitrary nonzero $V$ in $\Rep(u_{q,\kappa})$, and hence that induction is faithful.
\end{proof}

\subsection{Further analyses of restriction and induction}
\label{sect:vq}

We establish a few more preliminaries for the calculation of the intermediate fiber \eqref{eq:2034}.
\par

First, we note that surjectivity of the restriction functor $\res$ implies that the Hopf map $u_{q,\kappa}\to \hat{U}_q$ to the completed quantum enveloping algebra is injective.  One sees now, directly from the expression of the $R$-matrix given in Section \ref{sect:R}, that the $R$-matrix for $\Rep(G_q)$ sits in the subalgebra
\[
R\ \in u_{q,\kappa}\ot u_{q,\kappa}\subseteq \hat{U}_q\ot \hat{U}_q.
\]
This $R$-matrix then provides a braiding on the category of $u_{q,\kappa}$-representations.

\begin{proposition}
There is a unique braiding on the tensor category $\Rep(u_{q,\kappa})$ under which the restriction functor $\res:\Rep(G_q)\to\Rep(u_{q,\kappa})$ is braided monoidal.
\end{proposition}

The uniqueness claim here comes from surjectivity of restriction.
\par

As was discussed in Section \ref{sect:uquqk}, the restriction functor sends each $G^\ast_\varepsilon$-representation in $\Rep(G_q)$ to a direct sum of $1$-dimensional simples over $u_{q,\kappa}$.  These are precisely the simples with highest weights in $X^\ast/\opn{rad}(q,\kappa)$.  We consider the more restrictive class of simples which are labeled by elements in the quotient $X^{\opn{Tan}}/\opn{rad}(q,\kappa)$, take
\[
\Sigma:=\big(X^{\opn{Tan}}/\opn{rad}(q,\kappa)\big)^\vee,
\]
and observe a tensor embedding
\begin{equation}\label{eq:2144}
\Rep(\Sigma)\to \Rep(u_{q,\kappa}).
\end{equation}
This embedding defines, and is defined by, a Hopf surjection $u_{q,\kappa}\to k\Sigma$.
\par

Let us note that $\Rep(\Sigma)$ is precisely the image of $\Rep(\dG)\cong \Rep(\dG_\varepsilon)$ in $\Rep(u_{q,\kappa})$ under restriction.  Indeed, $\Sigma$ is identified with a discrete toral subgroup in $\dG$ and the induced map $\Rep(\dG_\varepsilon)\to \Rep(\Sigma)$ is just given by restriction along the inclusion $\Sigma\to\dG$.
\par

Under the natural braiding on $\Rep(u_{q,\kappa})$, the embedding $\Rep(\Sigma)\to \Rep(u_{q,\kappa})$ identifies $\Rep(\Sigma)$ with a M\"uger central subcategory in $\Rep(u_{q,\kappa})$.  One can show further that the Tannakian center in $\Rep(u_{q,\kappa})$ is precisely $\Rep(\Sigma)$, and one sees directly that the braiding on $\Rep(\Sigma)$ is simply induced by the dualized form $\varepsilon$.  So we may take
\[
\Rep(\Sigma_{\varepsilon})=\Rep(\Sigma)\ \text{with non-trivial symmetry provided by }\varepsilon
\]
and refine the embedding of \eqref{eq:2144} to a central tensor embedding from $\Rep(\Sigma_{\varepsilon})$ into $\Rep(u_{q,\kappa})$.
\par

We now have a diagram of braided tensor functors
\begin{equation}\label{eq:2135}
\xymatrix{
\Rep(\dG_\varepsilon)\ar[rr]^{\opn{Fr}}\ar[d]_{\res} & & \Rep(G_q)\ar[d]^{\res}\\
\Rep(\Sigma_{\varepsilon})\ar[rr] & & \Rep(u_{q,\kappa})
}
\end{equation}
in which the vertical maps are surjective and the horizontal maps are M\"uger central embeddings.

\begin{lemma}\label{lem:2143}
$\opn{ind}_{u_{q,\kappa}}^{G_q}(\1)=\opn{ind}_{\Sigma_{\varepsilon}}^{\dG_\varepsilon}(\1)$.
\end{lemma}

\begin{proof}
We have
\[
\opn{ind}_{u_{q,\kappa}}^{G_q}(\1)=\O(G_q)^{u_{q,\kappa}}\subseteq \O(G_q)^{\bar{u}_q}=\O(G^\ast_\varepsilon),
\]
where the final equality here follows by \cite[Corollary 7.5]{negron}.  We note that $\O(G^\ast_\varepsilon)$ is a $u_{q,\kappa}$-subrepresentation in $\O(G_q)$ which is annihilated by all of the root vectors, so that for $\Psi=(X^\ast/X^{\opn{Tan}})^\vee$ we refine the above inclusion to get
\[
\O(G_q)^{u_{q,\kappa}}=\O(G^\ast_\varepsilon)^{u_{q,\kappa}}\subseteq\O(G^\ast_\varepsilon)^\Psi=\O(\dG_\varepsilon).
\]
Finally, since the $u_{q,\kappa}$-action on $\O(\dG_\varepsilon)$ factors though the quotient $u_{q,\kappa}\to k\Sigma$, we find
\[
\opn{ind}_{u_{q,\kappa}}^{G_q}(\1)=\O(G_q)^{u_{q,\kappa}}=\O(\dG_\varepsilon)^\Sigma=\opn{ind}_{\Sigma_{\varepsilon}}^{\dG_\varepsilon}(\1).
\]
\end{proof}

Note that the form $\kappa$ on $X^{\opn{Tan}}$ descends to a form on the quotient $X^{\opn{Tan}}/\opn{rad}(q,\kappa)$, and hence that the equivalence $F_\kappa:\Rep(\dG)\overset{\sim}\to \Rep(\dG_\varepsilon)$ from Theorem \ref{thm:1021} reduces to a unique symmetric equivalence $\bar{F}_\kappa:\Rep(\Sigma)\to \Rep(\Sigma_{\varepsilon})$ which fits into a diagram
\begin{equation}\label{eq:2163}
\xymatrix{
\Rep(\dG)\ar[rr]^{F_\kappa}\ar[d]_{\res} & & \Rep(\dG_\varepsilon)\ar[d]^{\res}\\
\Rep(\Sigma)\ar[rr]^{\bar{F}_\kappa} & & \Rep(\Sigma_{\varepsilon}).
}
\end{equation}
The functor $\bar{F}_\kappa$ is the identity on the underlying linear categories and has tensor compatibility provided by the form $\kappa$, as in Theorem \ref{thm:1021}.
\par

We now append the diagram \eqref{eq:2163} to the diagram \eqref{eq:2135} to obtain a strictly commuting diagram of braided tensor functors
\begin{equation}\label{eq:2173}
\xymatrix{
\Rep(\dG)\ar[rr]^{\opn{Fr}_\kappa}\ar[d]_{\res} & & \Rep(G_q)\ar[d]^{\res}\\
\Rep(\Sigma)\ar[rr]^{\overline{\opn{Fr}}_\kappa} & & \Rep(u_{q,\kappa}),
}
\end{equation}
where $\opn{Fr}_\kappa$ and $\overline{\opn{Fr}}_\kappa$ are both M\"uger central.  The identification of Lemma \ref{lem:2143} is now expressed as follows.

\begin{lemma}\label{lem:2181}
The natural map
$\opn{Fr}_\kappa\opn{ind}_{\Sigma}^{\dG}(\1)\to \opn{ind}_{u_{q,\kappa}}^{G_q}(\1)$ is an isomorphism.
\end{lemma}

\subsection{Fiber calculations}

Via faithful exactness of induction, Proposition \ref{prop:ind_exact}, and the identification of Lemma \ref{lem:2181}, we can apply Proposition \ref{prop:braided_bc} to obtain a calculation of the category \eqref{eq:2034}.

\begin{proposition}\label{prop:int_basechange}
For any alternating square root $\kappa$ of $\varepsilon|_{X^{\opn{Tan}}\times X^{\opn{Tan}}}$, the restriction functor $\res:\Rep(G_q)\to \Rep(u_{q,\kappa})$ reduces to an equivalence of braided tensor categories
\[
\Rep(\Sigma)\ot_{\Rep(\dG)}\Rep(G_q)\overset{\sim}\to \Rep(u_{q,\kappa}).
\]
\end{proposition}

Though it is not strictly applicable to our study, we note that a similar analysis, with Proposition \ref{prop:ind_exact}, Lemma \ref{lem:2181}, and Proposition \ref{prop:braided_bc} replaced by \cite[Theorem 10.4, Corollary 7.5]{negron} and Theorem \ref{thm:takeuchi}, yields the following calculation.

\begin{theorem}\label{thm:base_change}
The restriction functor $\Rep(G_q)\to \Rep(\bar{u}_q)$ reduces to an equivalence of $\Rep(G_q)$-module categories
\begin{equation}\label{eq:2231}
\opn{Vect}\ot_{\Rep(G^\ast_\varepsilon)}\Rep(G_q)\overset{\sim}\to \Rep(\bar{u}_q).
\end{equation}
\end{theorem}

\subsection{Comparisons with earlier works I}
\label{sect:compI}

In our formulation of things, Theorem \ref{thm:base_change} is an application of Takeuchi's theorem in conjunction with some basic results for quantum group representations.  For $q$ in which all $q_\alpha$ are of odd order, surjectivity of the restriction functor $\Rep(G_q)\to \Rep(\bar{u}_q)$ and faithful exactness of induction $\Rep(\bar{u}_q)\to \Rep(G_q)$ follow from studies of Andersen, Polo, and Wen from the 90's.  See in particular \cite[Theorem 4.6, Theorem 4.8]{andersenpolowen92}.  So the fiber calculation
\begin{equation}\label{eq:1992}
\opn{Vect}\ot_{\Rep(G^\ast_\varepsilon)}\Rep(G_q)\overset{\sim}\to \Rep(\bar{u}_q)
\end{equation}
can be deduced directly from Takeuchi \cite{takeuchi79} and the original works of Andersen-Polo-Wen \cite{andersenpolowen91,andersenpolowen92}.
\par

In work of Arkhipov and Gaitsgory from the early 2000's \cite{arkhipovgaitsgory03} the authors follow this line of reasoning to obtain a calculation \eqref{eq:1992}, or more precisely a calculation of the fiber
\[
\opn{Vect}\ot_{\Rep(\dG_\varepsilon)}\Rep(G_q),
\]
at some locus of parameters $q$.  See in particular \cite[Theorem 2.8]{arkhipovgaitsgory03} and \cite[Proof of Proposition 3.10]{arkhipovgaitsgory03}.\footnote{In \cite{arkhipovgaitsgory03} the authors' independently reproduce the results of Takeuchi \cite{takeuchi79}, so that some portion of the paper is dedicated to recovering the foundations from \cite{takeuchi79}.}  In \cite{arkhipovgaitsgory03} one finds the first suggestion that the category \eqref{eq:1992}, constructed explicitly as a category of relative Hopf modules, has a universal property amongst all abelian categories which are equipped with an appropriately $\Rep(G^\ast_{\varepsilon})$-linear map from $\Rep(G_q)$ \cite[Proposition 4.2]{arkhipovgaitsgory03}.\footnote{At lest, it is the first to the author's knowledge.}

\begin{remark}
In \cite{arkhipovgaitsgory03} the authors propose to work specifically at parameters $q$ with all $q_\alpha$ of even order, though the absence of the dual parameter $\varepsilon$ in their expression of quantum Frobenius in \cite[\S\ 1.3]{arkhipovgaitsgory03} suggests that a more precise consideration of parameters may be in order (cf.\ \cite[\S\ 3.1, 5.1]{negron21} and Section \ref{sect:examples} above).
\end{remark}

\section{Analysis of the fiber $\opn{Vect}\ot_{\Rep(\dG)}\Rep(G_q)$}
\label{sect:modular}

This section is dedicated to an analysis of the fiber category $\opn{Vect}\ot_{\Rep(\dG)}\Rep(G_q)$, along with its natural braided tensor structure.  (We recall that this category is our universal model for the category of small quantum group representations.)
\par

We show that $\opn{Vect}\ot_{\Rep(\dG)}\Rep(G_q)$ is a finite (rigid) tensor category of a particular Frobenius-Perron dimension, and give conditions under which it is ribbon and non-degenerate.  In Sections \ref{sect:smallquantumgroup} and \ref{sect:smallrep} we construct a quasi-Hopf algebra whose representations recover the fiber in question, and study the representation theory of this algebra in detail.

\subsection{The main results}

To begin, let us recall our twisted Frobenius embedding
\[
\opn{Fr}_\kappa:\Rep(\dG)\to \Rep(G_q)
\]
from Section \ref{sect:Tan}.  Since the embedding $\opn{Fr}_\kappa:\Rep(\dG)\to \Rep(G_q)$ is M\"uger central, and the forgetful functor $\Rep(\dG)\to \opn{Vect}$ is a symmetric tensor functor, the category $\opn{Vect}\ot_{\Rep(\dG)}\Rep(G_q)$ inherits a unique braided monoidal structure under which the reduction map
\[
\Rep(G_q)\to \opn{Vect}\ot_{\Rep(\dG)}\Rep(G_q)
\]
is a braided monoidal functor.  We have the following generic description of the above fiber, along with its braided monoidal structure.

\begin{theorem}\label{thm:1}
The fiber $\opn{Vect}\ot_{\Rep(\dG)}\Rep(G_q)$ is a finite (rigid) braided tensor category of Frobenius-Perron dimension
\begin{equation}\label{eq:fpdim1}
\opn{FPdim}\big(\opn{Vect}\ot_{\Rep(\dG)}\Rep(G_q)\big) = [X:X^{\opn{Tan}}]\cdot (\prod_{\gamma\in \Phi^+} l_\gamma)^2.
\end{equation}
Furthermore, this category is integral, and hence is realized as representations for a quasi-Hopf algebra whose vector space dimension is equal to \eqref{eq:fpdim1}.
\end{theorem}

We discuss quasi-Hopf realizations of the category $\opn{Vect}\ot_{\Rep(\dG)}\Rep(G_q)$ in detail in Section \ref{sect:qhopf}.
\par

In the simply-connected setting, our calculation of the Tannakian center from Theorem \ref{thm:sc} allows us to be more precise in our description of the fiber.

\begin{theorem}\label{thm:2}
Suppose that $G$ is simply-connected, that $q$ is maximally non-degenerate (Definition \ref{def:max_nondegen}), and that all of the scalar parameters $q_\alpha$ are of even order.  Then the braided tensor category $\opn{Vect}\ot_{\Rep(\dG)}\Rep(G_q)$ is non-degenerate and admits a unique ribbon structure under which the reduction functor
\[
\Rep(G_q)\to \opn{Vect}\ot_{\Rep(\dG)}\Rep(G_q)
\]
is a map of ribbon tensor categories.  Furthermore, in this case the Frobenius-Perron dimension is precisely
\begin{equation}\label{eq:fpdim2}
\opn{FPdim}\big(\opn{Vect}\ot_{\Rep(\dG)}\Rep(G_q)\big) = |Z(G)|\cdot (\prod_{\alpha\in \Delta}l_\alpha)\cdot (\prod_{\gamma\in \Phi^+} l_\gamma)^2.
\end{equation}
\end{theorem}

One can check that all parameters of the form $q=\exp(\pi i(-,-)/l)$, with $l$ divisible by the lacing number for $G$, satisfy the hypotheses of Theorem \ref{thm:2}.  We recall also that a finite braided tensor category is called \emph{modular} if it is non-degenerate and ribbon \cite{creutziggannon17,lentneretal23}.  So, Theorem \ref{thm:2} says that the fiber $\opn{Vect}\ot_{\Rep(\dG)}\Rep(G_q)$ is a modular tensor category whenever $G$ is simply-connected and $q$ is of ``even order".  One need only consider our examples from Section \ref{sect:examples} to see that the conclusions of Theorem \ref{thm:2} do not hold under weaker hypotheses, generally speaking.

\begin{remark}
One can show that $\opn{Vect}\ot_{\Rep(\dG)}\Rep(G_q)$ is non-degenerate whenever the M\"uger center agrees with the Tannakian center in $\Rep(G_q)$, and is slightly degenerate otherwise.   (By slightly degenerate we mean that the M\"uger center is identified with $\opn{sVect}$.)  This is a straightforward application of findings from \cite[Appendix A]{negron21} which we won't recall here.
\end{remark}

\subsection{Rigidity and Frobenius-Perron dimension}
\label{sect:thm1}

Fix an alternating form $\kappa$ as in Theorem \ref{thm:1021}.  We consider the discrete toral subgroup $\Sigma \subseteq \dG$ and corresponding M\"uger central embedding
\[
\overline{\opn{Fr}}_\kappa:\Rep(\Sigma)\to \Rep(u_{q,\kappa})
\]
from Section \ref{sect:vq}.

\begin{lemma}\label{lem:2159}
The restriction functor $\res:\Rep(G_q)\to \Rep(u_{q,\kappa})$ reduces to an equivalence of braided monoidal categories
\begin{equation}\label{eq:2155}
\opn{Vect}\ot_{\Rep(\dG)}\Rep(G_q)\overset{\sim}\to\opn{Vect}\ot_{\Rep(\Sigma)}\Rep(u_{q,\kappa}).
\end{equation}
\end{lemma}

\begin{proof}
This follows from the calculation
\[
\Rep(\Sigma)\ot_{\Rep(\dG)}\Rep(G_q)\overset{\sim}\to \Rep(u_{q,\kappa})
\]
of Proposition \ref{prop:int_basechange} and associativity of the products $\ot_\msc{E}$ on the $2$-category of presentable categories.  We have specifically
\[
\opn{Vect}\ot_{\Rep(\dG)}\Rep(G_q)\cong \opn{Vect}\ot_{\Rep(\Sigma)}(\Rep(\Sigma)\ot_{\Rep(\dG)}\Rep(G_q))
\]
\[
\overset{\sim}\to \opn{Vect}\ot_{\Rep(\Sigma)}\Rep(u_{q,\kappa}).
\]
\end{proof}

We now obtain Theorem \ref{thm:1}, essentially as a corollary to Lemma \ref{lem:2159}.

\begin{proof}[Proof of Theorem \ref{thm:1}]
We have just seen that the fiber $\opn{Vect}\ot_{\Rep(\dG)}\Rep(G_q)$ is equivalent to the fiber $\opn{Vect}\ot_{\Rep(\Sigma)}\Rep(u_{q,\kappa})$.  By Lemma \ref{lem:de_ftc} it follows that the monoidal category $\opn{Vect}\ot_{\Rep(\dG)}\Rep(G_q)$ is a finite tensor category.
\par

As for the Frobenius-Perron dimension, we have
\[
\opn{FPdim}(\Rep(u_{q,\kappa}))=\dim(u_{q,\kappa})= [X:\opn{rad}(q,\kappa)]\cdot (\prod_{\gamma\in \Phi^+} l_\gamma)^2
\]
via Lemma \ref{lem:dims}, and $|\Sigma|=[X^{\opn{Tan}}:\opn{rad}(q,\kappa)]$.  Hence, for the fiber we have
\[
\opn{FPdim}(\opn{Vect}\ot_{\Rep(\dG)}\Rep(G_q))=\opn{FPdim}(\opn{Vect}\ot_{\Rep(\Sigma)}\Rep(u_{q,\kappa}))
\]
\[
=\frac{\opn{FPdim}(\Rep(u_{q,\kappa}))}{|\Sigma|}=[X:X^{\opn{Tan}}]\cdot (\prod_{\gamma\in \Phi^+} l_\gamma)^2
\]
\cite[Proposition 4.26]{dgno10}.  Integrality follows from the fact that any finite tensor category which admits a surjective tensor functor from a finite integral tensor category must itself be integral \cite[Corollary 6.2.5]{egno15}, or more immediately from \cite[Corollary 4.27]{dgno10}.
\end{proof}

\subsection{Modular structures}

\begin{proof}[Proof of Theorem \ref{thm:2}]
Non-degeneracy of the fiber $\opn{Vect}\ot_{\Rep(\dG)}\Rep(G_q)$ follows from the fact that the embedding $\opn{Fr}:\Rep(\dG)\to \Rep(G_q)$ is an equivalence onto the M\"uger center in $\Rep(G_q)$ in this case, by Theorem \ref{thm:sc}, in conjunction with the standard calculus of equivariantization and de-equivariantization.  See for example \cite[Proposition 3.6]{negron21} (cf.\ \cite[Proposition 4.19]{dgno10}).
\par

For the ribbon structure, we note that the pivotal element $K_{2\rho}$ satisfies $K_{2\rho}|_{lQ}\equiv 1$, and recall that in this case $X^{\opn{Tan}}=lQ$ by Theorem \ref{thm:sc}.  So the Frobenius embedding $\opn{Fr}_{\kappa}:\Rep(\dG)\to \Rep(G_q)$ becomes a map of pivotal braided tensor categories, and hence of ribbon tensor categories.  It follows that the reduction $\opn{Vect}\ot_{\Rep(\dG)}\Rep(G_q)$ inherits a ribbon structure from the ribbon structure on $\Rep(G_q)$.
\par

To elaborate, if we adopt the explicit model
\[
\opn{Vect}\ot_{\Rep(\dG)}\Rep(G_q)=\O(\dG)\text{-mod}_{\Rep(G_q)},
\]
then it is clear that the pivotal structure on $\Rep(G_q)$ induces a pivotal structure on the fiber which is just given by the standard equivalence
\[
M\overset{\sim}\to {M^\vee}^\vee,\ \ m\mapsto ev_m,
\]
composed with multiplication by $K_{2\rho}$ \cite[Theorem 1.17]{kirillovostrik02}.  This pivotal structure induces a balanced structure $\theta_-$ on the fiber so that the reduction map
\[
\Rep(G_q)\to \opn{Vect}\ot_{\Rep(\dG)}\Rep(G_q)
\]
is a map of balanced tensor categories.  Since $\Rep(G_q)$ is ribbon, and the two endomorphisms $\theta_{(-)^\ast}$ and $\theta^\ast_-$ are natural, surjectivity of the reduction functor implies an equality $\theta_{M^\ast}=\theta^\ast_{M}$ at all compact ($=$rigid) objects $M$ in the fiber.
\par

Finally, for the calculation of the Frobenius-Perron dimension, we have directly
\[
[X:X^{\opn{Tan}}]=[P:lQ]=[P:Q]\cdot [Q:lQ]=|Z(G)|\cdot (\prod_{\alpha\in\Delta}l_\alpha).
\]
So the formula \eqref{eq:fpdim1} reduces immediately to the formula \eqref{eq:fpdim2}.
\end{proof}

\subsection{A Question: Minimal non-degenerate extensions}

In Theorem \ref{thm:2} we've identified some situations where the fiber $\opn{Vect}\ot_{\Rep(\dG)}\Rep(G_q)$ is modular, and in particular non-degenerate.  This does not occur in general however.  We record the following question regarding these slightly-degenerate situations.

\begin{question}
Suppose $G$ and $q$ are such that the fiber $\opn{Vect}\ot_{\Rep(\dG)}\Rep(G_q)$ is slightly-degnerate, i.e\ has M\"uger center equivalent to $\opn{sVect}$.  Can one construct a braided tensor embedding $F:\opn{Vect}\ot_{\Rep(\dG)}\Rep(G_q)\to \msc{D}$ into a non-degenerate braided tensor category $\msc{D}$ with
\[
\opn{FPdim}(\msc{D})=2\cdot \opn{FPdim}(\opn{Vect}\ot_{\Rep(\dG)}\Rep(G_q))\ ?
\]
If so, can one classify all such non-degenerate extensions?
\end{question}

The context here is that, in the fusion setting, such minimal non-degenerate extensions always exist.  This is a recent result of Johnson-Fryed and Reutter \cite{johnsonfryedreutter23}.  Furthermore, there are only finitely many such extensions, and there is a general theory which controls how these extensions relate to each other \cite{delaneygalindoplavnikrowell21}.  We know of no such results in the finite non-semisimple setting.  The case of $\opn{SL}(2)$ at an odd order root of unity is of special interest (see Example \ref{ex:10}).

\subsection{Comparisons with earlier works II}
\label{sect:comp}

Theorems \ref{thm:1} and \ref{thm:2} were obtained for quantum groups at particularly advantageous parameters in the author's earlier work \cite{negron21}.  One assumes in \cite{negron21} that all of the scalar parameters $q_\alpha$ are of order divisible by $4\cdot r_\alpha$, where $r_\alpha$ is the lacing number for the almost-simple factor of $G$ in which $\alpha$ appears as a simple root.  For a simple comparison, our Theorem \ref{thm:2} holds for $\opn{SL}_n$ at arbitrary even order parameters while that of \cite{negron21} only applies to $\opn{SL}_n$ at parameters of order divisible by $4$, when $n>2$.  No analog of Theorem \ref{thm:1} appears in \cite{negron21}.
\par

Aside from these specific comparisons, there are some fundamental issues which we might highlight as well.  First, in all instances in which the dual parameter $\varepsilon$ does not have trivial restriction to $X^{\opn{Tan}}$, the forgetful functor $\Rep(\dG_\varepsilon)\to \opn{Vect}$ is \emph{not} isomorphic to any symmetric fiber functor for $\Rep(\dG_\varepsilon)$.  Indeed, it's simply not a symmetric tensor functor at all.
\par

So, in this particular setting it is clear that no calculation as in Theorem \ref{thm:1}, or \ref{thm:2}, has appeared in the literature up to this point.  This is because the fiber functor which one employs in the \emph{definition} of the reduction $\opn{Vect}\ot_{\Rep(\dG)}\Rep(G_q)$ has not been articulated in prior works.  (At least, this is our understanding of the situation based on our readings of the relevant texts.)  Having established the results of \cite{negron21}, the main point of this text is to deal precisely with these kinds of subtleties, and hence to extend the findings of \cite{negron21} to arbitrary $G$ and $q$.
\par

Finally, analogs of Theorem \ref{thm:2} are well-known in the adjoint, odd order setting, and are generically attributed to Arkhipov-Gaitsgory \cite{arkhipovgaitsgory03}.  One can see, for example, work of Davydov-Etingof-Nikshych \cite[\S\ 2.3]{davydovetingofnikshych18} in this regard.  We recall this case in Section \ref{sect:lusztig_intheend} below.

\section{The small quantum group(s) $u_q$}
\label{sect:smallquantumgroup}

We construct a ``small quantum group" $u_q$ whose representations recover the fiber category $\opn{Vect}\ot_{\Rep(\dG)}\Rep(G_q)$.  The construction of this (quasi-Hopf) algebra depends on various extraneous choices, and hence $u_q$ itself is \emph{not} an invariant for the pairing of $G$ with $q$.
\par

Despite this fact, we provide a completely uniform description of the representation theory for $u_q$ in Section \ref{sect:smallrep} below.  This uniform description employs familiar tools and techniques from classical studies of quantum groups \cite{parshallwang91,lusztig89,andersenpolowen91,andersenpolowen92,andersen03,negron} in conjunction with basic information about equivariantization and de-equivariantization over finite Tannakian categories \cite{arkhipovgaitsgory03,dgno10}.

\subsection{The algebras $u_q$}
\label{sect:uqkappa}

We fix our setup as in Section \ref{sect:vq}.  So, we've fixed some alternating square root $\kappa$ of $\varepsilon|_{X^{\opn{Tan}}\times X^{\opn{Tan}}}$ and associated finite abelian subgroup $\Sigma$ in $\dG$.  As in \ref{sect:vq}, the twisted Frobenius functor $\opn{Fr}_\kappa:\Rep(\dG)\to \Rep(G_q)$ reduces to a central embedding
\[
\overline{\opn{Fr}}_\kappa:\Rep(\Sigma)\to \Rep(u_{q,\kappa}).
\]
Let us begin our construction of the algebra $u_q$.

{\it Step I, extending $\kappa$}: For $u_{q,\kappa}$, we have the surjection from the grouplikes
\[
\Lambda=\big(X/\opn{rad}(q,\kappa)\big)^\vee \to \big(X^{\opn{Tan}}/\opn{rad}(q,\kappa)\big)^\vee=\Sigma
\]
and dualized inclusion $\Sigma^\vee\to \Lambda^\vee$.  We can write the form $\kappa$ as a map of abelian groups $\kappa:X^{\opn{Tan}}\ot_{\mathbb{Z}} X^{\opn{Tan}}\to \opn{tors}(k^\times)$, and since $\kappa$ vanishes on the radical $\opn{rad}(q,\kappa)$ this form reduces to a form on the quotient $\Sigma^\vee=X^{\opn{Tan}}/\opn{rad}(q,\kappa)$.  Since $\opn{tors}(k^\times)$ is a divisible group, and hence injective over $\mathbb{Z}$, we can extend $\kappa$ to a map $\psi:X\ot_\mathbb{Z}X\to \opn{tors}(k^\times)$.  This extension of $\kappa$ also vanishes on $\opn{rad}(q,\kappa)$, and hence reduces to a bilinear form on $\Lambda^\vee$ which agrees with $\kappa$ on $\Sigma^\vee$.

{\it Step II, twisting $u_{q,\kappa}$}: Our new form $\psi$ defines a Drinfeld twist for $u_{q,\kappa}$, and we twist by this element to produce a new Hopf algebra which we denote
\[
u^{\psi}_{q,\kappa}:=u_{q,\kappa}\ \text{with new comultiplication }\Delta^\psi(x)=\psi\Delta(x)\psi^{-1}.
\]
The $R$-matrix for $u_{q,\kappa}$ also twists along $\psi$ to provide an $R$-matrix for $u^{\psi}_{q,\kappa}$ \cite[Proposition 8.3.14]{egno15} and we have the canonical braided tensor equivalence $\Rep(u_{q,\kappa})\overset{\sim}\to \Rep(u^{\psi}_{q,\kappa})$, which is just the identity as a map of linear categories along with the tensor compatibility given by the action of $(\psi)^{-1}$.  We compose this equivalence with $\overline{\opn{Fr}}_\kappa$ to obtain a central embedding
\begin{equation}\label{eq:2472}
\Rep(\Sigma)\hookrightarrow \Rep(u^{\psi}_{q,\kappa})
\end{equation}
whose underlying linear functor agrees with that of $\overline{\opn{Fr}}_\kappa$, and whose tensor compatibility is now just the identity of vector spaces.
\par

Since the underlying linear functor for $\overline{\opn{Fr}}_\kappa$ was given by restriction along the unique algebra projection $u_{q,\kappa}\to k\Sigma$ which annihilates the $E$'s and $F$'s and lifts the natural projection $\Lambda\to \Sigma$, we see that the functor \eqref{eq:2472} is also given by restriction along the projection $u^{\psi}_{q,\kappa}=u_{q,\kappa}\to k\Sigma$.  Compatibility of \eqref{eq:2472} with the forgetful functors to $\opn{Vect}$, i.e.\ triviality of the tensor compatibility, tells us that this projection from $u^\psi_{q,\kappa}$ is a map of Hopf algebras.

{\it Step III, take $\Sigma$-coinvariants}:

\begin{definition}
The algebra $u_q$ is the coideal subalgebra in $u^{\psi}_{q,\kappa}$ of left $k\Sigma$-coinvariants
\[
u_q{\color{gray}(=u_{q,\psi})}:={^{\Sigma}(u^{\psi}_{q,\kappa})}.
\]
\end{definition}

\begin{remark}
One should be clear the algebra $u_q$ is not uniquely determined by the pairing of $G$ with $q$, since there is some ambiguity in the choice of $\psi$ in general.  We deemphasize the role of $\psi$ in our notation since we avoid--with intention--any analyses of the algebra $u_q$ which depends on the choice of $\psi$ in any explicit way.
\par

To mix metaphors, one might think of $u_q$ as existing in a superposition over all choices of $\psi$.  Only under direct scrutiny will any particular choice of $\psi$ become relevant.  We perform no such direct examinations of this work.
\end{remark}

\subsection{$u_q$-representations and the fiber $\opn{Vect}\ot_{\Rep(\dG)}\Rep(G_q)$}
\label{sect:X}

Via its construction as a coideal subalgebra, the representation category $\Rep(u_q)$ is naturally a pointed module category over $\Rep(u^{\psi}_{q,\kappa})$, and we restrict along the tensor functor
\[
\Rep(G_q)\to \Rep(u_{q,\kappa})\cong \Rep(u^{\psi}_{q,\kappa})
\]
to endow $\Rep(u_q)$ with a pointed module category structure over $\Rep(G_q)$.  This pointed structure defines, and is defined by, the associated restriction functor
\[
\res_{\psi}:=\big(\Rep(G_q)\overset{\res}\to \Rep(u_{q,\kappa})\cong\Rep(u^{\psi}_{q,\kappa})\overset{\opn{restrict}}\to \Rep(u_q)\big).
\]

\begin{proposition}\label{prop:2464}
The functor $\res_{\psi}:\Rep(G_q)\to \Rep(u_q)$ reduces to an equivalence of $\Rep(G_q)$-module categories
\[
\opn{Vect}\ot_{\Rep(\dG)}\Rep(G_q)\overset{\sim}\to \Rep(u_q).
\]
\end{proposition}

\begin{proof}
As in the proof of Lemma \ref{lem:2159} we have an equivalence of $\Rep(G_q)$-module categories
\[
\opn{Vect}\ot_{\Rep(\dG)}\Rep(G_q)\overset{\sim}\to \opn{Vect}\ot_{\Rep(\Sigma)}\Rep(u_{q,\kappa})\cong \opn{Vect}\ot_{\Rep(\Sigma)}\Rep(u^{\psi}_{q,\kappa}).
\]
So it suffices to prove that restriction induces an equivalence of $\Rep(u^{\psi}_{q,\kappa})$-module categories
\[
\opn{Vect}\ot_{\Rep(\Sigma)}\Rep(u^{\psi}_{q,\kappa})=\O(\Sigma)\text{-mod}_{\Rep(u^{\psi}_{q,\kappa})}\overset{\sim}\to \Rep(u_q).
\]
However, this is just an application of Lemma \ref{lem:fin_basechange}.
\end{proof}

\begin{remark}
We note that $u_{q,\kappa}=u^{\psi}_{q,\kappa}$ as algebras, and that the underlying linear equivalence in the tensor equivalence $\Rep(u_{q,\kappa})\cong\Rep(u^{\psi}_{q,\kappa})$ is just the identity.  So the above restriction functor $\res_\psi$ is, at a linear level, just restriction along the algebra inclusion $u_q\subseteq u_{q,\kappa}\to \hat{U}_q$.  The form $\psi$ only provides $\res_\psi$ with its module category structure over $\Rep(G_q)$.
\end{remark}

\subsection{Quasi-Hopf structures on $u_q$}
\label{sect:qhopf}

We begin with the following generic result from Angiono, Galindo, and Pereira \cite{angionogalindopereira14}.

\begin{theorem}[{\cite[Theorem 2.8]{angionogalindopereira14}}]\label{thm:2244}
Let $A$ be a finite-dimensional Hopf algebra, $\Pi$ be a finite group, and $A\to k\Pi$ be a surjective Hopf map.  Suppose the corresponding tensor functor $\Rep(\Pi)\to \Rep(A)$ admits a central structure.  Then the left coinvariant subalgebra $B={^\Pi{A}}$ admits a quasi-Hopf structure under which the induced map of $\Rep(A)$-module categories
\begin{equation}\label{eq:2527}
\opn{Vect}\ot_{\Rep(\Pi)}\Rep(A)\to \Rep(B)
\end{equation}
enhances to a map of tensor categories.  Furthermore, the map \eqref{eq:2527} is an equivalence.
\end{theorem}

We apply Theorem \ref{thm:2244} to deduce a quasi-Hopf structure on the subalgebra $u_q\subseteq u_{q,\kappa}$ under which the restriction map $\Rep(u_{q,\kappa})=\Rep(u^{\psi}_{q,\kappa})\to \Rep(u_q)$ reduces to an equivalence of tensor categories
\[
\opn{Vect}\ot_{\Rep(\Sigma)}\Rep(u_{q,\kappa})\cong \opn{Vect}\ot_{\Rep(\Sigma)}\Rep(u^{\psi}_{q,\kappa})\overset{\sim}\to\Rep(u_q).
\]
Via the braided structure on the above fiber, we see furthermore that $u_q$ admits a quasitriangular structure under which the above equivalence is braided monoidal.  We now apply the equivalence of Lemma \ref{lem:2159} to obtain the following.

\begin{theorem}\label{thm:quasiHopf}
The algebra $u_q$ admits a quasitriangular quasi-Hopf structure under which the restriction functor $\res_{\psi}:\Rep(G_q)\to \Rep(u_q)$ enhances to a map of braided tensor categories.  Furthermore, this braided tensor functor induces an equivalence of braided tensor categories
\begin{equation}\label{eq:2504}
\opn{Vect}\ot_{\Rep(\dG)}\Rep(G_q)\overset{\sim}\to \Rep(u_q)
\end{equation}
\end{theorem}

As the algebra $u_q$ is non-canonical, and the quasi-Hopf structure on this algebra is even noner-canonicaler, we are happy to simply establish the existence of such a quasi-Hopf structure on $u_q$ and corresponding calculation \eqref{eq:2504}. We leave further investigations in this quasi-Hopf structure to the interested reader.  (Cf.\ \cite{gainutdinovrunkel17,creutziggainutdinovrunkel20,gainutdinovlentnerohrmann,creutziglentnerrupert2}.)

\subsection{Recovering Lusztig at odd orders \cite{arkhipovgaitsgory03,davydovetingofnikshych18}}
\label{sect:lusztig_intheend}

We consider $G$ of adjoint type, and an odd integer $l$ which is coprime to both the lacing number and the determinant of the Cartan matrix for $G$ (Example \ref{ex:odd}).  Let us fix the quantum parameter $q=\exp(2\pi i (-,-)/l)$.
\par

In this case the dual parameter $\varepsilon$ for $G^\ast$ is trivial, so that we may take $\kappa$ to be identically $1$.  Following the construction of Section \ref{sect:uqkappa}, we now take $\psi=1$ to obtain a distinguished choice for the coideal subalgebra $u_q$ in $u_{q,\kappa}$.  Furthermore, via our coprimeness assumptions we see that $q$ vanishes on $lQ$ and reduces to a non-degenerate form on the quotient $Q/lQ$.  This forces $X^{\opn{Tan}}=X^{\text{M\"ug}}=lQ$, and calculates the grouplikes in $u_q$ as
\[
(Q/lQ)^\vee=\{ q(\lambda,-):\lambda\in Q\}=\langle K_\alpha:\alpha\in\Delta\rangle\cong (\mathbb{Z}/l\mathbb{Z})^\Delta.
\]

Now, since all of the grouplikes $K_\alpha$ live in $u_q$, and since the $E$'s and $F$'s are annihilated by the Hopf algebra map $u_{q,\kappa}\to k\Sigma$, we see that all $E_\alpha$ and $F_\alpha$ live in $u_q$ as well.  This, along with the dimension calculation \eqref{eq:fpdim1}, provides an identification of our small quantum group
\[
\begin{array}{rl}
u_q&=\text{the subalgebra in $u_{q,\kappa}$ generated by the $K_\alpha$, $E_{\alpha}$, and $F_\alpha$}\vspace{2mm}\\
&=\text{the subalgebra in $\hat{U}_q$ generated by the $K_\alpha$, $E_{\alpha}$, and $F_\alpha$}\vspace{2mm}\\
&=\text{Lusztig's original finite-dimensional algebra from \cite{lusztig90,lusztig90II}}.
\end{array}
\]

Furthermore, one sees that the $R$-matrix for the big quantum group lies in $u_q\ot u_q$ in this case, so that the usual Hopf structure on $u_q$ gives $\Rep(u_q)$ a braided tensor structure under which the restriction functor $\res:\Rep(G_q)\to \Rep(u_q)$ is braided monoidal.  It follows that the induced functor from the fiber
\[
\opn{Vect}\ot_{\Rep(\dG)}\Rep(G_q)\to \Rep(u_q)
\]
is a map of braided tensor categories, and hence an equivalence of braided tensor categories by Proposition \ref{prop:2464}.

\subsection{Agreement with Gainutdinov-Lentner-Ohrmann}
\label{sect:glo}

In \cite{gainutdinovlentnerohrmann} Gainutdinov, Lentner, and Ohrmann construct factorizable quasi-Hopf algebras for any pairing of a semisimple algebraic group $G$ with a form $q$ \cite[\S 4, 5]{gainutdinovlentnerohrmann}, just as in the present text.  For the moment, let's denote their quasi-Hopf algebras by $u^{\opn{GLO}}_q$.

For universality reasons, our quasi-Hopf algebra $u_q$ has the same representation category as that of $u^{\opn{GLO}}_q$.  In particular, we have a braided tensor equivalence
\[
\opn{Vect}\ot_{\Rep(\dG)}\Rep(G_q)\cong \Rep(u_q)\cong \Rep(u^{\opn{GLO}}_q).
\]
This equivalence is forced by the conclusions of \cite[Theorem 6.7 (2)]{gainutdinovlentnerohrmann} and Lemma \ref{lem:2159}, in particular.  So the presentation from this section agrees with the earlier constructions of ``modular small quantum groups" from Gainutdinov-Lentner-Ohrmann, at least at a categorical level.
\par

One can see \cite{gainutdinovlentnerohrmann} for an extensive, and remarkably hands-on, analysis of quasi-Hopf algebras which are like, if not equal to, our algebra $u_q$.

\section{Small quantum representations}
\label{sect:smallrep}

We consider representations for our small quantum group $u_q$.  Even though, as explained in Section \ref{sect:smallquantumgroup}, the underlying algebra for $u_q$ is not uniquely determined, we obtain a uniform and fairly explicit description of the representation category $\Rep(u_q)$ via its consistent identifications with the fibers $\opn{Vect}\ot_{\Rep(\dG)}\Rep(G_q)$ and $\opn{Vect}\ot_{\Rep(\Sigma)}\Rep(u_{q,\kappa})$.  We record some basic information below.

\subsection{Simples, projectives, and restriction}
\label{sect:smallrep_1}

\begin{lemma}[{cf.\ \cite[proof of Proposition 4.26]{dgno10}}]\label{lem:simple_kappa}
\begin{enumerate}
\item Any simple $L$ in $\Rep(u_{q,\kappa})$ restricts to a simple representation in $\Rep(u_q)$.
\item Two simples $L$ and $L'$ in $\Rep(u_{q,\kappa})$ have isomorphic restrictions in $\Rep(u_q)$ if and only if $L'\cong \sigma\ot L$ for some invertible simple $\sigma$ in $\Rep(\Sigma)$.
\item Restriction induces a bijection
\begin{equation}\label{eq:2518}
\opn{Irrep}(u_{q,\kappa})/\Sigma^\vee\overset{\sim}\to \opn{Irrep}(u_q).
\end{equation}
\end{enumerate}
\end{lemma}

To be clear, in the above formula \eqref{eq:2518} we consider the identification $\Sigma^\vee=\opn{Irrep}(\Sigma)$ and let $\Sigma^\vee$ act on $\opn{Irrep}(u_{q,\kappa})$ via the translations $\sigma\ot-$.

\begin{proof}
Under the equivalence
\[
\opn{Vect}\ot_{\Rep(\Sigma)}\Rep(u_{q,\kappa})=\O(\Sigma)\text{-mod}_{\Rep(u_{q,\kappa})}\overset{\sim}\to \Rep(u_q)
\]
the restriction functor is identified with the free module functor
\[
\O(\Sigma)\ot-:\Rep(u_{q,\kappa})\to \O(\Sigma)\text{-mod}_{\Rep(u_{q,\kappa})}.
\]
We recall that the simples $L(\mu)$ in $\Rep(u_{q,\kappa})$ are classified by their highest weights, and that these weights are elements in the additive group $X/\opn{rad}(q,\kappa)=\Lambda^\vee$.  For any element $\sigma\in X^{\opn{Tan}}/\opn{rad}(q,\kappa)=\Sigma^\vee$ the tensor translations are given by the formula $\sigma\ot L(\mu)=L(\mu +\sigma)$.
\par

Take $\O=\O(\Sigma)$.  For (1) we consider the free module
\[
\O\ot L(\mu)=\oplus_{\sigma\in \Sigma^\vee} \big(\sigma\ot L(\mu)\big)=\oplus_{\sigma\in \Sigma^\vee}L(\mu+\sigma).
\]
As a $u_{q,\kappa}$-representation, this object is semisimple, and each simple summand in $\O\ot L(\mu)$ occurs with multiplicity $1$.  Hence any nonzero $\O$-submodule $W$ in $\O\ot L(\mu)$ contains some simple summand $L(\mu')$, where $\mu'=\mu+\sigma$ for some $\sigma$, and thus contains the image of the action map
\[
\O\ot L(\mu')\to \oplus_{\sigma\in \Sigma^\vee} \sigma\ot L(\mu).
\]
But this action map is an isomorphism, so that $W=\O\ot L(\mu)$ necessarily.  It follows that for each simple $L$ in $\Rep(u_{q,\kappa})$ the associated free module $\O\ot L$ is simple in $\O\text{-mod}_{\Rep(u_{q,\kappa})}$.
\par

For (2), it is clear that there is an identification of free modules $\O\ot L=\O\ot (\sigma\ot L)$ for each $\sigma\in \Sigma^\vee$.  Furthermore, if the objects $\O\ot L(\mu)$ and $\O\ot L(\mu')$ agree as $\O$-modules then they share a simple summand over $u_{q,\kappa}$.  But this occurs if and only if $\mu$ and $\mu'$ are $\Sigma^\vee$-translates of each other, and hence if and only if $L(\mu')=\sigma\ot L(\mu)$ for some $\sigma$.
\par

For (3), any simple object $W$ in $\O\text{-mod}_{\Rep(u_{q,\kappa})}$ admits a non-zero map from some free module $\O\ot L\to W$ with $L$ simple over $u_{q,\kappa}$.  Since both $\O\ot L$ and $W$ are simple, this map must be an isomorphism.  If follows that the free module functor induces a surjection
\[
\O\ot-:\opn{Irrep}(u_{q,\kappa})\to \{\text{Simples in }\O\text{-mod}_{\Rep(u_{q,\kappa})}\}.
\]
By (2) this surjection reduces to a bijection
\[
\opn{Irrep}(u_{q,\kappa})/\Sigma^{\vee}\overset{\sim}\to \{\text{Simples in }\O\text{-mod}_{\Rep(u_{q,\kappa})}\}.
\]
\end{proof}

\begin{lemma}\label{lem:proj_kappa}
For any indecomposable projective $P$ in $\Rep(u_{q,\kappa})$, $P$ has projective and indecomposable restriction to $\Rep(u_q)$.  Furthermore, restriction induces a bijection
\begin{equation}\label{eq:2564}
\opn{IndecProj}(u_{q,\kappa})/\Sigma^\vee\overset{\sim}\to \opn{IndecProj}(u_q).
\end{equation}
\end{lemma}

\begin{proof}
First note that $u_{q,\kappa}$ is projective over $u_q$, by Skryabin's theorem \cite[Theorem 6.1]{skryabin07} for example.  So projectives in $\Rep(u_{q,\kappa})$ restrict to projectives in $\Rep(u_q)$, and we need only deal with indecomposability and the proposed bijection \eqref{eq:2564}.
\par

We take $\O=\O(\Sigma)$ and proceed as in the proof of Lemma \ref{lem:simple_kappa}.  For any indecomposable projective $P$ we have a uniquely associated simple $L$ which admits a surjection $P\to L$.  Hence for any such $P$ the free module $\O\ot P$ has cosocle $\O\ot L$ over $u_{q,\kappa}$.  It follows that the surjection $\O\ot P\to \O\ot L$ is the unique $\O$-module map from $\O\ot P$ onto a simple in the category
\[
\opn{Vect}\ot_{\Rep(\Sigma)}\Rep(u_{q,\kappa})=\O\text{-mod}_{\Rep(u_{q,\kappa})}.
\]
In particular, $\O\ot P$ has simple cosocle, as an $\O$-module, and is therefore indecomposable as an $\O$-module.  The second claim now follows by Lemma \ref{lem:simple_kappa} (3).
\end{proof}

\begin{lemma}\label{lem:fjlks}
A $u_{q,\kappa}$-representation $W$ is projective in $\Rep(u_{q,\kappa})$ if and only if it restricts to a projective $u_q$-representation.
\end{lemma}

\begin{proof}
The fact that projectives in $\Rep(u_{q,\kappa})$ restrict to projectives in $\Rep(u_q)$ follows by Lemma \ref{lem:proj_kappa}.  Now, if $W$ has projective restriction to $\Rep(u_q)$, then the free module $\O\ot W$ is projective in $\O\text{-mod}_{\Rep(u_{q,\kappa})}$.  From the description of projectives in this module category provided in Lemma \ref{lem:proj_kappa}, if follows that $\O\ot W$ is isomorphic to a free module $\O\ot P$ with $P$ projective in $\Rep(u_{q,\kappa})$.  But now, $W$ is a summand of $\O\ot W\cong \O\ot P$ as a $u_{q,\kappa}$-representation, and $\O\ot P$ is projective in $\Rep(u_{q,\kappa})$, so that $W$ must be projective.
\end{proof}

We recall that any finite tensor category is Frobenius \cite[Proposition 2.3]{etingofostrik04}.  It follows that projectives and injectives agree in $\Rep(u_q)$.  We now apply Lemma \ref{lem:fjlks} and Theorem \ref{thm:1359} to observe the following.

\begin{proposition}
An object in $\Rep(G_q)$ is projective (equivalently injective) if and only if its restriction to $\Rep(u_q)$ is projective (equivalently injective).
\end{proposition}

\subsection{Representations via the toral algebra $\dot{\mbf{u}}_q$}
\label{sect:smallrep_2}

One can provide alternate versions of Lemmas \ref{lem:simple_kappa} and \ref{lem:proj_kappa} which replace the small quantum algebra $u_{q,\kappa}$ with the torally extended small quantum group $\dot{\mbf{u}}_q$ from \cite[\S\ 36.2]{lusztig93}.
\par

One sees, essentially from the Steinberg decompositions for simples in $\Rep(G_q)$ \cite[Theorem 7.4]{lusztig89} \cite[Theorem 8.7]{negron}, that all of the simples in $\Rep(u_{q,\kappa})$ are restricted from simples over the toral algebra $\dot{\mbf{u}}_q$.  Furthermore, any simple in $\Rep(\dot{\mbf{u}}_q)$ restricts to a simple in $\Rep(u_{q,\kappa})$.  So we obtain, via Lemmas \ref{lem:simple_kappa} and \ref{lem:proj_kappa}, a classification of simples and indecomposable projectives in $\Rep(u_q)$ via restriction from $\Rep(\dot{\mbf{u}}_q)$;
\[
\opn{Irrep}(\dot{\mbf{u}}_q)/X^{\opn{Tan}}\overset{\sim}\to \opn{Irrep}(u_q),\ \ \opn{IndecProj}(\dot{\mbf{u}}_q)/X^{\opn{Tan}}\overset{\sim}\to \opn{IndecProj}(u_q).
\]
This description is likely preferable to that of Lemma \ref{lem:simple_kappa}, as the algebra $u_{q,\kappa}$ is not a standard object of study.

\subsection{Simples and highest weights}

We recall that $u_q$ is constructed as a subalgebra in a twisting of $u_{q,\kappa}$ along some character $\psi$ on the lattice $X$.  Since the algebra $u^{\psi}_{q,\kappa}$ is obtained from $u_{q,\kappa}$ via such a toral twisting, we have the twisted non-negative Borel $(u^{\psi}_{q,\kappa})^{\geq 0}=(u^{\geq 0}_{q,\kappa})^\psi$ in $u^{\psi}_{q,\kappa}$.  This small Borel is just equal to $u^{\geq 0}_{q,\kappa}$ as an algebra.  We similarly have the non-positive subalgebra and furthermore the negative subalgebra $(u^{\psi}_{q,\kappa})^-$ of left $k\Lambda$-coinvariant elements in $(u^{\psi}_{q,\kappa})^{\leq 0}$, where $\Lambda=(X/\opn{rad}(q,\kappa))^\vee$ is the group of grouplikes in $u^\psi_{q,\kappa}$.
\par

This negative subalgebra lies in $u_q$, so that we may take $u_q^-:=(u^{\psi}_{q,\kappa})^-$.  We define also $u_q^{\geq 0}$ to be the $k\Sigma$-coinvariants in $(u^{\psi}_{q,\kappa})^{\geq 0}$ to obtain an almost-triangular decomposition
\[
u_q^{\geq 0}\ot u_q^-\overset{\sim}\to u_q.
\]
\par

Via these subalgebras, and associated triangular decoposition, one obtains a standard analysis of simple $u_q$-representations in terms of highest weights.  These highest weights live in the truncated character group $X/X^{\opn{Tan}}$, and one sees that restriction from $\Rep(\dot{\mbf{u}}_q)$ sends the simple $L(\lambda)$ of highest weight $\lambda\in X$ to the corresponding simple $L(\bar{\lambda})$ in $\Rep(u_q)$ of highest weight $\bar{\lambda}\in X/X^{\opn{Tan}}$.

\begin{lemma}
Simple $u_q$-representations are labeled by their highest weights, and there is bijection
\[
X/X^{\opn{Tan}}\overset{\sim}\to \opn{Irrep}(u_q).
\]
\end{lemma}

\subsection{Blocks}

From the identification of Theorem \ref{thm:quasiHopf} one finds that the blocks in $\Rep(u_q)$ are identified with the $\Rep(\dG)$-blocks in $\Rep(G_q)$.  Specifically, $\Rep(G_q)$ decomposes into a sum of indecomposable, orthogonal $\Rep(\dG)$-submodule categories and we consider the collection $\opn{Bl}_{\dG}(\Rep(G_q))$ of these indecomposable submodules.

\begin{proposition}[{cf.\ \cite[Proposition 5.2]{arkhipovgaitsgory03}}]\label{prop:blocks}
Consider the restriction functor $\res_{\psi}:\Rep(G_q)\to \Rep(u_q)$ from Section \ref{sect:X}.  Intersecting with the image of $\Rep(G_q)$ defines a bijection
\[
\opn{Bl}_{\dG}(\Rep(G_q))\overset{\sim}\to\opn{Bl}(\Rep(u_q)).
\]
\end{proposition}

One obtains this result by identifying $\Rep(G_q)$ with the category of $\dG$-equivariant objects in $\Rep(u_q)$, and noting that the natural $\dG$-action on $\Rep(u_q)$ preserves all of the blocks.  See Theorem \ref{thm:act} below and Corollary \ref{cor:blocks_action} (cf.\ \cite[\S\ 5]{arkhipovgaitsgory03}).

\begin{remark}
It is not the case that all of the (usual) blocks in $\Rep(G_q)$ are determined by the linkage principle \cite{andersen92,andersen03}.  In particular, one can see by considering the $\Rep(G^\ast_\varepsilon)$-orbit of the Steinberg representation in $\Rep(G_q)$, and the Steinberg decompositions of simple representations, that $\Rep(G_q)$ has infinitely many blocks \cite{andersen18}.  So, the first claim of \cite[Proposition 5.2]{arkhipovgaitsgory03} is not accurate.  The analysis of the principle block in \cite[Proposition 5.2]{arkhipovgaitsgory03} should be accurate however, at least when $q$ is of not-too-small order.
\par

In considering the $\Rep(\dG)$-action on $\Rep(G_q)$, we expect that the $\Rep(\dG)$-blocks in $\Rep(G_q)$ \emph{are} determined by the linkage principle however, so that the blocks in $\Rep(u_q)$ are determined by this principle as well.  This point is likely deducible to experts.  (Cf.\ \cite{donkin80,vay}.)
\end{remark}

\section{The canonical $\dG$-action on $\Rep(u_q)$}
\label{sect:dG_action}

We conclude by recalling a calculus of equivariantization and de-equivariantization \cite{arkhipovgaitsgory03,dgno10} which intrinsically links the braided tensor categories of big and small quantum group representations.

\subsection{The happenings}

Throughout this section we let $\msc{M}\ot_k\msc{N}$ denote the linear product $\msc{M}\ot_{\opn{Vect}}\msc{N}$ for presentable $k$-linear categories $\msc{M}$ and $\msc{N}$.
\par

Consider an affine algebraic group $H$, and its category of quasicoherent sheaves $\QCoh(H)$.  This category is symmetric monoidal under the product $\ot_{\O_H}$, so that $\QCoh(H)$ can be considered as a commutative algebra object in the $2$-category of presentable linear categories.  We have
\[
\QCoh(H)\ot_k\QCoh(H)=\QCoh(H\times H)
\]
and pullback along the group structure map $m:H\times H\to H$ provides $\QCoh(H)$ with a coassociative map of algebra objects
\[
m^\ast:\QCoh(H)\to \QCoh(H)\ot_k\QCoh(H).
\]
In this way $\QCoh(H)$ becomes a commutative Hopf algebra in presentable linear categories.

\begin{definition}
A rational action of $H$ on a presentable linear category $\msc{A}$ is a choice of a $\QCoh(H)$-coaction on $\msc{A}$.  When $\msc{A}$ is a (braided) tensor category, we say $H$ acts by (braided) tensor automorphisms on $\msc{A}$ if the coaction map
\[
\rho:\msc{A}\to \QCoh(H)\ot_k\msc{A}
\]
is endowed with the structure of a map of (braided) monoidal categories.
\end{definition}

\begin{remark}
The above defintion is adapted from unpublished notes of Dennis Gaitsgory.
\end{remark}

As we explain below--at least in the quantum group setting--for any $\Rep(H)$-module category $\msc{M}$ the fiber category $\opn{Vect}\ot_{\Rep(H)}\msc{M}$ inherits a natural rational action of $H$.  In our particular situation, we obtain a natural action of the dual group $\dG$ on our category of small quantum group representations.

\begin{theorem}[\cite{arkhipovgaitsgory03}]\label{thm:act}
The fiber $\opn{Vect}\ot_{\Rep(\dG)}\Rep(G_q)\cong\Rep(u_q)$ admits a rational action of $\dG$ by braided tensor automorphisms.  Furthermore, the restriction functor induces a braided tensor equivalence
\begin{equation}\label{eq:2735}
\Rep(G_q)\overset{\sim}\to (\opn{Vect}\ot_{\Rep(\dG)}\Rep(G_q))^{\dG}\cong \Rep(u_q)^{\dG}
\end{equation}
onto the (non-full) subcategory of $\dG$-equivariant representations in $\Rep(u_q)$.
\end{theorem}

\begin{proof}
In the language and notations of \cite{negron21}, $\opn{Vect}\ot_{\Rep(\dG)}\Rep(G_q)$ is the so-called de-equivariantization $\Rep(G_q)_{\dG}$ of $\Rep(G_q)$ along the central embedding $\opn{Fr}_{\kappa}:\Rep(\dG)\to \Rep(G_q)$.  So the result follows by \cite[Proposition 4.4]{arkhipovgaitsgory03} \cite[Proposition A.2]{negron21}.  (See also \cite[\S\ 4]{dgno10} \cite[Appendix B]{gannonnegron}.)
\end{proof}

\subsection{Theorem \ref{thm:act} in explicit terms}
\label{sect:exp}

For those who might not be as familiar with algebraic group actions on categories, let us say few words which will hopefully clarify the claims of Theorem \ref{thm:act}.
\par

In the most explicit terms, $\opn{Vect}\ot_{\Rep(\dG)}\Rep(G_q)$ is realized as the category of modules over the commutative algebra object $\O(\dG)$ in $\Rep(G_q)$, where we ignore any subtleties about twistings etc.  Let us take
\[
\msc{C}=\opn{Vect}\ot_{\Rep(\dG)}\Rep(G_q),
\]
interpreted explicitly as this category of $\O(\dG)$-modules.
\par

If we let $R$ denote $\O(\dG)$, considered as a trivial algebra object in $\Rep(\dG)$, then the Hopf structure on $R$ in $\opn{Vect}$ lifts to an $R$-comodule algebra structure on $\O(\dG)$ in $\Rep(\dG)$, or in $\Rep(G_q)$,
\[
\Delta:\O(\dG)\to R\ot \O(\dG).
\]
For any $\O(\dG)$-module $M$ in $\Rep(G_q)$ we define the new $\O(\dG)$-module $\rho(M):=R\ot M$ on which $\O(\dG)$ acts diagonally $f\cdot (r\ot m)=(f_1 r)\ot (f_2m)$.  The object $\rho(M)$ also has a left $R$-action via the first factor, which gives it the structure of an $R$-module in $\msc{C}$.
\par

The construction $\rho(M)$ is functorial, and so provides us with a braided monoidal functor to the base change
\[
\rho:\msc{C}\to R\text{-mod}_{\msc{C}}=\QCoh(\dG)\ot_k \msc{C}.
\]
This defines our rational $\dG$-action on the fiber category.
\par

As for the claimed equivalence \eqref{eq:2735}, for any $M$ in $\msc{C}$ we have the ``trivial" object $R\ot M$ in $R\text{-mod}_{\msc{C}}$ with $R$ acting through the first factor and $\O(\dG)$ acting through the second factor.  By definition, a $\dG$-equivariant structure on an object $M$ in $\msc{C}$ is the choice of a map $\rho_M:M\to \rho(M)$ of $\O(\dG)$-modules in $\Rep(G_q)$ which induces an isomorphism of $R$-modules
\[
R\ot M\overset{\sim}\to \rho(M),\ \ r\ot m\mapsto r\cdot \rho_M(m).
\]
We let $\msc{C}^{\dG}$ denote the non-full monoidal subcategory of $\dG$-equivariant objects in $\msc{C}$.
\par

We recall that the structure map for $\msc{C}$ is just the free module map
\[
\Rep(G_q)\to \msc{C},\ \ V\mapsto \O(\dG)\ot V,
\]
and note that the $R$-coaction on $\O(\dG)$ gives $\O(\dG)\ot V$ an equivariant structure over $\dG$.  So we restrict the codomain to obtain a braided monoidal functor
\[
\Rep(G_q)\to \msc{C}^{\dG}.
\]
The fact that this functor is an equivalence essentially follows by the fundamental theorem of Hopf modules \cite[\S\ 1.9]{montgomery93}.  We've now recovered Theorem \ref{thm:act}.

\subsection{Theorem \ref{thm:act} in abstract terms}

The unit map $\opn{Vect}\to \QCoh(\dG)$ calculates the self-intersection of the unique $k$-point in $B\dG$ as
\[
\opn{Vect}\ot_{\Rep(\dG)}\opn{Vect}\overset{\sim}\to \QCoh(\dG).
\]
We interpret this self-intersection as a symmetric monoidal category over $\Rep(\dG)$ via the trivial action, i.e.\ with $\Rep(\dG)$ acting on $\QCoh(\dG)$ through the fiber functor $\Rep(\dG)\to \opn{Vect}$.
\par

We have then
\[
\begin{array}{l}
(\opn{Vect}\ot_{\Rep(\dG)}\opn{Vect})\ot_{\Rep(\dG)}\Rep(G_q)\\
\hspace{2cm}\cong
(\opn{Vect}\ot_{\Rep(\dG)}\opn{Vect})\ot_k \opn{Vect}\ot_{\Rep(\dG)}\Rep(G_q)\\
\hspace{2cm}\cong \QCoh(\dG)\ot_k(\opn{Vect}\ot_{\Rep(\dG)}\Rep(G_q)).
\end{array}
\]
Hence applying $\opn{Vect}\ot_{\Rep(\dG)}-$ to the structure map $\Rep(G_q)\to \opn{Vect}\ot_{\Rep(\dG)}\Rep(G_q)$ yields a coaction
\[
\rho:\opn{Vect}\ot_{\Rep(\dG)}\Rep(G_q)\to \QCoh(\dG)\ot_k(\opn{Vect} \ot_{\Rep(\dG)}\Rep(G_q)).
\]
In this way we recover the $\dG$-action of Section \ref{sect:exp} in a manner which operates purely at the level of presentable categories, and so recover the first claim of Theorem \ref{thm:act}.
\par

We do not know how to show that the induced functor
\[
\Rep(\dG)\to (\opn{Vect}\ot_{\Rep(\dG)}\Rep(G_q))^{\dG}
\]
is an equivalence in this abstract setting, i.e.\ without choosing an explicit model for the fiber $\opn{Vect}\ot_{\Rep(\dG)}\Rep(G_q)$, though this may be an accessible computation for the adventurous reader.

\appendix

\section{Proof of Theorem \ref{thm:1021}}
\label{sect:A}

\begin{proof}[Proof of Theorem \ref{thm:1021}]
It suffices to consider the case where $q=\varepsilon$, $G=\dG$, and $X=X^{\opn{Tan}}$.  We then have $\varepsilon^2|_{X\times X}=1$ and all $\varepsilon_\alpha=\pm 1$.  We have the normalized killing form $(-,-)$ on $X^{\opn{Tan}}$, and note that the $R$-matrix for $\Rep(\dG_\varepsilon)$ reduces to the form $\varepsilon$ in this case.  Recall that $\kappa$ is a bilinear form on $X^{\opn{Tan}}$ which satisfies $\kappa^2=\varepsilon$ and $\kappa(\lambda,\lambda)=1$ at all $\lambda$.  Since $\varepsilon=\varepsilon^{-1}$ we have $\kappa^{-2}=\varepsilon$ as well.
\par

Consider the generators $e_\alpha$, $f_\alpha$ for $U_\varepsilon(\check{\mfk{g}})$, and the Drinfeld twist $\Rep(\dG_{\varepsilon})_{\kappa}$ of $\Rep(\dG_{\varepsilon})$ via the given form $\kappa$.  The twist is precisely the linear category $\Rep(\dG_{\varepsilon})$ equipped with the new tensor product $V\ot^{\kappa}W$ whose underlying $X^{\opn{Tan}}$-graded vector space is $V\ot W$ and whose $e_\alpha$ and $f_\alpha$ actions are given by
\[
\begin{array}{l}
e_\alpha\cdot_{\kappa}(v\ot w)\vspace{2mm}\\
:=\kappa(\alpha+\lambda,\mu)\kappa^{-1}(\lambda,\mu)e_\alpha v\ot w+\kappa(\lambda,\alpha+\mu)\kappa^{-1}(\lambda,\mu)\varepsilon^{(\alpha,\mu)}v\ot e_\alpha w\vspace{2mm}\\
=\kappa(\alpha,\mu)e_\alpha v\ot w+\varepsilon^{(\alpha,\mu)}\kappa(\lambda,\alpha)v\ot e_\alpha w\vspace{2mm}\\
=\kappa(\alpha,\mu)e_\alpha v\ot w+\kappa^{-1}(\lambda,\alpha)v\ot e_\alpha w\vspace{2mm}\\
=\kappa(\alpha,\mu)e_\alpha v\ot w+\kappa(\alpha,\lambda)v\ot e_\alpha w
\end{array}
\]
and
\[
\begin{array}{rl}
f_\alpha\cdot_{\kappa}(v\ot w) &:=\varepsilon^{(-\alpha,\mu)}\kappa(-\alpha,\mu)f_\alpha v\ot w+\kappa(\lambda,-\alpha)v\ot f_\alpha w\vspace{2mm}\\
&=\kappa(\alpha,\mu)f_\alpha v\ot w+\kappa(\alpha,\lambda) v\ot f_\alpha w,
\end{array}
\]
where $\lambda=\deg(v)$ and $\mu=\deg(w)$.  The trivial vector space symmetry $v\ot w\mapsto w\ot v$ endows $\Rep(\dG_\varepsilon)_\kappa$ with a symmetric structure under which the identity functor enhances to a symmetric tensor equivalence
\[
\opn{Tw}_\kappa:\Rep(\dG_{\varepsilon})_{\kappa}\overset{\sim}\to \Rep(\dG_\varepsilon)
\]
with tensor compatibility
\[
\opn{Tw}_\kappa(V)\ot \opn{Tw}_\kappa(W)\to \opn{Tw}_\kappa(V\ot W),\ \ v\ot w\mapsto \kappa(\lambda,\mu)v\ot w.
\]
\par

Consider now the characters $M_{\alpha}:X^{\opn{Tan}}\to k^\times$, $M_\alpha(\lambda):=\kappa^{-1}( \alpha,\lambda)$, and the normalized vectors $\mbf{e}_\alpha=M_\alpha e_\alpha$ and $\mbf{f}_\alpha=\varepsilon_{\alpha}M_\alpha f_\alpha$.  After normalization these elements act as primitive operators on the products $V\ot_\kappa W$, i.e.\ satisfy
\[
\mbf{e}_\alpha\cdot_\kappa(v\ot w)=\mbf{e}_\alpha \cdot v\ot w+v\ot \mbf{e}_\alpha \cdot w\ \ \text{and}\ \ \mbf{f}_\alpha\cdot_\kappa(v\ot w)=\mbf{f}_\alpha \cdot v\ot w+ v\ot \mbf{f}_\alpha \cdot w.
\]
We claim that the normalized vectors $\mbf{e}_\alpha$ and $\mbf{f}_\alpha$ satisfy the classical Serre relations for the Lie algebra $\check{\mfk{g}}$.  We first check the relations between positive and negative generators, then check the remaining relations amongst the $\mbf{e}$'s and $\mbf{f}$'s.
\par

For the commutators of the $\mbf{e}_\alpha$ and $\mbf{f}_\alpha$ we have
\[
[\mbf{e}_\alpha,\mbf{f}_\alpha]\cdot v=\varepsilon_{\alpha}M^2_\alpha[e_\alpha,f_\alpha]\cdot v=\varepsilon_{\alpha}K_\alpha[e_\alpha,f_\beta]\cdot v
\]
\begin{equation}\label{eq:comm1}
=\varepsilon_{\alpha}K_\alpha\binom{\langle\lambda,\alpha\rangle}{l_\alpha}_{\varepsilon_{\alpha}}\cdot v=\varepsilon_{\alpha}^{\langle \lambda,\alpha\rangle}\langle\lambda,\alpha\rangle K_\alpha\cdot v=\langle \lambda,\alpha\rangle\cdot v,
\end{equation}
where $\langle \lambda,\alpha\rangle=\frac{1}{d_\alpha}(\lambda,\alpha)$.  Here we have employed the identity $\binom{m}{1}_{(\pm 1)}=(\pm 1)^{(m+1)}m$ deduced from \cite[Lemma 34.1.2]{lusztig93}.
\par

At distinct $\alpha$ and $\beta$ we have
\[
[\mbf{e}_\alpha,\mbf{f}_\beta]\cdot v=M_\beta(-\alpha)M_\alpha M_\beta e_\alpha f_\beta\cdot v- M_\alpha(\beta) M_\alpha M_\beta f_\beta e_\alpha\cdot v.
\]
By anti-symmetry of the form $\kappa$, $M_{\beta}(-\alpha)=M^{-1}_{\alpha}(\beta)$ so that the above expression reduces to give
\begin{equation}\label{eq:comm2}
[\mbf{e}_\alpha,\mbf{f}_\beta]\cdot v=M_\alpha(\beta)M_\alpha M_\beta [e_\alpha,f_\beta]\cdot v=0.
\end{equation}
\par

Let us now address the positive Serre relations, and their negative counterparts.  Take for the moment $x^{[n]}=x^n/[n]_{\varepsilon_\alpha}!$ and $x^{(n)}=x^n/n!$.  We have $[n]_{\varepsilon_\alpha}=\varepsilon_\alpha^{n-1}n$ at each $n$ so that
\[
e_\alpha^{[n]}= \varepsilon_\alpha^{n+\sum_{j=1}^nj}\cdot e^n_\alpha/n!\ \ \Rightarrow\ \ e_\alpha^{[r]}e_\beta e_\alpha^{[m-r]}=\big(\varepsilon_\alpha^{m(m-1)/2}\varepsilon_\alpha^{r(m-r)}\big)\cdot e_\alpha^{(r)}e_\beta e^{(m-r)}_\alpha,
\]
and similarly
\[
f_\alpha^{(r)}f_\beta f_\alpha^{(m-r)}=\big(\varepsilon_\alpha^{m(m-1)/2}\varepsilon_\alpha^{r(m-r)}\big)\cdot f_\alpha^{(r)} f_\beta f^{(m-r)}.
\]
So the $\varepsilon$-Serre relations for the positive and negative subalgebras in $U_\varepsilon$ appear as
\[
0=\sum_{r+s=1-(\alpha,\beta)/d_\alpha}\pm \varepsilon_\alpha^{rs}\cdot e_\alpha^{(r)}e_\beta e^{(s)}_\alpha\ \ \text{and}\ \ 0=\sum_{r+s=1-(\alpha,\beta)/d_\alpha}\pm \varepsilon_\alpha^{rs}\cdot f_\alpha^{(r)}f_\beta f^{(s)}_\alpha.
\]

We check the classical Serre relations among the $\mbf{e}_\alpha$ and the $\mbf{f}_\alpha$.  We only establish these relation for the positive root vectors, as the realtions for the $\mbf{f}_\alpha$ are completely similar.
\vspace{2mm}

{\bf Case I}: Suppose $|\alpha|^2\geq |\beta|^2$ and $(\alpha,\beta)\neq 0$.  Then $(\alpha,\beta)/d_\alpha=1$ and $\varepsilon_\alpha=M_\alpha(\beta)^2$ so that we check vanishing of the equation
\[
\sum_{r+s=2}\pm \mbf{e}_\alpha^{(r)}\mbf{e}_\beta\mbf{e}_\alpha^{(s)}=M_\alpha^{2}\sum_{r+s=2}\pm M_\alpha^s(-\beta)M_\beta(-r\alpha) e_\alpha^{(r)}e_\beta e^{(s)}_\alpha
\]
\[
=M_\alpha^{2}\sum_{r+s=2} \pm M_\alpha^{r-s}(\beta) e_\alpha^{(r)}e_\beta e^{(s)}_\alpha.
\]
For $r=2$, $1$, $0$ we have $r-s=2$, $0$, $-2$ while $rs=0$, $1$, $0$.  So $M_\alpha^{r-s}(\beta)=\varepsilon_\alpha \varepsilon_\alpha^{rs}$ in the above expression, and we reduce to obtained the desired relations
\[
\sum_{r+s=2}\pm \mbf{e}_\alpha^{(r)}\mbf{e}_\beta\mbf{e}_\alpha^{(s)}=\varepsilon_\alpha M_\alpha^{2}\sum_{r+s=2} \pm \varepsilon_\alpha^{rs}\cdot e_\alpha^{(r)}e_\beta e^{(s)}_\alpha=0.
\]
\vspace{2mm}

{\bf Case II}:  Suppose $|\beta|^2/|\alpha|^2=2$ and $(\alpha,\beta)\neq 0$.  Then $(\alpha,\beta)/d_\alpha=2$, $\varepsilon(\alpha,\beta)=1$, and subsequently $M_\alpha(\beta)=M^{-1}_\alpha(\beta)$.  We now check the expression
\[
\sum_{r+s=3}\pm \mbf{e}_\alpha^{(r)}\mbf{e}_\beta\mbf{e}_\alpha^{(s)}=M_\alpha^{2}\sum_{r+s=3} \pm M_\alpha^{r+s}(\beta) e_\alpha^{(r)}e_\beta e^{(s)}_\alpha
\]
\[
=M_\alpha^3(\beta) M_\alpha^{2}\sum_{r+s=3} \pm e_\alpha^{(r)}e_\beta e^{(s)}_\alpha
\]
Now for $r=3$, $2$, $1$, $0$ we have $r\cdot s=0$, $2$, $2$, $0$ so that $\varepsilon^{rs}_\alpha=1$, and $e_\alpha$ and $e_\beta$ satisfy the classical Serre relations in this case.  So the above expression vanishes,
\[
\sum_{r+s=3}\pm \mbf{e}_\alpha^{(r)}\mbf{e}_\beta\mbf{e}_\alpha^{(s)}=M_\alpha^3(\beta) M_\alpha^{2}\sum_{r+s=3} \pm e_\alpha^{(r)}e_\beta e^{(s)}_\alpha=0,
\]
and we observe the Serre relations for $\mbf{e}_\alpha$ and $\mbf{e}_\beta$.
\vspace{2mm}

{\bf Case III:} Suppose $|\beta|^2/|\alpha|^2=3$ and $(\alpha,\beta)\neq 0$.  Then $(\alpha,\beta)/d_\alpha=3$ and $M^2\alpha(\beta)=\varepsilon(\alpha,\beta)=\varepsilon_\alpha$.  We have
\[
\sum_{r+s=4}\pm \mbf{e}_\alpha^{(r)}\mbf{e}_\beta\mbf{e}_\alpha^{(s)}=M_\alpha^{2}\sum_{r+s=2} \pm M_\alpha^{r-s}(\beta) e_\alpha^{(r)}e_\beta e^{(s)}_\alpha.
\]
For $r=4$, $3$, $2$, $1$, $0$ one checks the difference $r-s=4,\ 2,\ 0,\ -2,\ -4$ and also $rs=0,\ 3,\ 4,\ 3,\ 0$ to observe an equality $M^{r-s}_\alpha(\beta)=\varepsilon^{rs}_\alpha$ at all $r$ and $s$.  The above expression then reduces to give
\[
\sum_{r+s=4}\pm \mbf{e}_\alpha^{(r)}\mbf{e}_\beta\mbf{e}_\alpha^{(s)}=M_\alpha^{2}\sum_{r+s=2} \pm \varepsilon_\alpha^{rs}\cdot e_\alpha^{(r)}e_\beta e^{(s)}_\alpha=0.
\]
\vspace{2mm}

The above calculations show that the rescaled vectors $\mbf{e}_\alpha$ satisfy the classical Serre relations in $U_\varepsilon$, and one finds similarly that the $\mbf{f}_\alpha$ satisfy the classical Serre relations.  Indeed, the classical Serre relations on these vectors are \emph{equivalent} to the $\varepsilon$-Serre relations for the $e_\alpha$ and $f_\alpha$.
\par

After recalling the commutativity relations \eqref{eq:comm1} and \eqref{eq:comm2}, we now see that the primitive operators $\mbf{e}_\alpha,\mbf{f}_\alpha:V\to V$ define an action of the classical Lie algebra $\check{\mfk{g}}$ on any object $V$ in $\Rep(\dG_\varepsilon)_\kappa$, and we obtain a functor
\[
E:\Rep(\dG_\varepsilon)_{\kappa}\overset{\sim}\to \Rep(\dG)
\]
\[
\big(V\ \text{with operators }e_\alpha, f_\alpha\big)\ \mapsto\ \big(V\ \text{with operators }\mbf{e}_\alpha, \mbf{f}_\alpha\big).
\]
One immediately constructs the inverse by un-normalizing the vectors $\mbf{e}_\alpha$ and $\mbf{f}_\alpha$, so that this functor is seen to be an equivalence.
\par

Take the inverse $E^{-1}$ and compose with $\opn{Tw}_\kappa$ to obtain the claimed symmetric monoidal equivalence $F_\kappa: \Rep(\dG)\overset{\sim}\to \Rep(\dG_{\varepsilon})$.
\end{proof}

\section{Group actions and blocks}
\label{sect:B}

\begin{proposition}\label{prop:blocks_action}
Suppose $\msc{M}$ is a compactly generated, locally finite, presentable, linear abelian category equipped with a rational action of an affine algebraic group $H$.  Suppose additionally that $H$ is connected, and that the $H$-action on $\msc{M}$ preserves compact objects.  Then each block $\msc{M}_\xi\subseteq \msc{M}$ is stable under the action of $H$.
\end{proposition}

By preservation of compact objects we mean that the action map $\msc{M}\to \QCoh(H)\ot_k\msc{M}$ sends compact objects to compact objects.  In the case where $H$ acts on a tensor category by tensor automorphisms, preservation of rigid objects along monoidal functors implies preservation of compact objects.

\begin{proof}
Take $R=\O(H)$, considered as an algebra object in $\opn{Vect}$, and take $\msc{M}_R=R\text{-mod}_{\msc{M}}$ for simplicity.  Then the block decomposition $\msc{M}=\oplus_{\xi\in \Xi} \msc{M}_i$ induces a block decomposition on the base change
\[
\QCoh(H)\ot_k\msc{M}=\msc{M}_R=\oplus_\xi(\msc{M}_\xi)_R.
\]

Suppose we have a compact object $M$ in $\msc{M}_R$ whose fibers $x^\ast M$ at all points in $H$ are simple, and in particular nonzero.  Via compactness $M$ decomposes as a finite sum over the blocks $M=M_1\oplus \cdots \oplus M_n$ and via simplicity of the fibers we have
\[
\opn{supp}(M_i)\cap \opn{supp}(M_j)=\emptyset\ \ \text{whenever}\ \ i\neq j.
\]
Since $M$ is supported everywhere we also have $H=\cup_{i=1}^n\opn{supp}(M_i)$.  By compactness, all of the supports $\opn{supp}(M_i)$ are closed--see for example \cite[Lemma A.5]{negron21}--and hence connectedness of $H$ forces $M$ to be indecomposable.  Rather, $M=M_i$ for some $i$, $M$ lives in a single block $(\msc{M}_\xi)_R$, and all of the fibers $x^\ast M$ are simples in the same block.
\par

We consider the case $M=\rho(V)$ with $V$ simple in $\msc{M}$, where $\rho:\msc{M}\to \QCoh(H)\ot_k\msc{M}$ is the action map.  Let $\msc{M}_\xi$ be the block in $\msc{M}$ which contains $V$.  In this case the fiber at the identity $1^\ast M$ is isomorphic to $V$, and all of the fibers are simple since each composite $x^\ast\rho$ is a linear automorphism of $\msc{M}$.  It follows that $\rho(\msc{M}_\xi)\subseteq (\msc{M}_\xi)_R$ by the above arguments.  So we see that the $H$-action decomposes diagonally along the blocks
\[
\rho=\opn{diag}(\rho_\xi:\xi\in \Xi):\oplus_\xi\msc{M}_\xi\to \oplus_\xi\QCoh(H)\ot_k\msc{M}_\xi.
\]
\end{proof}

We apply Proposition \ref{prop:blocks_action} to the tensor setting.

\begin{corollary}\label{cor:blocks_action}
Suppose $\msc{A}$ is a tensor category and that $H$ acts on $\msc{A}$ by tensor automorphisms.  Suppose additionally that $H$ is connected.  Then each block in $\msc{A}$ is stable under the $H$-action.
\end{corollary}

\begin{remark}
In the case where $\msc{A}$ is a finite tensor category, and $H$ is again connected, one can show that each of the objects $x\cdot P$ in the $H$-orbit of an indecomposable projective $P$ are isomorphic.  In particular, they are all isomorphic to $P$.  Equivalently, each object $x\cdot V$ in the orbit of a given simple $V$ is isomorphic to $V$ itself.  This refines Corollary \ref{cor:blocks_action} in the finite setting.
\end{remark}

\bibliographystyle{abbrv}

\end{document}